\documentclass[11 pt]{amsart}

\usepackage{aliascnt}
\usepackage{amsfonts}
\usepackage{amsmath}
\usepackage{amssymb}
\usepackage{amsthm}
\usepackage{amsxtra}
\usepackage{array}
\usepackage[english]{babel}
\usepackage{bbm}
\usepackage[en-GB]{datetime2}
\usepackage{enumitem}
\usepackage{fancyhdr}
\usepackage[T1]{fontenc}
\usepackage{graphicx}
\usepackage{hyperref}
\usepackage[utf8]{inputenc}
\usepackage{mathtools}
\usepackage{mathrsfs}
\usepackage{mleftright}
\usepackage{moreenum}
\usepackage{polski}
\usepackage[scr=boondox]{mathalfa}
\usepackage{slashed}
\usepackage{tabularx}
\usepackage{tasks}
\usepackage{textcomp}
\usepackage{tikz-cd}
\usepackage{ytableau}
\usepackage{xparse}

\addto\extrasenglish{%
}

\calclayout

\graphicspath{ {./images/} }

\hypersetup{
    colorlinks=true,
    linkcolor=blue,
    filecolor=blue,      
    urlcolor=blue,
    citecolor=purple
}

\theoremstyle{plain}

\newtheorem{thm}{Theorem}[section]

\newaliascnt{prop}{thm}
    \newtheorem{prop}[prop]{Proposition}
    \aliascntresetthe{prop}

\newaliascnt{lem}{thm}
    \newtheorem{lem}[lem]{Lemma}
    \aliascntresetthe{lem}

\newaliascnt{cor}{thm}
    \newtheorem{cor}[cor]{Corollary}
    \aliascntresetthe{cor}

\newaliascnt{conj}{thm}
    \newtheorem{conj}[conj]{Conjecture}
    \aliascntresetthe{conj}

\newaliascnt{que}{thm}
    
    \aliascntresetthe{que}

\newaliascnt{as}{thm}
    
    \aliascntresetthe{as}

\theoremstyle{definition}

\newaliascnt{exm}{thm}
    
    \aliascntresetthe{exm}

\newaliascnt{dfn}{thm}
    \newtheorem{dfn}[dfn]{Definition}
    \aliascntresetthe{dfn}

\newaliascnt{rem}{thm}
    \newtheorem{rem}[rem]{Remark}
    \aliascntresetthe{rem}

\numberwithin{equation}{section}

\NewDocumentCommand{\evalat}{sO{\big}mm}{%
  \IfBooleanTF{#1}
   {\mleft. #3 \mright|_{#4}}
   {#3#2|_{#4}}%
}
\newcommand{\R}{\mathbb{R}}
\newcommand{\C}{\mathbb{C}}
\newcommand{\Z}{\mathbb{Z}}

\newcommand{\M}{\mathcal{M}}
\newcommand{\g}{\mathfrak{g}}

\newcommand{\G}{\mathcal{G}}
\newcommand{\A}{\mathcal{A}}

\newcommand\smvee{\raise0.9ex\hbox{$\scriptscriptstyle\vee$}}

\newcounter{sqindex}

\DeclareMathOperator{\cok}{coker}

\DeclareMathOperator{\im}{im}

\DeclareMathOperator{\Rea}{Re}
\DeclareMathOperator{\Ima}{Im}
\DeclareMathOperator{\supp}{supp}
\DeclareMathOperator{\Hom}{Hom}

\DeclareMathOperator{\id}{id}
\DeclareMathOperator{\Sf}{Sf}
\DeclareMathOperator{\Bf}{Bf}

\DeclareMathOperator{\Spec}{Spec}

\DeclareMathOperator{\ind}{ind}

\DeclareMathOperator{\Res}{Res}
\DeclareMathOperator{\sign}{sign}

\DeclareMathOperator{\Str}{Str}
\DeclareMathOperator{\tr}{tr}
\DeclareMathOperator{\Span}{Span}

\DeclareMathOperator{\Ad}{Ad}

\DeclareMathOperator{\Tr}{Tr}

\DeclareMathOperator{\Hol}{Hol}

\DeclareMathOperator{\PD}{PD}

\DeclareMathOperator{\vol}{vol}

\DeclareMathOperator{\ev}{ev}

\DeclareMathOperator{\Red}{red}

\DeclareMathOperator{\Stab}{Stab}

\mathchardef\mhyphen="2D

\begin{document}

\title[Periodic Index Theory and Equivariant Torus Signature]{Periodic Index Theory and Equivariant Torus Signature}
\author{Langte Ma}
\address{Simons Center for Geometry and Physics, 100 Nicolls Road, Stony Brook, NY 11790}
\email{lma@scgp.stonybrook.edu}

\maketitle

\begin{abstract}
We deduce an index jump formula for first order elliptic complexes over end-periodic manifolds, which generalizes the corresponding result for the DeRham complex. In the case of the anti-self-dual DeRham complex, we define the periodic rho invariant for a class of $4$-manifolds, and identify it with the periodic spectral flow of this complex. As an application, we prove the equivalence (under a mild homological assumption) of two signatures invariants defined by means of Yang-Mills theory and geometric topology respectively for essentially embedded tori in homology $S^1 \times S^3$. We also prove a surgery formula for the singular Furuta-Ohta invariant, which corresponds to a potential exact triangle of singular instanton homology for knots.  
\end{abstract}

\section{\large \bf Introduction}

\subsection{\em Index Theory}\label{ss1.1} \hfill
 
\vspace{3mm}

Since the work of Atiyah-Patodi-Singer \cite{APS1}, the index of elliptic operators has become a fruitful resource for producing and relating invariants over manifolds with cylindrical end. There is another interesting class of non-compact manifolds whose end takes a periodic form. Compared to the overwhelmed study and application of index theory over manifolds with cylindrical end, the strength of index theory over this class of non-compact manifolds was largely undermined. The main theme of this paper is to further develop the index theory of elliptic complexes over manifolds with periodic end and provide applications after the work of Taubes \cite{T87} and Mrowka-Ruberman-Saveliev \cite{MRS16}. 

Roughly speaking, an end of a non-compact manifolds is a punctured neighborhood of the infinity in the one-point compactification. A manifold is said to have a periodic end if this end can be identified with one end of an infinite cyclic cover of a compact manifold. It is proved by Hughes-Ranicki \cite{HR96} that the end of any topological manifold of dimension greater than $5$ is periodic after imposing a mild homotopy assumption. In dimension $4$, periodic ends are related to the exoticness of smooth manifolds as uncovered by Taubes \cite{T87} from applying Yang-Mills theory to such manifolds. To illustrate the main results in this paper, we shall start with the framework built up by Taubes. Let's denote by $Z$ a smooth $n$-manifold with a periodic end modeled on the infinite cyclic cover $\pi: \tilde{V} \to V$ of a compact manifold $V$. An (end-periodic) elliptic complex over $Z$ is a chain complex:
\begin{equation}\label{e1.1}
0 \longrightarrow C^{\infty}_0(Z, \tilde{E}_0) \xrightarrow{\partial^0} C^{\infty}_0(Z, \tilde{E}_1) \xrightarrow{\partial^1} ... \xrightarrow{\partial^{n-1}} C^{\infty}_0(Z, \tilde{E}_n) \longrightarrow 0,
\end{equation}
where $\tilde{E}_j$'s are complex vector bundles over $Z$ whose restriction to the end is given by the pull-back $\pi^*E_j$ of complex bundles over $V$, and $\partial^j$'s are differential operators whose restriction to the end of $Z$ are given by the pull-back of those over $V$, and whose associated symbol sequence is exact. In order to transform (\ref{e1.1}) into a Fredholm complex, namely a chain complex of finite dimensional homology, one needs to complete the functional spaces $C^{\infty}_0(Z, \tilde{E}_j)$ with weighted Sobolev norms $L^2_{k, \delta} = e^{-\delta \rho/2}L^2_k$, where $\delta \in \R$ and $\rho: Z \to \R$ is a smooth function that grows linearly with respect to the covering transformation over the end. We shall denote the $L^2_{k, \delta}$-completed chain complex by $E_{\delta}(Z)$. In \cite{T87}, Taubes provided a necessary and sufficient condition to determine for which values of the weight $\delta$ the chain complex $E_{\delta}(Z)$ is Fredholm. Moreover, $(\ref{e1.1})$ fails to be Fredholm with respect to either all $\delta \in \R$ or a discrete subset in $\R$ without accumulation points. In the latter case, it is a natural question to ask how the index of the differential complex $(\ref{e1.1})$ depends on the choice of $\delta$. Our first theorem gives a complete answer to this question when the differential operators $\partial^j$ are of first order. 

To this end, we introduce a version of equivariant cohomology associated to the elliptic complex $(\ref{e1.1})$ assuming it's of first order. Let's choose a smooth function $f: V \to S^1$ representing a cohomology class in $H^1(V; \Z)$ that corresponds to the infinite cyclic cover $\tilde{V}$. The restriction of $\partial^j$ over the end is given by the pull-back of a differential operator over $V$, which we still denote by $\partial^j$. We write $\sigma^j$ for the evaluation of the symbol of $\partial^j$ over the $1$-form $df$, which becomes a zero-th order operator over $V$. Given $z \in \C$, we can form a $z$-twisted elliptic complex over $V$:
\begin{equation}\tag{$E_z(V)$}
0 \longrightarrow L^2_k(V, E_0) \xrightarrow{\partial^0_z} L^2_{k-1}(V, E_1) \xrightarrow{\partial^1_z} ... \xrightarrow{\partial^{n-1}_z} L^2_{k-n}(V, E_n) \longrightarrow 0,
\end{equation}
where $\partial^j_z = \partial^j - z \sigma^j$. Let's denote the $j$-th chain of $E_z(V)$ by $C^j_z$, and $\mathcal{R}:=\C[u, u^{-1}]/\C[u]$ the ring of Laurent polynomials in a formal variable $u$ with only the principal part. The equivariant cohomology $\widehat{H}^*_z(V)$ associated to $(\ref{e1.1})$ is defined as the cohomology of the cochain complex $\widehat{C}^*_z:= C^j_z \otimes \mathcal{R}$ with differential $\widehat{\partial}^*_z:= \partial^*_z - u\sigma^j$. For convention reasons, we shall write $\ind E_{\delta}(Z)$ as minus the Euler characteristic of $E_{\delta}(Z)$ and refer to it as the index. 

\begin{thm}\label{t1.1}
Suppose $\delta_1 > \delta_2$ are two weights in $\R$ so that the weighted Sobolev completion $E_{\delta_1}(Z)$ and $E_{\delta_2}(Z)$ are both Fredholm. Then their index difference is given by 
\[
\ind E_{\delta_1}(Z) - \ind E_{\delta_2}(Z) = \sum_{z \in (\delta_2/2, \delta_1/2) \times i[0, 2\pi)} \chi\left(\widehat{H}^*_z(V) \right). 
\]
\end{thm}

We refer to \autoref{t1.1} as the index jump formula. Although the sum on the right seems to be infinite, we shall see later that $\widehat{H}^*_z(V)$ is non-vanishing only at finitely many points in the considered region under our assumption. The idea of the proof is to transform this index comparison to the computation of the index of an elliptic complex $E_{\tilde{\delta}}(\tilde{V})$ over the periodic manifold $\tilde{V}$ using the excision principle \cite{AS68}. We will actually prove a much stronger result in \autoref{t2.12} which identifies the cohomology $H^*(E_{\tilde{\delta}}(\tilde{V}))$ explicitly with direct sums of $\widehat{H}^*_z(V)$ with the help of a new notion called meromorphic homotopy resovlent defined with respect to a holomorphic family of index zero elliptic complexes over closed manifolds. Indeed, $\widehat{H}^*_z(V)$ will be identified with the generalized $e^{-z}$-eigenspaces of $H^*(E_{\tilde{\delta}}(\tilde{V}))$ under the action of the covering transformation. 

When the length of the elliptic complex (\ref{e1.1}) is $2$, we get a first order elliptic operator. In this case, \autoref{t1.1} recovers the result proved by Mrowka-Ruberman-Saveliev \cite[Section 6]{MRS11}. By taking the formal $L^2$-adjoint of $\partial^j$ and restricting to the chains of even degree, we get an elliptic operator from the elliptic complex. However the method of Mrowka-Ruberman-Saveliev cannot be applied to this adjoint operator to get the result. Because we are working with weighted Sobolev spaces, and to match up the index one has to take the $L^2_{\delta}$-adjoint rather than $L^2$-adjoint, which results in operators that depend on $\delta$ a priori. Moreover, the proof for elliptic operators make use of the Fourier-Laplace transform \cite{T87}, which does not commute with taking adjoint. 

When $(\ref{e1.1})$ is the DeRham complex, \autoref{t1.1} is essentially proved in \cite{MRS14} by means of algebraic topology developed in \cite{M06} and \cite{HR96}. We shall explain in the end of \autoref{s2} why \autoref{t1.1} is a genuine generalization of the main result in \cite{MRS14}, although our method is entirely analytic. 

The index jump formula \autoref{t1.1} transforms the analytic problem of computing the index to an homological algebra problem of understanding $\widehat{H}^*_z(V)$. Note that the cochain complex $(\widehat{C}^*_z, \widehat{\partial}^*_z)$ admits a $u$-degree from the ring $\mathcal{R}$, which gives us a filtration respecting the differential. This leads to a spectral sequence $(\widehat{E}^{p,q}_{z,r}, d_r)$ that abuts to $\widehat{H}^*_z(V)$. The sufficient condition that Taubes gave in \cite[Theorem 3.1]{T87} to ensure the existence of $\delta$ making $E_{\delta}(Z)$ Fredholm is equivalent to say $\widehat{E}^{p,q}_{0, r}$ collapses at the second page and only the terms in $E^{0, q}_{0, r}$ survive. In particular, this implies that $\widehat{H}^*_0(V)$ is finite dimensional. Motivated by this perspective, we make the following conjecture:

\begin{conj}
$E_{\delta}(Z)$ is Fredholm with respect to some $\delta \in \R$ if and only if $\widehat{H}^*_z(V)$ is finite dimensional with respect to some $z \in \C$. 
\end{conj}

Taubes \cite[Lemma 4.3]{T87} proved that $E_{\delta}(Z)$ is Freholm if and only if $H^*(E_z(V)) = 0$ for all $z$ with $\Rea z = \delta/2$. In particular, we know $\widehat{H}^*_z(V) = 0$ for such $z$'s (c.f. \autoref{p2.15}). Conversely, when $\widehat{H}^*_z(V)$ is finite dimensional, we wish to show that $H^*(E_w(V)) = 0$ for all $w$ with $\Rea w = \delta_z/2$ for some real number $\delta_z$ close to $2\Rea z$. It takes the flavor of meromorphic functions where poles have to be discrete. 

To better cooperate with the Yang-Mills gauge theory in dimension $4$, we focuses on the anti-self-dual(ASD) DeRham complex 
\begin{equation}
0 \longrightarrow L^2_{k+1}(V, \Lambda^0_{\C}) \xrightarrow{-d} L^2_k(V, \Lambda^1_{\C}) \xrightarrow{d^+} L^2_{k-1}(V, \Lambda^+_{\C}) \longrightarrow 0,
\end{equation}
where $\Lambda^*_{\C} = \Lambda^*_{\C}T^*V$ means the bundle of complex alternating tensors over $V$. Given a primitive class $1_V \in H^1(V;\Z)$, we refer to the assumption in \cite[Theorem 3.1]{T87} as the non-degeneracy of the pair $(V, 1_V)$ in \autoref{d2.4}. For such a pair, we can define the periodic eta invariant $\tilde{\eta}(V)$ in the spirit of \cite{MRS16}. An end-periodic $Z$ can be constructed from the pair $(V, 1_V)$ explicitly as follows. Let $f: V \to S^1$ be a smooth representative of $[f] = 1_V$, and $Y$ a regular level set of $f$ chosen to be smooth and connected. Such a codimension-$1$ submanifold $Y$ will be referred to as a cross-section of the pair $(V, 1_V)$. Cutting $V$ open along $Y$ gives us a cobordism $W$ from $Y$ to itself. Given a smooth manifold $M$ with boundary $\partial M = Y$, we can form a non-compact manifold $Z$ by attaching infinite copies of $W$ to $M$ in a way from head to tail. Then $Z$ is end-periodic whose end is modeled on the infinite cylic cover $\tilde{V}$ associated to the pair $(V, 1_V)$. Let's write $\mathfrak{a}$ for the index form of the ASD DeRham operator $-d^*\oplus d^+$. The non-degeneracy assumption on the pair $(V, 1_V)$ implies that the index form $\mathfrak{a}(V)$ is exact on $V$, say $d \mathfrak{z}(V) = \mathfrak{a}(V)$. Then argument in \cite[Theorem C]{MRS16} can then be applied to the ASD DeRham complex resulting in the following index theorem.

\begin{thm}\label{t1.3}
Let $Z$ be an end-periodic $4$-manifold with its end modeled on $\tilde{V}$ associated to a non-degenerate pair $(V, 1_V)$. Then there exists $\delta_0 > 0$ such that for all $\delta \in (0, \delta_0)$ the index of the ASD DeRham complex over $Z$ is 
\[
\ind E_{\delta}(Z) = \int_M \mathfrak{a}(M) - \int_Y \mathfrak{z}(V) + \int_V df \wedge \mathfrak{z}(V) - \frac{\tilde{h}(V) + \tilde{\eta}(V)}{2},
\]
where $\tilde{h}(V) = \sum_{z \in \{0\} \times i[0, 2\pi)} - \chi\left(\widehat{H}^*_z(V) \right)$. 
\end{thm} 

We note that over the cylinder $(0, 1) \times Y$, the ASD DeRham operator takes the form $-d^*\oplus d^+  =d/dt + L'$ with $L'=(-1)^p(\star d - d\star)$ the even signature operator defined on $\Omega^{\ev}_{\C}(Y)$ \cite{APS1}. So it's natural to expect a periodic version of the rho invariant introduced in \cite{APS2}. Indeed, we restrict ourselves to $U(1)$-representations $\varphi: \pi_1(V) \to U(1)$ and define the periodic rho invariant $\tilde{\rho}_{\varphi}(V) = \tilde{\eta}_{\varphi}(V) - \tilde{\eta}(V)$ as the reduced eta invariant under an admissibility assumption \autoref{d3.6}. It is a smooth invariant of $V$ that depends additionally on the primitive class $1_V \in H^1(V; \Z)$. With the help of the index theorem \autoref{t1.3}, the periodic rho invariant of the ASD DeRham complex and the rho invariant of the even signature operator can be compared as follows.

\begin{prop}\label{p1.4}
Let $(V, 1_V)$ be a non-degenerate pair, $\varphi: \pi_1(V) \to U(1)$ an admissible representation, and $Y$ a corresponding cross-section. Then 
\[
\tilde{\rho}_{\varphi}(V) - \rho_{\varphi}(Y) = \tilde{\eta}(V) - \eta(Y).
\]
\end{prop}

With a little more effort, \autoref{t1.3} could be deduced for arbitrary first order elliptic complexes. So a corresponding result of \autoref{p1.4} can be expected for unitary representations of higher rank. However, for the purpose of application in this paper, we are satisfied with the current form. 

In \cite{APS3}, the relation between the spectral flow and the eta invariant was explored. This is also true in the periodic case. The notion of periodic spectral flow was introduced for a path of Dirac operators $\slashed{\partial}_t$ in \cite{MRS11}, which essentially counts the number of points $(t, z) \in [0, 1] \times i [0, 2\pi )$ over which the twisted Dirac operator $\slashed{\partial}_{t, z}$ fails to be invertible. Given a path of ASD DeRham operators $Q_t:=-d^*_{A_t} \oplus d^+_{A_t}$ with $A_t$ a path of connections on the trivial bundle $\underline{\C}$ of $V$, we can define the periodic spectral flow correspondingly. The only exception happens when $A_0$ is the trivial connection so $\ker Q_0 = H^1(V;\C) \neq 0$. In this case, we introduce a weight $\delta \in \R$ to get rid of the situation, which would be seen to coincide with the weight in the completion of the functional spaces on the end-periodic manifold $Z$. Then we denote the corresponding periodic spectral flow by $\widetilde{\Sf}(Q_{t, \delta})$. When the representation $\varphi$ is connected to the trivial representation via a continuous path $(\varphi_t)$, we deduce that the periodic rho invariant $\tilde{\rho}_{\varphi}(V)$ is given by the periodic spectral flow of this path. 

\begin{prop}\label{p1.5}
Suppose $\varphi_t: \pi_1(V) \to U(1)$ is a smooth path of representations connecting the trivial representation $\varphi_0$ to an admissible representation $\varphi_1$. Then there exists $\delta_1 > 0$ such that for all $\delta \in (0, \delta_1)$ one has 
\[
-\tilde{\rho}_{\varphi_1}(V) = \widetilde{\Sf}(Q_{t, \delta}) + \widetilde{\Sf}(Q_{t, -\delta}),
\]
where $Q_t = -d^*_{\varphi_t} \oplus d^+_{\varphi_t}$ is the ASD DeRham operator twisted by $\varphi_t$. 
\end{prop}

\subsection{\em Torus Signature}\label{ss1.2} \hfill
 
\vspace{3mm}

The second half of this paper deals with an application of the periodic index theory to $4$-dimensional topology. Inspired by different perspectives towards the Levine-Tristram signature invariant for knots \cite{L69, T69}, lately two authors considered four dimensional analogues with respect to essentially embedded tori in homology $S^1 \times S^3$. The first approach is due to Echeverria \cite{E19} by counting degree zero singular instantons. The second is due to Ruberman \cite{R20} using a more topological construction. Conjecturally the signature invariants defined via both methods should agree. In the following, we give an affirmative answer to this conjecture under certain mild assumptions. 

We shall write $X$ for a smooth closed oriented $4$-manifold satisfying 
\begin{equation}
H_*(X;\Z) \simeq H_*(S^1 \times S^3; \Z),
\end{equation}
which we refer to as an integral homology $S^1 \times S^3$. We also fix a primitive class $1_X \in H^1(X; \Z)$ serving the role of homology orientations in Yang-Mills theory. Let $\mathcal{T}$ be an embedded torus in $X$ inducing a surjective map $H_1(\mathcal{T}; \Z) \to H_1(X; \Z)$. We call $\mathcal{T}$ an essentially embedded torus. Given a holonomy parameter $\alpha \in (0, 1/2)$, adopting the framework of $\alpha$-singular instantons in \cite{KM93} Echeverria \cite{E19} defined the $\alpha$-singular Furuta-Ohta invariant $\lambda_{FO}(X, \mathcal{T}, \alpha)$ to be the signed count of the degree zero irreducible $\alpha$-singular instantons on $(X, \mathcal{T})$ assuming $\alpha$ satisfies a non-degeneracy condition \autoref{d5.2}. Heuristically, an $\alpha$-singular instanton is an anti-self-dual connection on $X \backslash \mathcal{T}$ whose holonomy around the meridians of $\mathcal{T}$ is asymptotic to 
\[
e^{-2\pi i \alpha} \in U(1) \subset SU(2).
\]
The non-degeneracy condition for $\alpha$ is equivalent to the fact that the reducible locus of the $\alpha$-singular moduli space is isolated from the irreducible one. With another non-degeneracy condition (\ref{e5.16}), one can define the usual Furuta-Ohta invariant $\lambda_{FO}(X)$ as a quarter of the count of degree zero instantons. Echeverria \cite{E19} considered the difference 
\begin{equation}\label{e1.4}
\lambda_{FO}(X, \mathcal{T}, \alpha) - 8\lambda_{FO}(X)
\end{equation} 
as an invariant for the isotopy class of essential embeddings of a $2$-torus, and showed this difference is the Levine-Tristram signature invariant $\sigma_{2\alpha}(Y, \mathcal{K})$ in the product case when $(X, \mathcal{T}) = S^1 \times (Y, \mathcal{K})$. Here $\mathcal{K}$ is a knot in an integral homology sphere $Y$. 

On the other hand, given a pair $(X, \mathcal{T})$ as above, one can perform the $0$-surgery along the torus $\mathcal{T}$ resulting in a cohomology $T^2 \times S^2$, which we denote by $X_0(\mathcal{T})$. Moreover the primitive class $1_X$ induces a primitive class $1_{X_0} \in H^1(X_0(\mathcal{T});\Z)$. Let $Y_0$ be a cross-section of $(X_0(\mathcal{T}), 1_{X_0})$. The holonomy parameter $\alpha$ gives rise to a unitary representation $\varphi_{2\alpha}: \pi_1(X_0(\mathcal{T})) \to U(1)$ mapping the meridian of $\mathcal{T}$ to $e^{-4\pi i \alpha}$ and the curve dual to $1_{X_0}$ to $1$. Ruberman \cite{R20} showed that the rho invariant $\rho_{\varphi_{2\alpha}}(Y_0)$ is independent of the choice of the cross-section $Y_0$, and defined it as the signature invariant $\sigma_{2\alpha}(X, \mathcal{T})$ of $(X, \mathcal{T})$. When $(X, \mathcal{T})$ is the product $S^1 \times (Y, \mathcal{K})$, the equality $\sigma_{2\alpha}(X, \mathcal{T}) = \sigma_{2\alpha}(Y, \mathcal{K})$ follows from properties of the Levine-Tristram signature. 

\begin{thm}\label{t1.6}
Let $\mathcal{T} \subset X$ be an essentially embedded torus in a  homology $S^1 \times S^3$. Assume the holonomy parameter $\alpha \in (0, 1/2)$ is non-degenerate. Then 
\[
\lambda_{FO}(X, \mathcal{T}, \alpha) - 8 \lambda_{FO}(X) = \tilde{\rho}_{\varphi_{2\alpha}}(X_0(\mathcal{T})) - \lim_{\epsilon \to 0^+} \tilde{\rho}_{\varphi_{2\epsilon}}(X_0(\mathcal{T})). 
\]
\end{thm}

Technically the definition of $\lambda_{FO}(X, \mathcal{T}, \alpha)$ requires using an orbifold metric with cone angle $2\pi/\nu$ along $\mathcal{T}$ for some $\nu$ sufficiently large. Echeverria \cite{E19} conjectured that the invariant is independent of the choice of the cone angle. The proof of \autoref{t1.6} also confirms this claim. 

To identify the difference of rho invariant in \autoref{t1.6} with the rho invariant $\rho_{\varphi_{2\alpha}}(Y_0)$, we impose another homological assumption. Let $L'_{\varphi_{2\alpha}}$ be the even signature operator twisted by the flat connection corresponding to $\varphi_{2\alpha}$. Let's consider a symmetric pairing on $H^0(Y_0;\C) \oplus H^2(Y_0;\C)$ defined by 
\begin{equation}\label{e1.5}
C'_{\alpha}(a, b) = \langle a, (L'_{\varphi_{2\alpha}} - L') b\rangle_{L^2}.
\end{equation}

\begin{thm}\label{t1.7}
Under the situation of \autoref{t1.6}, we further assume the symmetric pairing $C'_{\alpha}$ is non-degenerate and has zero signature. Then
\[
\tilde{\rho}_{\varphi_{2\alpha}}(X_0(\mathcal{T})) - \lim_{\epsilon \to 0^+} \tilde{\rho}_{\varphi_{2\epsilon}}(X_0(\mathcal{T})) = \sigma_{2\alpha}(X, \mathcal{T}). 
\]
\end{thm}

Since the pairing $C'_{\alpha}$ differs by a non-zero scalar with respect to different choices of $\alpha$, the assumption on $C'_{\alpha}$ is actually intrinsic to the pair $(X, \mathcal{T})$. We note this assumption holds for a large class of pairs $(X, \mathcal{T})$, for instance when a cross-section $Y_0$ of $X_0(\mathcal{T})$ satisfies $b_1(Y_0) = 1$. We shall explain the reason for considering $C'_{\alpha}$ and provide more examples during the proof of \autoref{t1.7}. 

When $X$ is an integral homology $S^1 \times S^3$, and $Y$ a integral homology sphere representing a cross-section of $X$, Furuta-Ohta \cite{FO93} conjectured that $\lambda_{FO}(X) = \mu(Y) \mod 2$, where $Y$ is a cross-section of $X$ and $\mu(Y)$ is the Rohlin invariant of $Y$. In the case when there exists an essentially embedded torus $\mathcal{T}$, \cite[Lemma 2.3]{M1} implies that one can choose a cross-section $Y$ so that $\mathcal{T} \cap Y = \mathcal{K}$ is a knot. Then $Y_0=Y_0(\mathcal{K})$ is a cross-section of $X_0(\mathcal{T})$. Suppose one can find an oriented $4$-manifold $M$ with boundary $\partial M = Y_0$ so that $\varphi_{2 \alpha}$ extends over $\pi_1(M)$. Then $\sigma_{2\alpha}(X, \mathcal{T}) = \sigma(M) - \sigma_{\varphi_{2\alpha}}(M)$. Note that $\sigma(M)/8 = \mu(Y) \mod 2$. Thus the equivalence $\lambda_{FO}(X, \mathcal{T}, \alpha) - 8\lambda_{FO}(X) = \sigma_{2\alpha}(X, \mathcal{T})$ implies that $\lambda_{FO}(X, \mathcal{T}, \alpha) - \sigma_{\varphi_{2\alpha}}(M)$ is independent of $\alpha \in (0, 1/2)$. So Furuta-Ohta's conjecture motivates the following conjecture:

\begin{conj}\label{c1.8}
Given the situation above, we have 
\[
\lambda_{FO}(X, \mathcal{T}, \alpha) + \sigma_{\varphi_{2\alpha}}(M) = 0 \mod 16
\]
with respect to some non-degenerate holonomy parameter $\alpha \in (0, 1/2)$. 
\end{conj}

The equivalence of the signature invariants can be justified in the case of mapping tori. Let's consider a knot $\mathcal{K} \subset Y$ in a integral homology sphere. Fix a positive integer $n$, we denote by $\Sigma_n(Y, \mathcal{K})$ the $n$-fold cyclic cover of $Y$ branched along $\mathcal{K}$, and $\tau_n: \Sigma_n(Y, \mathcal{K}) \to \Sigma_n(Y, \mathcal{K})$ the covering translation. \cite[Proposition 6.1]{LRS20} tells us that the mapping torus $X_n(Y, \mathcal{K})$ of $\Sigma_n(Y, \mathcal{K})$ under $\tau_n$ is an integral homology $S^1 \times S^3$. When $\Sigma_n(Y, \mathcal{K})$ is a rational homology sphere, $X_n(Y, \mathcal{K})$ satisfies the non-degenerate condition (\ref{e5.16}) automatically. Denote by $\mathcal{K}_n \subset \Sigma_n(Y, \mathcal{K})$ the preimage of $\mathcal{K}$ in the branched cover, and $\mathcal{T}_n \subset X_n(Y, \mathcal{K})$ the mapping torus of $\mathcal{K}_n$ under $\tau_n$. 

\begin{prop}\label{p1.9}
With the notations as above, let $\alpha \in (0, 1/2)$ be a holonomy parameter so that $e^{-4\pi i (\alpha+j)/n}$ is not a root of the Alexander polynomial $\Delta_{(Y, \mathcal{K})}$ for each $j=0 , ..., n-1$. Then 
\[
\sigma_{2\alpha}(X_n(Y, \mathcal{K})) = -\sum_{j=1}^{n-1} \sigma_{j/n}(Y, \mathcal{K}) + \sum_{j=0}^{n-1} \sigma_{2(\alpha+j)/n}(Y, \mathcal{K}),
\]
where $\sigma_{k/l}(Y, \mathcal{K})$ is the Levine-Tristram signature of the pair $(Y, \mathcal{K})$. 
\end{prop}

We note that in \cite{R20} Ruberman computed the signature invariant $\sigma_{2\alpha}(X_n(Y, \mathcal{K}))$ when $2\alpha = q/p^r$ with $p$ prime. It's straightforward to see that our computation agrees with his result \cite[Theorem 4.1]{R20} after unraveling the notations. However our proof goes through a different path by computing the singular Furuta-Ohta invariant. It was computed in \cite{LRS20} that 
\begin{equation}
\lambda_{FO}(X_n(Y, \mathcal{K})) = n\lambda(Y) + \frac{1}{8} \sum_{j=1}^{n-1} \sigma_{j/n}(Y, \mathcal{K}),
\end{equation}
where $\lambda(Y)$ is the Casson invariant of $Y$. As pointed out in \cite[Section 8]{E19}, the assumption of $\alpha$ in \autoref{p1.4} implies the non-degeneracy condition in \autoref{d5.2}. Thus all we need is to derive  
\begin{equation}\label{e1.7}
\lambda_{FO}(X, \mathcal{T}, \alpha) = 8n\lambda(Y) + \sum_{j=0}^{n-1} \sigma_{2(\alpha+j)/n}(Y, \mathcal{K}). 
\end{equation}
As communicated between the author and Echeverria, a correction of \cite[Theorem 46]{E19} will appear in a forthcoming version to establish the equality (\ref{e1.7}). For the sake of completeness, we supply a sketch of proof in \autoref{ss6.3} following the original proof of Echeverria's paper. 

We also deduce a surgery formula for the singular Furuta-Ohta invariant. Given a pair $(X, \mathcal{T})$ as before, one can perform $1/q$ surgery along the torus for an integer $q$ by gluing the tubular neighborhood $\nu(\mathcal{T})$ with another boundary diffeomorphism to its complement. We denote the resulting manifold by $X_{1/q}(\mathcal{T})$, and the core of $\nu(\mathcal{T}) = D^2 \times T^2$ in $X_{1/q}(\mathcal{T})$ by $\mathcal{T}_q$. We note that $X_{1/q}(\mathcal{T})$ is still a homology $S^1 \times S^3$ with $\mathcal{T}_q$ essentially embedded. In the case of $0$-surgery, we still denote by $\mathcal{T}_0$ the core of the gluing copy of $D^2 \times T^2$. Although $X_0(\mathcal{T})$ is a homology $T^2 \times S^2$, we can still count $\alpha$-singular instantons of the pair $(X_0(\mathcal{T}), \mathcal{T}_0)$, which gives us an invariant denoted by $\lambda_{FO}(X_0(\mathcal{T}), \mathcal{T}_0, \alpha)$. 

\begin{thm}\label{t1.10}
Let $\alpha \in (0, 1/2)$ be a non-degenerate holonomy parameter for the three pairs $(X, \mathcal{T})$, $(X_{1/q}(\mathcal{T}), \mathcal{T}_q)$, and $(X_0(\mathcal{T}), \mathcal{T})$. Then we have 
\[
\lambda_{FO}(X_{1/q}(\mathcal{T}), \mathcal{T}_q, \alpha) = \lambda_{FO}(X, \mathcal{T}, \alpha) + q \lambda_{FO}(X_0(\mathcal{T}), \mathcal{T}_0, \alpha). 
\]
\end{thm}

When $\alpha = 1/4$, Echeverria \cite{E19} expressed the singular Furuta-Ohta invariant $\lambda_{FO}(X, \mathcal{T}, \alpha)$ as the Lefschetz number of the induced cobordism map $I(W, \mathcal{S})$ on the singular instanton homology $I(Y, \mathcal{K})$ \cite{DS19}, where $\mathcal{S}$ is an embedded surface obtained by cutting $\mathcal{T}$ open along a cross-section $Y$ and $\mathcal{K}$ is one of its boundary components. So \autoref{t1.10} can be regarded as an evidence towards the existence of a surgery exact triangle relating the terms $I(Y, \mathcal{K})$, $I(Y_0(\mathcal{K}), \mathcal{K}_0)$, and $I(Y_1(\mathcal{K}), \mathcal{K}_1)$, where $\mathcal{K}_i$ means the core of the gluing copy $D^2 \times S^1$ in the surgery process. To relating the singular sets $\mathcal{K}_i \subset Y_i(\mathcal{K})$ using the $2$-handle cobordism, one can either use the orbifold set-up developed by Kronheimer-Mrowka \cite{KM19} or blow-up the singular points and use the usual set-up \cite{KM11a}, which we expect to be equivalent. For the $I^{\sharp}$-version defined by \cite{KM11}, the corresponding surgery exact triangle is expected to be identified with the by-pass exact triangle in sutured Floer homology \cite{BS18}. The author plans to explain this circle of ideas concretely in a future paper. 

Combining the surgery formula for the Furuta-Ohta invariant \cite[Theorem 1.2]{M2}, we get a surgery formula formula for the torus signature.

\begin{cor}\label{c1.11}
Suppose $\alpha \in (0, 1/2)$ satisfies the assumptions in \autoref{t1.6}. Then we have 
\[
\sigma_{2\alpha}(X_{1/q}(\mathcal{T}), \mathcal{T}_q) - \sigma_{2\alpha}(X, \mathcal{T}) = q\lambda_{FO}(X_0(\mathcal{T}), \mathcal{T}_0, \alpha) - 4qD^0_{w_{\mathcal{T}}}(X_0(\mathcal{T})),
\]
where $D^0_{w_{\mathcal{T}}}(X_0(\mathcal{T}))$ is the count of irreducible ASD $SO(3)$-instantons on the $SO(3)$-bundle over $X_0(\mathcal{T})$ with $p_1 = 0$ and $w_2= \PD [T_0]$. 
\end{cor}

\subsection*{Outline} 
We outline the content of this paper. In \autoref{s2}, we introduce the notion of meromorphic homotopy resolvent and prove \autoref{t1.1}. In \autoref{s3}, we develop the index theory for the ASD DeRham complex and prove \autoref{t1.3} and \autoref{p1.4}. In \autoref{s4}, we consider the periodic spectral flow for the ASD DeRham operator and prove \autoref{p1.5}. In \autoref{s5}, we review the definition of the singular Furuta-Ohta invariant, reformulate it as the count of finte energy instantons over manifolds with cylindrical end, and prove the surgery formula \autoref{t1.10}. In \autoref{s6}, we prove the equivalence of the signature invariants, i.e. \autoref{t1.6} and \autoref{t1.7}. 

\subsection*{Acknowledgment}
The author wishes to thank Mariano Echeverria, Claude LeBrun, Daniel Ruberman, Nikolai Saveliev for helpful discussions, and Simon Donaldson for his encouragement. He also wants to thank Donghao Wang for communications on the proof of \autoref{p2.10}.

\section{\large \bf An Index Jump Formula}\label{s2}

\subsection{\em Periodic Elliptic Complexes}\label{ss2.1} \hfill

\vspace{3mm}

Let $V$ be a closed smooth oriented $n$-manifold endowed with a primitive class $1_V \in H^1(V; \Z)$. We choose an embedded $(n-1)$-manifold $Y \subset V$ and a smooth function $f: V \to S^1$ satisfying the following:
\begin{enumerate}
\item $Y$ is a non-separating connected $(n-1)$-manifold Poincaré dual to $1_V$;
\item $f$ admits $Y$ as a regular level set with $df$ supported in a collar neighborhood of $Y$. 
\end{enumerate}
Cutting $V$ open along $Y$ results in a cobordism $W$ from $Y$ to itself. A cyclic cover of $V$ with respect to $1_V$ can be realized by gluing infinitely copies of $W$ together along one of their common boundaries, which we shall write as $\tilde{V}:= \cup_{n \in \Z} W_n$ with each $W_n$ standing for a copy of $W$. Moreover we denote $\tilde{V}_+ = \cup_{n \geq 0}W_n$ and $\tilde{V}_- = \cup_{n \leq -1}W_n$. Such a non-compact manifold $\tilde{V}$ will be referred to as a periodic $4$-manifold. Let's denote by $\pi: \tilde{V} \to V$ the projection map and $\tau: \tilde{V} \to \tilde{V}$ the generating covering transformation. For simplicity, we shall write $\tau(x) = x+1$ for $x \in \tilde{V}$ to indicate that $\tau$ is the translation towards the "positive" direction. We then fix a lift $\tilde{f}: \tilde{V} \to \R$ of $f$ so that $\tilde{f}(x+1) = \tilde{f} + 1$. 

Let $E_j$ be a sequence of finite rank complex vector bundles over $V$, $j=0, ..., n$, and $\partial^j: C^{\infty}(V, E_j) \to C^{\infty}(V, E_{j+1})$ a sequence of first order differential operators. The sequence 
\begin{equation}
0 \longrightarrow C^{\infty}(V, E_0) \xrightarrow{\partial^0} C^{\infty}(V, E_1) \xrightarrow{\partial^1} ... \xrightarrow{\partial^{n-1}} C^{\infty}(V, E_n) \longrightarrow 0
\end{equation}
is called an elliptic complex if $\partial^j \circ \partial^{j-1} = 0$ and the corresponding symbol complex is exact. Denote by $\tilde{E}_j:= \pi^* E_j$ the pull-back bundle over $\tilde{V}$. The corresponding sequence is referred to as a periodic elliptic complex:
\begin{equation}\label{e2.2}
0 \longrightarrow C^{\infty}(\tilde{V}, \tilde{E}_0) \xrightarrow{\partial^0} C^{\infty}(\tilde{V}, \tilde{E}_1) \xrightarrow{\partial^1} ... \xrightarrow{\partial^{n-1}} C^{\infty}(\tilde{V}, \tilde{E}_n) \longrightarrow 0. 
\end{equation}
Unlike the case of compact manifolds, the $L^2_k$-completion of (\ref{e2.2}) is not necessarily Fredholm, i.e. the corresponding homology is finitely dimensional. In order to overcome this difficulty, one needs to consider the completion with respect to weighted Sobolev spaces. 

Let $\delta: \tilde{V} \to \R$ be an arbitrary smooth function. Let's fix $k \geq n$ a positive integer. The $\delta$-weighted Sobolev space $L^2_{k, \delta}(\tilde{V}, \tilde{E})$ of sections on a periodic vector bundle $\tilde{E} \to \tilde{V}$ is defined to be the completion of $C^{\infty}_0(\tilde{V}, \tilde{E})$ with respect to the norm:
\begin{equation}
\|s\|_{L^2_{k, \delta}} := \left(\int_{\tilde{V}} e^{\delta \tilde{f}} \sum_{j=0}^k |\nabla^{(j)} s|^2 d\vol \right)^{1/2},
\end{equation}
where $\nabla^{(j)}$ is the $j$-th covariant derivative given by some pre-assigned periodic connection on $\tilde{E}$. The $L^2_{\delta}$-inner product is defined to be 
\begin{equation}
\langle s_1, s_2 \rangle_{L^2_{\delta}}:= \int_{\tilde{V}} (e^{\delta\tilde{f}/2}s_1, e^{\delta\tilde{f}/2}s_2) d \vol = \int_{\tilde{V}} e^{\delta\tilde{f}} s_1 \wedge \star \bar{s}_2. 
\end{equation}
Note that $\|s\|_{L^2_{k, \delta}}$ and $\|e^{-\delta\tilde{f}/2}s\|_{L^2_k}$ are commensurate. We will simply identify $L^2_k = e^{\delta\tilde{f}/2}L^2_{k, \delta}$, which suffices for our purpose. In most cases, the weight function $\delta$ is constant and will be thought of as a real number. Then we get a periodic elliptic complex of bounded differential maps under the $L^2_{k, \delta}$-completion:
\begin{equation}\tag{$E_{\delta}(\tilde{V})$}
0 \longrightarrow L^2_{k, \delta}(\tilde{V}, \tilde{E}_0) \xrightarrow{\partial^0} L^2_{k-1, \delta}(\tilde{V}, \tilde{E}_1) \xrightarrow{\partial^1} ... \xrightarrow{\partial^{n-1}} L^2_{k-n, \delta}(\tilde{V}, \tilde{E}_n) \longrightarrow 0. 
\end{equation}

In order to understand the Fredholmness of the complex $E_{\delta}(\tilde{V})$, we recall the Fourier-Laplace transform introduced in the work of Taubes \cite{T87}. Let $s \in C^{\infty}_0(\tilde{V}, \tilde{E})$ be a section of complex periodic vector bundle $\tilde{E}$ and $z \in \C$ a complex number, the Fourier-Laplace transform of $s$ with respect to $z$ is 
\begin{equation}\label{e2.5}
\mathscr{F}_z(s)(x):=\hat{s}_z(x):= e^{z\tilde{f}(x)} \sum_{n= -\infty}^{\infty} e^{zn} s(x+n). 
\end{equation} 
The equivariance of $\tilde{f}$ tells us that $\hat{s}_z(x) = \hat{s}_z(x+1)$. Thus $\hat{s}_z$ defines a section of the bundle $E$ on the closed $n$-manifold $V$. Following the notation in \cite{MRS11}, we write $I(w) := \{ w+i\theta: 0 \leq \theta \leq 2\pi\}$ for the line segment in $\C$. Then the inversion of the Fourier-Laplace transform takes the form
\begin{equation}
s(x+n)= \frac{1}{2\pi i} \int_{I(w)} e^{-z(\tilde{f}(x)+n)} \hat{s}_z(\pi(x)) dz, \qquad \forall w \in \C
\end{equation}
Here we record two basic properties of the Fourier-Laplace transform for later use.

\begin{lem}[{\cite[Proposition 4.2 $\&$ Lemma 4.3 ]{MRS11}}]\label{l2.1}
Let's fix a weight $\delta > 0$ and a non-negative integer $k$. Then the following hold. 
\begin{enumerate}[label=(\greek*)]
\item The map $s \mapsto \mathscr{F}(s)|_{I(w)}$ extends uniquely to a linear isomorphism from $L^2_{k, \delta}(\tilde{V}, \tilde{E})$ to $L^2(I(w), L^2_k(V, E))$ for all $w \in \C$ with $\Rea w = \delta/2$.
\item Suppose $s \in L^2_{k, \delta}(\tilde{V}, \tilde{E})$ satisfies $\supp u \subset \tilde{V}_+$. Then $\hat{s}_z$is holomorphic in the half plane $\Rea z < \delta/2$. 
\end{enumerate}
\end{lem}

Now we apply the Fourier-Laplace transform to the periodic elliptic complex (\ref{e2.2}) of compactly supported sections. After $L^2_k$-completion, we get a family of Fredholm differential complexes parametrized by $z \in \C$:
\begin{equation}\tag{$E_z(V)$}
0 \longrightarrow L^2_k(V, E_0) \xrightarrow{\partial^0_z} L^2_{k-1}(V, E_1) \xrightarrow{\partial^1_z} ... \xrightarrow{\partial^{n-1}_z} L^2_{k-n}(V, E_n) \longrightarrow 0,
\end{equation}
where $\partial^j_z = \partial^j - z \sigma^j$ and $\sigma^j = \sigma_{\partial^j}(df)$ is the evaluation of the symbol of $\partial^j$ against the $1$-form $df$. For simplicity, we shall write $L^2_{k-j}(V, E_j)$ as $C^j_z(V)$ for emphasizing the role as a chain complex depending on the variable $z$ rather than a single functional space. 

\begin{dfn}
We say the complex $E_z(V)$ is acyclic if $H^j(E_z(V)) = 0$ for all $j=0, ..., n$. $z$ is called a spectral point of the family $\{E_z(V)\}_{z \in \C}$ if $E_z(V)$ is not acyclic. The set of spectral points of $\{E_z(V)\}_{z \in \C}$ is denoted by $\Sigma(E)$. 
\end{dfn}

The importance of studying $E_z(V)$ is manifested by the following Fredholmness criterion proved by Taubes \cite{T87}.

\begin{lem}[{\cite[Lemma 4.3 $\&$ Lemma 4.5 ]{T87}}]\label{l2.3}
The elliptic complex $E_{\delta}(\tilde{V})$ is Fredholm if and only if $E_z(V)$ is acyclic for all $z$ with $\Rea z = \delta/2$. Moreover the spectral set $\Sigma(E)$ is discrete with no accumulation points unless $\Sigma(E) = \C$. 
\end{lem}

It follows from the fact $\partial^j_{z + 2\pi i } = e^{2\pi i f(x)} \partial^j_z e^{-2\pi i f(x)}$ that the spectral set $\Sigma(E)$ is $2\pi i $-periodic. Note that $\partial^{j+1}_z \circ \partial^j_z = 0$ gives us 
\[
\partial^{j+1} \circ \partial^{j} - z(\partial^{j+1}\circ \sigma^j + \sigma^{j+1} \circ \partial^j) + z^2 \sigma^{j+1} \circ \sigma^j = 0.
\]
Since $z \in \C$ is arbitrary, we conclude that 
\begin{equation}\label{e2.7}
\partial^{j+1}\circ \sigma^j + \sigma^{j+1} \circ \partial^j=0 \text{ and } \sigma^{j+1} \circ \sigma^j =0.
\end{equation}
In particular $\sigma^j$ induces a map on the homology of $E_0(V)$, which we denote by $\bar{\sigma}^j: H^j(E(V)) \to H^{j+1}(E_(V))$. Note that the map $\bar{\sigma}^j$ only depends on the cohomology class $1_V \in H^1(V;\Z)$. For simplicity we shall write $E_0(V)$ as $E(V)$. Motivated by Taubes \cite[Theorem 3.1]{T87}, we adopt the following non-degeneracy definition.

\begin{dfn}\label{d2.4}
We say a pair $(E(V), 1_V)$ defined as above is non-degenerate if the induced symbol sequence on cohomoology 
\begin{equation*}
0 \longrightarrow H^0(E(V)) \xrightarrow{\bar{\sigma}^0} H^1(E(V)) \xrightarrow{\bar{\sigma}^1} ... \xrightarrow{\bar{\sigma}^{n-1}} H^n(E(V)) \longrightarrow 0.
\end{equation*}
is exact.
\end{dfn}

\begin{rem}
For the results concerning index theory of end-periodic complexes, we always assume the non-degeneracy of $(E(V), 1_V)$ in the statement. However, this is actually stronger than we need. Due to the alternative of the spectral set $\Sigma(E)$ in \autoref{l2.3}, one can replace the assumption $(E(V), 1_V)$ by the weaker assumption that $\Sigma(E) \neq \C$. On the other hand, all the examples we know related to the index theory come from non-degenerate pairs. 
\end{rem}

\begin{lem}[{\cite[Theorem 3.1]{T87}}]\label{l2.5}
Suppose $(E(V), 1_V)$ is a non-degenerate pair. Then $\Sigma(E) \neq \C$. In particular $\Sigma(E)$ is discrete with no accumulation points.
\end{lem}

\begin{proof}
Due to \cite[Theorem 3.1]{T87}, it suffices to check that the index of $E(V)$ vanishes and $\bar{\sigma}^{j}: H^j(E(V))/ \im \bar{\sigma}^{j-1} \to H^{j+1}(E(V))$ is injective. These two conditions are equivalent to the exactness of the symbol sequence in \autoref{d2.4}. 
\end{proof}

We also consider end-periodic manifolds whose end is modeled on $\tilde{V}$ in the following sense. Suppose $M$ is a compact smooth $4$-manifold with boundary $\partial M = Y$. Then we form an end-periodic manifold $Z = M \cup \tilde{V}_+$ by attaching the positive end of $\tilde{V}$ to $M$ along their common boundary. We then extend $\tilde{f}|_{\tilde{V}_+}$ to a smooth function $\rho: Z \to \R$, and define $L^2_{k, \delta}(Z) = e^{-\delta\rho/2} L^2_k(Z)$. The corresponding $L^2_{k, \delta}$-completed elliptic complex is denoted by $E_{\delta}(Z)$. As pointed out in \cite[Proposition 4.1]{T87}, a slight modification of the argument in 
\cite{LM85} shows that $E_{\delta}(Z)$ is Fredholm if and only if $E_{\delta}(\tilde{V})$ is. 

Sometimes it would be more convenient to work with the differential operator associated to the differential complex by taking the adjoint of half of the complex. However, the dual space of $L^2_{k, \delta}$ is $L^2_{k, -\delta}$. So we adopt the following convention for the purpose of relating indices. The formal $L^2$-adjoint of $\partial^j$ is denoted by $\partial^{j, *}$. We write 
\begin{equation}
Q(V):= \partial \oplus \partial^* : \bigoplus_{j \text{ odd }} L^2_k(V, E_j) \longrightarrow \bigoplus_{j \text{ even }} L^2_{k-1}(V, E_j)
\end{equation}
for the associated elliptic operator, which satisfies $\ind Q(V) = \ind E(V)$. The formal $L^2_{\delta}$-adjoint of $\partial^j$ is given by 
\begin{equation}
\partial^{j, *}_{\delta}:= e^{-\delta \tilde{f}} \partial^{j,*} e^{\delta \tilde{f}} = \partial^{j, *} - \delta \sigma^{j, *},
\end{equation}
where $\sigma^{j, *}:= \sigma^{j, *}(d\tilde{f})$ is the symbol of the dual operator $\partial^{j, *}$ evaluated on $d\tilde{f}$. Then one can check easily that $\langle \partial^j v, u \rangle_{L^2_{\delta}} = \langle v, \partial^{j, *}u \rangle_{L^2_{\delta}}$. Similarly we write 
\begin{equation}
\partial^j_{\delta} :=e^{\delta \tilde{f}} \partial^j e^{-\delta \tilde{f}} = \partial^j - \delta \sigma^j.
\end{equation}
Due to the identification $L^2_{k, \delta} = e^{-\delta \tilde{f}/2}L^2_k$, the associated operator of the weighted elliptic complex $E_{\delta}(\tilde{V})$ has the form 
\begin{equation}
Q_{\delta}(\tilde{V}):= \partial_{\delta/2} \oplus \partial^*_{\delta/2}: \bigoplus_{j \text{ odd }} L^2_k(\tilde{V}, \tilde{E}_i) \longrightarrow \bigoplus_{j \text{ even }} L^2_{k-1}(\tilde{V}, \tilde{E}_i).
\end{equation}
Then we have $\ind Q_{\delta}(\tilde{V}) = \ind E_{\delta}(\tilde{V})$ when $E_{\delta}(\tilde{V})$ is Fredholm. 

\subsection{\em The Homotopy Resolvent}\label{ss2.2} \hfill
 
\vspace{3mm}

Let $(E(V), 1_V)$ be a non-degenerate pair as above. Then the associated holomorphic family $\{E_z(V)\}_{z \in \C}$ is non-acyclic only on a discrete set $\Sigma(E) \subset \C$. When the length of $E_V$ is two, we get a holomorphic family of elliptic operators $\partial^0_z: L^2_k(V, E_0) \to L^2_{k-1}(V, E_1)$ that are invertible away from $\Sigma(E)$. The inverse of such a family is proved to form a meromorphic family of operators in \cite[Theorem 4.6]{MRS11}, which plays an important role in deducing the dependence of the index of such elliptic operators on the weights in completing the functional spaces. To generalize the family of the inverse for operators to the case of differential complexes, we introduce the notion of homotopy resolvent. 

\begin{dfn}
Let $E_z=(C_z, \partial_z)$ be a holomorphic family of differential complexes parametrized by $z \in \C$. A homotopy resolvent of $(C_z, \partial_z)$ is a family of chain homotopies $\{R_z\}$ defined for 
\[
z \notin \Sigma(E):=\{ z \in \C: E_z \text{ is not acyclic } \}
\]
so that $R_z$ connects the identity map to the zero map, i.e. $R_z^j: C^j_z \to C^{j-1}_z$ satisfies 
\begin{equation}\label{r2.12}
\partial^{j-1}_z \circ R^j_z + R^{j+1}_z \circ \partial^j_z = \id.
\end{equation}
\end{dfn}

\begin{rem}
For all complexes we will be considering, $C_z$ is independent of $z$ as vector spaces, and only the differentials $\partial_z$ depend on $z$. Since the space of differentials of $C_z$ is a linear space, we can take derivatives with respect to $\bar{z}$ when a family is parametrized by an open subset in $\C$. When the derivative vanishes, the family $E_z=(C_z, \partial_z)$ is referred to as a holomorphic family. In general, a family of complexes parametrized by $\C$ being holomorphic means the vector spaces $C_z$ can be locally trivialized over $\C$ with holomorphic transition maps, and the differentials $\partial_z$ vary holomorphically under the local trivializations. 
\end{rem}

When the length of the differential complex $E_z$ is two, it follows from the definition that a homotopy resolvent $R_z$ is the inverse of $\partial_z$ defined for $z \in \C$ so that $\partial_z$ is invertible. Thus the homotopy resolvent is unique. In general, homotopy resolvents of a family $E_z$ are not unique, and the choices can be determined as follows.

Given a cochain complex $E=(C, \partial)$ graded by $\Z_{\geq 0}$, we can consider the corresponding $\Z$-graded object in the differential graded category of cochain complexes, say $\mathscr{E}=(\mathscr{C}, \mathscr{d})$, given by 
\[
\mathscr{E} = \bigoplus_{n \in \Z} \Hom(C, C[-n]) \qquad \mathscr{d}_n(f_n)|_{C^j} = f^{j+1}_n \circ \partial^j + (-1)^{n+1}\partial^{j-n}\circ f^j_n,
\]
where $C[-n]$ is the cochain complex obtained from $C$ by shifting its degree by $-n$. Then $\mathscr{E}$ is a chain complex, whose homology is denoted by $H_*(\mathscr{E})$. The following property of $\mathscr{E}$ is well-known in homological algebra. We supply a proof here since we cannot find an appropriate reference. 

\begin{lem}\label{l2.8}
Let $E=(C, \partial)$ and $\mathscr{E}=(\mathscr{C}, \mathscr{d})$ be two chain complexes related as above. Then $\mathscr{E}$ is acyclic if $E$ is acyclic. 
\end{lem} 

\begin{proof}
Suppose $f \in \mathscr{C}_n$ satisfies $\mathscr{d}_n(f) = 0$. We wish to construct $g \in \mathscr{C}_{n+1}$ so that $\mathscr{d}_{n+1}(g) = f$, which would imply $H_n(\mathscr{E}) = 0$. The construction of $g$ is based on induction. Since $g \in \Hom(C, C[-n-1])$ and $C$ is graded by $\Z_{\geq 0}$, we have $g^j = 0$ for $j \leq n$. For each vector space $C^j$, we choose an arbitrary splitting $C^j = \im \partial^{j-1} \oplus (\im \partial^{j-1})^{\perp}$. When $j=n+1$, we note that 
\[
0 = \mathscr{d}_n(f)|_{C^{n-1}} = f^n \circ \partial^{n-1}. 
\]
Since $\partial^n: (\im \partial^{n-1})^{\perp} \to \im \partial^n$ is an isomorphism, we can set $g^{n+1}: C^{n+1} \to C^0$ as
\[
g^{n+1}(a) = 
\begin{cases}
0 & \text{ if $a \in (\im \partial^n)^{\perp}$} \\
f^n \circ (\partial^n)^{-1} & \text{ if $a \in \im \partial^n$}
\end{cases}
\]
which satisfies $f^n = g^{n+1} \circ \partial^n$. Now suppose we have constructed $g^j$ for some $j > n+1$. Then we know 
\[
\begin{split}
0 = \mathscr{d}_nf|_{C^{j-1}} & = f^j \circ \partial^{j-1} + (-1)^{n+1} \partial^{j-1-n} \circ f^{j-1} \\
& = f^j \circ \partial^{j-1} + (-1)^{n+1} \partial^{j-1-n} \circ g^{j} \circ \partial^{j-1}. 
\end{split}
\]
Applying the above argument again, we can find $g^{j+1}: C^{j+1} \to C^{j-n}$ so that 
\[
g^{j+1} \circ \partial^j = f^j + (-1)^{n+1} \partial^{j-1-n} \circ g^j,
\]
which is equivalent to $\mathscr{d}_n(g)|_{C^j} = f^j$. This completes the proof. 
\end{proof}

With the help of \autoref{l2.8}, we can give a description for the choices of homotopy resolvents. 

\begin{cor}\label{c2.9}
Let $E=(C, \partial)$ be an acyclic $\Z_{\geq 0}$-graded cochain complex. Then the space of chain homotopies from $\id$ to $0$ can be identified with $\im \mathscr{d}_2 = \ker \mathscr{d}_1 \subset \mathscr{C}_1$. 
\end{cor}

\begin{proof}
We start with the existence of such chain homotopies. Note that $\id \in \ker \mathscr{d}_0$. By \autoref{l2.8}, one can find $R \in \mathscr{C}_1$ so that $\mathscr{d}_1(R) = \id$. The difference of two such chain homotopies lies in $\ker \mathscr{d}_1$, which is identified with $\im \mathscr{d}_2$. 
\end{proof}

Now we move back to consider the elliptic complex $(E_z(V), \partial_z)$ defined over the closed $4$-manifold $V$. We wish to construct a homotopy resolvent for this family. Presumably there are plenty of choices by applying \autoref{c2.9} pointwise to $z \notin \Sigma(E)$. So we would like to impose further restriction on this construction to get better properties. The following result is one of the key steps in proving the index change formula, whose proof is based on discussion with Donghao Wang. 

\begin{prop}\label{p2.10}
Let $E(V)$ be an elliptic complex associated to a closed $4$-manifold $V$ and a primitive class $1_V \in H^1(V;\Z)$ so that $(E(V), 1_V)$ is non-degenerate. Then the Laplace-Fourier transformed family $(E_z(V), \partial_z)$ admits a meromorphic homotopy resolvent $R_z$ whose poles consist of elements in the spectral set $\Sigma(E)$. Moreover $R_{z+2\pi i } = e^{2\pi i f} R_z e^{-2\pi i f}$.
\end{prop}

\begin{proof}
The proof is divided into two parts. The first part is concerned with a local construction of the meromorphic homotopy resolvent near a spectral point of $E_z(V)$. The second part is devoted to patching up these local meromorphic homotopy resolvents by treating it as a Mittag-Leffler type problem. 

{\bf \em \large{Step 1. }} Let's start with the local construction. Fix a spectral point $z_o \in \Sigma(E)$. Due to the non-degeneracy assumption of $(E(V), 1_V)$, $z_o$ is isolated from the rest points in $\Sigma(E)$. For each $C_{z_o}^j = L^2_{k-j}(V, E_j)$, we fix a splitting of the form $C_{z_o}^j=C^j_1 \oplus C^j_2 \oplus C^j_3$ so that $C^j_1 = \im \partial^{j-1}_{z_o}$, $C^j_2 \simeq H^j(E_{z_o}(V))$, and $C^j_3 = (\ker \partial_{z_o}^j)^{\perp}$. With respect to this splitting, only the $(1,3)$-entry $\partial^j_{13}(z_o)$ of $\partial^j_{z_o}$ is nonvanishing, which is also an isomorphism. Appealing to the relation $\partial^{j+1}_{z_o} \circ \sigma^j + \sigma^{j+1} \circ \partial^j_{z_o} = 0$, we see that $\sigma^j=(\sigma^j_{kl})$ is represented by an upper triangular matrix. We then choose a disk neighborhood $U_o$ of $z_o$ so that $\partial^j_{13}(z_o) - (z-z_o) \sigma^j_{13}$ is invertible for all $j=0, ..., n$ and $z \in U_o$. Then $\partial^j_z$ is represented by an upper triangular matrix whose $(1,3)$-entry $\partial^j_{13}(z)$ is invertible for $z \in U_o$.  

We claim that one can apply isomorphisms $\phi^j_z: C^j_{z_o} \to C^j_{z_o}$ that depend holomorphically on $z \in U_o$ so that the matrix representing $\phi^{j+1}_z \circ \partial^j_z \circ (\phi^j_z)^{-1} : C^j_{z_o} \to C^{j+1}_{z_o}$ takes the form 
\[
a^j_z = 
\begin{pmatrix}
0 & 0 & a^j_{13}(z) \\
0 & a^j_{22}(z) & 0 \\
0 & 0 & 0 
\end{pmatrix}
\]
with $a^j_{13}(z)$ invertible. Indeed, since $\partial^j_{13}(z)$ is invertible, one can apply column operations making the $(1,1)$- and $(1, 2)$-entries zero. Combining the relation $\partial^{j+1}=0 \circ \partial^j_z = 0$ and the invertibility of $\partial^j_{13}(z)$, we see that the $(3,1)$- and $(3, 2)$-entries have to vanish as well. Then we apply row operations to eliminate the $(2, 3)$- and $(3,3)$-entries. Again, the relation $\partial^{j+1}_z\circ \partial^j_z = 0$ implies that the $(2,1)$-entry has to vanish. So only the $(2,2)$- and $(1,3)$-entries survive in this process, which gives the desired form above. Note that performing column and row operations corresponds to conjugating $\partial^j_z$ via isomorphisms of the domain $C^j_{z_o}$ and range $C^{j+1}_{z_o}$ which are represented by matrices with entries given by the product of $(\partial^j_{13}(z))^{-1}$ and other entries of $\partial^j_z$. Since one can write $\partial^j_{13}(z) = \partial^j_{13}(z_o) - (z-z_o)\sigma^j_{13}$ with $\partial^j_{13}(z_o)$ invertible and $\sigma^j_{13}$ is a compact operator, we see that $(\partial^j_{13}(z))^{-1}$ depends holomorphically on $z$, which further implies that the isomorphisms on $C^j_{z_o}$ depend holomorphically on $z$. 

Now it suffices to construct a meromorphic homotopy resolvent $P^j_z$ for $a^j_z = \phi^{j+1}_z \circ \partial^j_z \circ (\phi^j_z)^{-1}$. Then $R^j_z = (\phi^{j-1}_z)^{-1} \circ P^j_z \circ \phi^j_z$ is a meromorphic homotopy resolvent for $\partial^j_z$. Let's write $\alpha^j_z:= a^j_{22}(z)$. Then $(H^*(E_{z_o}(V)), \alpha^*_z)$ forms a holomorphic family of finite dimensional cochain complexes that are acyclic when $z \neq z_o$. Tracing back the construction, the differential $\alpha^j_z$ takes the form
\[
\begin{split}
\partial^j_{22}(z) - & \partial^j_{12}(z) (\partial^j_{13}(z))^{-1} \partial^j_{23}(z) 
= -(z-z_o)\sigma^j_{22} - (z-z_o)^2\sigma^j_{12}\left(\partial^j_{13}(z_o) - (z-z_o)\sigma^j_{13}\right)^{-1} \sigma^j_{23} \\
& = -(z-z_o)\sigma^j_{22} - (z-z_o)\sigma^j_{12}\left ((z-z_o)^{-1} - (\partial^j_{13}(z_o))^{-1}\sigma^j_{13}\right)^{-1}(\partial^j_{13}(z_o))^{-1} \sigma^j_{23}
\end{split}
\]
Since operator $\sigma^j_{13}$ comes from the zeroth order differential operator given by the symbol of $\partial^j$, the operator $(\partial^j_{13}(z_o))^{-1}\sigma^j_{13}$ is compact, whose resolvent $(\lambda - (\partial^j_{13}(z_o))^{-1}\sigma^j_{13})^{-1}$ is thus meromorphic in $\lambda$. Substituting $\lambda$ by $(z-z_o)^{-1}$, we see that $\alpha^j_z$ is expressed as a polynomial of $z-z_o$. 

With this description, we can now construct a meromorphic homotopy resolvent $Q^j_z$ for $\alpha^j_z$ by induction on $j$. When $j=1$, we have the short exact sequence:
\begin{equation}\label{e2.13}
0 \to H^0(E_{z_o}(V)) \xrightarrow{\alpha^0_z} H^1(E_{z_o}(V)) \xrightarrow{\alpha^1_z} \im \alpha^1_z \to 0. 
\end{equation}
Let's write $U^c_o = U^o\backslash{z_o}$ for the punctured disk around $z_o$. Since $\alpha^j_z$ is holomorphic in $z$, $\im \alpha^1_z$ forms a holomorphic bundle over $U^c_o$ as we vary $z$. So (\ref{e2.13}) can be thought as a short exact sequence of holomorphic bundles over $U^c_o$. All possible extensions of the trivial bundle $H^0(E_{z_o}(V)) \times U^c_o$ by $\im \alpha^1_z$ are parametrized by $H^1(U^c_o, H^0(E_{z_o}(V)) \otimes (\im \alpha^1_z)\smvee)$ which vanishes due to the fact that $U^c_o$ is a connected non-compact Riemann surface. Thus we can find a holomorphic splitting $H^1(E_{z_o}(V)) \times U^c_o = A^1_z \oplus B^1_z$ with $A^1_z=\im \alpha^0_z$ and $B^1_z$ both holomorphic bundles over $U^c_o$. The corresponding splitting map is the homotopy resolvent 
\begin{equation}
Q^1_z(a)=
\begin{cases}
(\alpha^0_z)^{-1} a & \text{ if $a \in A^1_z$} \\
0 & \text{ if $a \in B^1_z$}
\end{cases}
\end{equation}
Since $\alpha^0_z: H^0(E_{z_o}(V)) \to A^1_z$ is a bundle isomorphism that can be expressed as a polynomial in $z$, it's inverse is a rational function in $z$. This shows that the homotopy resolvent $Q^1_z$ is meromorphic. Inductively at the $j$-th stage for $j > 1$, we get the short exact sequence:
\[
0 \to B^{j-1}_z \xrightarrow{\alpha^{j-1}_z} H^j(E_{z_o}(V)) \xrightarrow{\alpha^j_z} \im \alpha^j_z \to 0,
\]
which enables us to choose a holomorphic splitting $H^j(E_{z_o}(V)) = A^j_z \oplus B^j_z$ with $A^j_z = \im \alpha^{j-1}_z$, and the splitting map gives the homotopy resolvent $Q^j_z$ which is meromorphic due to the same reason above. 

Recall that $a^j_{13}(z) = \partial^j_{13}(z_o) - (z-z_o) \sigma^{13} = \partial^j_{13}(z_o)(1 - (z-z_o)(\partial^j_{13}(z_o))^{-1}\sigma^j_{13})$ whose inverse is meromorphic due to the compactness of $\sigma^j_{13}$. Then it's straightforward to check that $P^j_z:=(a^j_{13}(z))^{-1} \oplus Q^j_z$ is a meromorphic homotopy resolvent of $a^j_z$. In terms of matrices, $a^j_z$ and $P^j_z$ take the following form:
\[
a^j_z=
\begin{pmatrix}
0 & 0 & a^j_{13}(z) \\
0 & a^j_{22}(z) & 0 \\
0 & 0 & 0  
\end{pmatrix}
\qquad 
P^j_z:=
\begin{pmatrix}
0 & 0 & 0 \\
0 & Q^j_z & 0 \\
(a^j_{13}(z))^{-1} & 0 & 0
\end{pmatrix}.
\]
This completes the local construction. 

{\bf \em \large{Step 2. }} Next we move on to consider the global construction. To better cooperate with the $2\pi i$-periodicity of $\partial^j_z$, we consider the infinite cylinder $A:= \C/ \sim$ with $z \sim z+ 2\pi i$ and the holomorphic Hilbert bundles $\mathcal{C}^j:= \C \times L^2_{k-j}(V, E^j) / \sim$ with $(z, s) \sim (z+2\pi i , e^{2\pi i f} s e^{-2\pi if})$. Due to the $2\pi i$-periodicity of $\partial^j_z$, we get a holomorphic differential complex $(\mathcal{C}^*, \mathfrak{d}^*)$ consisting of holomorphic Hilbert bundles over $A$, where $\mathfrak{d}^j: \mathcal{C}^j \to \mathcal{C}^{j+1}$ is the induced bundle map from $\partial^j_z$. A meromorphic homotopy resolvent of $(\mathcal{C}^*, \mathfrak{d}^*)$ will then gives back a meromorphic homotopy resolvent for $(E_z(V), \partial_z)$ satisfying the required $2\pi i $-periodicity condition. 

Let $\Hom(\mathcal{C}, \mathcal{C}[-n])$ be the holomorphic bundle over $A$ consisting of fiberwise bounded homomorphisms of degree $-n$. The corresponding Dolbeault complexes $\Omega^{0,q}(A, \Hom(\mathcal{C}, \mathcal{C}[-n]))$ are equipped with chain maps 
\[
\mathscr{d}_n: \Omega^{0,q}(A, \Hom(\mathcal{C}, \mathcal{C}[-n])) \longrightarrow \Omega^{0,q}(A, \Hom(\mathcal{C}, \mathcal{C}[-n+1])). 
\]
We note that $\mathcal{C}^j$ is holomorphically trivial over $A$ under the identification $[z, s] \mapsto ([z], e^{zf} se^{-zf})$. Thus $\Hom(\mathcal{C}, \mathcal{C}[-n])$ is holomorphically trivial over $A$ as well. In particular, the Dolbeault cohomology $H^{0,1}(A, \Hom(\mathcal{C}, \mathcal{C}[-n])) =0$ since $A$ is a non-compact Riemann surface. 

Although in step 1 we only constructed meromorphic homotopy resolvent locally around a spectral point $z_o \in \Sigma(E)$, the construction readily applies to non-spectral points and results in local holomorphic homotopy resolvents. Now we cover the infinite cylinder $A$ with countable many open disks $U_{\alpha} \subset \C$ so that each spectral point is the center of the unique open disk containing it. We further assume over each open disk $U_\alpha$ a meromorphic or holomorphic resolvent $R_{\alpha}$ has been constructed depending whether it contains a spectral point or not. Let $\{\rho_{\alpha}\}$ be a partition of unity subordinate to the cover $\{U_{\alpha}\}$. Then $r = \sum \rho_{\alpha}R_{\alpha}$ is a $C^{\infty}$ homotopy resolvent whose $(0,1)$-differential $\bar{\partial} r$ is a closed $(0,1)$-form valued in $\Hom(\mathcal{C}, \mathcal{C}[-1])$ and satisfies $\mathscr{d}_1(\bar{\partial}r) = 0$. It follows from \autoref{l2.8} that there exists $\xi \in \Omega^1_2$ with $\mathscr{d}_2 \xi = \bar{\partial} r$. Since $A$ is a Riemann surface, the $(0,1)$-form $\xi$ is automatically $\bar{\partial}$-closed. Invoking the vanishing fact $H^{0,1}=0$, we then find $\eta \in \Omega^0_2$ with $\bar{\partial} \eta = \xi$. Then $R = r - \mathscr{d}_2 \eta$ is a meromorphic homotopy resolvent defined over $A$ whose poles consist of the spectral points. 
\end{proof}

\subsection{\em Index Comparison}\label{ss2.4} \hfill
 
\vspace{3mm}

Let $(E(V), 1_V)$ be a non-degenerate pair in the sense of \autoref{d2.4}, and $Z$ an end-periodic manifold formed as before Suppose $\delta_1$ and $\delta_2$ are two constant weights chosen so that $E_{\delta_i}(Z)$ is Fredholm for $i=1, 2$. We wish to compare the difference between $\ind E_{\delta_1}(Z)$ and $\ind E_{\delta_2}(Z)$. To this end, we choose a smooth function $\tilde{\delta}: \tilde{V} \to \R$ satisfying that 
\begin{equation}
\tilde{\delta}|_{W_n}=
\begin{cases}
\delta_1 & \text{ if $n \leq -1$} \\
\delta_2 & \text{ if $n \geq 1$}
\end{cases}
\end{equation}
serving as a weight function. We further require that $\supp d\tilde{\delta} \cap \supp \tilde{f} = \varnothing$ after a possible modification near $W_0$. This implies that $d(\tilde{\delta}\tilde{f}) = \tilde{\delta} d\tilde{f}$. 

\begin{lem}\label{l2.11}
With the notations above, we have 
\[
\ind E_{\delta_2}(Z) - \ind E_{\delta_1}(Z) = \ind E_{\tilde{\delta}}(\tilde{V}). 
\]
\end{lem}

\begin{proof}
The argument is an application of the excision principle relating indices of elliptic operators. We shall work with the corresponding operators $Q$ rather than the complexes $E$. 

Let $Z^* = \tilde{V}_- \cup \overline{M}$ be the manifold obtained by attaching the negative end to the orientation-reversed manifold $\overline{M}$. We identify collar neighborhoods $\overline{Y} = \partial \overline{M}$ and $Y = \partial \tilde{V}_-$ with $[0, \epsilon) \times Y$ and $(-\epsilon, 0] \times Y$ respectively. Choose a smooth function $\rho^*: Z^* \to \R$ satisfying $\rho^*|_{\overline{M} \backslash [0, \epsilon) \times Y} = 0$, $\rho^*|_{\tilde{V}_-} = \tilde{f}$, and $\rho^*|_{(-\epsilon/2, \epsilon/2) \times Y} = \rho|_{(-\epsilon/2, \epsilon/2) \times Y}$. Thus we have 
\[
Q_{\delta_1}(Z)|_{(-\epsilon/2, \epsilon/2) \times Y} = Q_{\delta_1}(Z^*)|_{(-\epsilon/2, \epsilon/2) \times Y}. 
\]
Since the closure of the overlap $(-\epsilon/2, \epsilon/2) \times Y$ is compact, the excision principle of elliptic operators (c.f. \cite[Section 5.3]{MRS11} tells us that 
\begin{equation}\label{e2.12}
\ind Q_{\delta_1}(Z) + \ind Q_{\delta_1}(Z^*) = \ind Q(M \cup \overline{M}) + \ind Q_{\delta_1}(\tilde{V}).
\end{equation}

We claim that $\ind Q_{\delta_1}(\tilde{V}) = 0$. Indeed, given $s \in \ker Q_{\delta_1}(\tilde{V})$, the ellipticity implies that $s \in C^{\infty} \cap L^2_k$. Thus the Fourier-Laplace transform (\ref{e2.5}) $\hat{s}_z$ converges for all $z$ with $\Rea z = 0$ and $Q_{\delta_1, z}(\tilde{V}) \hat{s}_z = 0$. Note that $Q_{\delta_1, z}(\tilde{V}) = \partial_{\delta_1/2+z} \oplus \partial^*_{\delta_1/2+z}$ for all $\Rea z = 0$, where the latter is the operator corresponding to the elliptic complex $E_{\delta_1/2+z}(V)$. Due to the Fredholmness criterion, $E_{\delta_1/2+z}(V)$ is acyclic when $\Rea z =0$. Thus $\ker Q_{\delta_1, z}(\tilde{V}) = 0$, and $\hat{s}_z = 0$ as a consequence. Applying the inverse of the Fourier-Laplace transform, we conclude that $\ker Q_{\delta_1}(\tilde{V}) = 0$. The same argument can be applied to the adjoint $Q^*_{\delta}(\tilde{V})$ to conclude that $\cok Q_{\delta_1}(\tilde{V}) = 0$. This completes the proof of the claim. 

A priori, we can modify $\rho$ and $\rho^*$ respectively near $W_0 \subset Z$ and $W_{-1} \subset Z^*$ to get $\rho'$ and ${\rho^*}'$ satisfying $\delta_2 \rho'|_{\tilde{V}_+} = \tilde{\delta}\tilde{f}|_{\tilde{V}_+}$ and $\delta_1 {\rho^*}'|_{\tilde{V}_-} = \tilde{\delta}\tilde{f}|_{\tilde{V}_-}$. Since this modification is made over a compact region, the indices are not changed. Applying the exicision principle to $Q_{\delta_2}(Z)$ and $Q_{\delta_1}(Z^*)$ gives us 
\begin{equation}\label{e2.13}
\ind Q_{\delta_2}(Z) + \ind Q_{\delta_1}(Z^*) = \ind Q(M \cup \overline{M}) + \ind Q_{\tilde{\delta}}(\tilde{V}). 
\end{equation}
Combining (\ref{e2.12}) and (\ref{e2.13}), we get the desired result. 
\end{proof}

The next step is to formulate $\ind E_{\tilde{\delta}}(\tilde{V})$ in terms of certain spectral asymmetry coming from the holomorphic family of elliptic complexes $\{E_z(V)\}$. In what follows, we will actually identify the cohomology $H^*(E_{\tilde{\delta}}(\tilde{V})$ explicitly in terms of the cohomology $H^*_z(V)$. The idea lies behind the following observation. The cochain complex $E_{\tilde{\delta}}(\tilde{V})$ is a $\C[t, t^{-1}]$-module where $t$ acts by the pull-back of the covering transformation $\tau^*$. Since this complex is Fredholm, $H^j(E_{\tilde{\delta}}(\tilde{V}))$ must be a torsion $\C[t, t^{-1}]$-module, thus takes the following form:
\begin{equation}
H^j(E_{\tilde{\delta}}(\tilde{V})) \simeq \bigoplus_{z} \C[t, t^{-1}]/ (t - z)^{m_z},
\end{equation}
where $m_z$ is a positive integer given by the dimension of the generalized $z$-eigenspace of $\tau^*$ acting on $H^j(E_{\tilde{\delta}}(\tilde{V}))$. So our task is to determine which values of $z$ could come into the picture and formulate the multiplicity $m_z$ into a proper way. We shall construct representatives that are annihilated by $(t-z)^{m_z}$ directly and prove a one-to-one correspondence. To this end, we introduce a version of equivariant cohomology induced from $(E_z(V), \partial_z)$. 

Let $\C[u, u^{-1}]$ the ring of Laurent polynomials of a formal variable $u$, and $\mathcal{R}:= \C[u, u^{-1}]/ \C[u]$ the ring consisting of the principal part only. We define the cochain complex $(\widehat{C}_z, \widehat{\partial}_z)$ to be 
\begin{equation}
\widehat{C}^j_z:= C^j_z \otimes_{\C} \mathcal{R} \text{ and } \widehat{\partial}^j_z(\pmb{\alpha}) := \partial^j_z (\pmb{\alpha}) - u\sigma^j (\pmb{\alpha}).
\end{equation}
When $\pmb{\alpha}$ is represented by $\sum_{l=-m}^{-1} \alpha_lu^l$ with $\alpha_l \in C^j_z=L^2_{k-j}(V, E_j)$, the boundary operator $\widehat{\partial}^j_z$ is equivalent to the system of differential operators:
\begin{equation}
\partial^j_z \alpha_{-1} - \sigma^j \alpha_{-2}, \quad ... \quad , \partial^j_z \alpha_{-m+1} - \sigma^j \alpha_{-m}, \quad\partial^j_z \alpha_{-m}. 
\end{equation} 
It follows from the relations (\ref{e2.7}) that $\widehat{\partial}_z^2 = 0$, thus $(\widehat{C}_z, \widehat{\partial}_z)$ forms a cochain complex whose homology is denoted by $\widehat{H}^*_z(V):= H^*(\widehat{C}_z, \widehat{\partial}_z)$. We call\ $\widehat{H}^*_z(V)$ an equivariant cohomology because its construction takes a similar pattern of Cartan's definition of the equivariant cohomology on a $U(1)$-manifold. 

Given $\pmb{\alpha} = \sum_{l=-m}^{-1} \alpha_l u^l$, we refer to $m$ as the $u$-degree of $\pmb{\alpha}$ if $\alpha_{-m} \neq 0$ and $\alpha_{-k} =0$ for all $k \geq m$. We denote  the $u$-degree of $\pmb{\alpha}$ by $m(\pmb{\alpha})$. Then the $u$-degree gives us a decreasing filtration of the cochain complex $\widehat{C}_z$:
 \[
 \{0\} = F^0\widehat{C}_z \subset F^{-1} \widehat{C}_z \subset ... \subset F^{i}\widehat{C}_z \subset ... \subset \widehat{C}_z,
 \]
where $F^i\widehat{C}_z$ consists of elements with $u$-degree at most $-i$. Since this filtration is bounded above and preserved by the boundary operator $\widehat{\partial}_z$ , we get a spectral sequence $\{\widehat{E}^{*,*}_{z,r}, d_r\}$ that abuts to $\widehat{H}^*_z(V)$. The first page of the spectral sequence can be read as (see for instance \cite[Theorem 2.6]{M01}):
\begin{equation}
\widehat{E}^{i,j}_{z, 1} = H^{i+j}(E_z(V)) \cdot u^i \qquad d^{i,j}_1 = -u\bar{\sigma}^{i+j},
\end{equation}
where $\bar{\sigma}^{i+j}: H^{i+j}(E_z(V)) \to H^{i+j+1}(E_z(V))$ is induced from $\sigma^{i+j}$. 

Now we are ready to establish the index jump formula. 

\begin{thm}\label{t2.12}
Suppose $\delta_1 > \delta_2$. Let's consider the map:
\begin{equation}\label{e2.22}
\begin{split}
\varphi: \bigoplus_{[z]} \widehat{C}^j_z & \longrightarrow E^j_{\tilde{\delta}}(\tilde{V}) \\
\bigoplus_{[z]} \sum_{l=-m_z}^{-1} \alpha_l(z) u^l & \longmapsto \sum_z e^{-z \tilde{f}} \sum_{l=-m_z}^{-1} (-1)^{-l-1}(\tilde{f})^{-l-1}\alpha_l(z)/(-l-1)!
\end{split}
\end{equation}
where the sum ranges over $[z] \in \Sigma(E)/2\pi i$ with $\Rea z \in (\delta_2/2, \delta_1/2)$. Then the map $\varphi$ induces an isomorphism on cohomology. Moreover $H^n(E_{\tilde{\delta}}(\tilde{V})) = 0$. 
\end{thm}

\begin{proof}
It's straightforward to check that $\varphi \circ \widehat{\partial} = \partial \circ \varphi$. So $\varphi$ induces a map on cohomology, which we denote by $\varphi^*$

We first establish the surjectivity, which is divided into three cases. The first case is when $j=0$, and follows a similar pattern as the operator case in \cite[Section 6.2]{MRS11}. Let $s \in L^2_{k, \tilde{\delta}}(\tilde{V}, \tilde{E}_0)$ be a representative of $[s] \in H^0(E_{\tilde{\delta}}(\tilde{V}))$. Choose a cut-off function $\eta=(\tilde{\delta}- \delta_2)/(\delta_1 - \delta_2): \tilde{V} \to \R$. Then 
\[
\eta|_{W_n} = 
\begin{cases}
1 & \text{ if $n \leq -1$} \\
0 & \text{ if $n \geq 1$}
\end{cases}.
\]
Moreover $\tilde{\delta} = \delta_1\eta + \delta_2(1-\eta)$. We decompose $s$ into two parts: $s = \eta s + (1-\eta) s=: s_1 + s_2$. In particular $s_i \in L^2_{k, \delta_i}(\tilde{V}, \tilde{E}_0)$ for $i=1, 2$, and $\partial^0 s_1 = - \partial^0 s_2$. Denote by $\zeta:= \partial^0 s_1$. Elliptic regularity tells us that $\zeta \in C^{\infty}_0(\tilde{V})$ with $\supp \zeta \subset \supp d\eta$. 
Note that $\supp s_1 \subset \tilde{V}_-$ and $\supp s_2 \subset \tilde{V}_+$. Applying the Fourier-Laplace transform gives us 
\begin{equation}\label{e2.23}
\partial^0_z \hat{s}_{1, z} = \hat{\zeta}_z \text{ when $\Rea z \geq \delta_1/2$} \qquad \partial^0_z \hat{s}_{2, z} = -\hat{\zeta}_z \text{ when $\Rea z \leq \delta_2/2$}.
\end{equation}
Since $R^1_z \circ \partial^0_z = \id$, applying the inverse Fourier-Laplace transform results in 
\[
s(x+n) = \frac{1}{2\pi i } \left( \int_{I(\delta_1/2+i\alpha)} e^{-z(\tilde{f}(x)+n)} R^1_z \hat{\zeta}_z(\pi(x)) dz - \int_{I(\delta_2/2+i\alpha)} e^{-z(\tilde{f}(x)+n)} R^1_z \hat{\zeta}_z(\pi(x)) dz \right),
\]
where $\alpha \in \R$ is chosen so that the poles of $R^1_z$ avoid the line $\Ima z = i\alpha$. Consider the rectangular region $S=[\delta_2/2, \delta_1/2] \times i [\alpha, \alpha+2\pi] \subset \C$ and $\Gamma = \partial S$ the contour endowed with the boundary orientation. Due to the fact that $R^1_{z+2\pi i } = e^{2\pi i f(x)} \circ R^1_z \circ e^{-2\pi i f(x)}$, the integration of $e^{-z(\tilde{f}(x)+n)} R^1_z \hat{\zeta}_z(\pi(x))$ along the two horizontal sides of $S$ cancels each other (c.f. \cite[Lemma 6.2]{MRS11}). Thus Cauchy's residue theorem tells us 
\begin{equation}\label{e2.24}
s(x+n) = \sum_{z_k \in \Sigma(E) \cap S} \Res(e^{-z(\tilde{f}(x)+n)} R^1_z \hat{\zeta}_z(\pi(x)); z_k). 
\end{equation}
Locally near a pole $z_k \in \Sigma(E) \cap S$, we write the Laurent series of $R^1_z\hat{\zeta}_z(\pi(x))$ as 
\begin{equation}\label{e2.25}
R^1_z\hat{\zeta}_z(\pi(x)) = \sum_{l=-m_k}^{\infty} \alpha_l(\pi(x)) (z - z_k)^l,
\end{equation}
where $\alpha_l \in L^2_k(V, E_0)$. Note that when $\Rea z > \delta_1/2$, we have $\partial^1_z \hat{\zeta}_z = \partial^1_z \circ \partial^0_z \hat{s}_{1, z} = 0$. Since $\partial^1_z \hat{\zeta}_z$ is an entire function in $z$, we see that $\partial^1_z \hat{\zeta}_z = 0$ for all $z \in \C$. Combining (\ref{e2.25}) with the identity $\partial^0_z \circ R^1_z + R^2_z \circ \partial^1_z = \id$, we conclude that 
\begin{equation}\label{e2.26}
\hat{\zeta}_z = \sum_{l=-m_k}^{\infty} \partial^0_z \alpha_l (z - z_k)^l = \sum_{l=-m_k}^{\infty} \left( \partial^0_{z_k} \alpha_l \cdot (z-z_k)^l  - \sigma^0 \alpha_l \cdot (z-z_k)^{l+1} \right),
\end{equation}
where we used the fact that $\partial^0_z = \partial^0_{z_k} - (z-z_k)\sigma^0$. Since $\hat{\zeta}_z$ is an entire function in $z$, the principal part of (\ref{e2.26}) has to vanish. Thus we have 
\begin{equation}\label{e2.27}
\partial^0_{z_k} \alpha_{-1} = \sigma^0 \alpha_{-2}, ..., \; \partial^0_{z_k} \alpha_{-m_k+1} = \sigma^0 \alpha_{-m_k}, \; \partial^0_{z_k} \alpha_{-m_k} = 0. 
\end{equation}
Substituting back to (\ref{e2.24}), we conclude that
\[
s(x) = \sum_{z_k \in \Sigma(E) \cap S} e^{-z_k (\tilde{f}(x)} \sum_{l=-m_k}^{-1} (-1)^{-l-1} (\tilde{f}(x))^{-l-1} \alpha_{-l}(\pi(x))/(-l-1)!,
\]
which is $\sum_{z_k} \varphi(\pmb{\alpha}(z_k))$ with $\pmb{\alpha}(z_k) = \sum_{l=-m_k}^{-1} \alpha_l u^l$. 

Next we deal with the case when $0<j<n$. Let $s \in L^2_{k-j, \tilde{\delta}}(\tilde{V}, \tilde{E}_j)$ be a representative of $[s] \in H^j(E_{\tilde{\delta}}(\tilde{V}))$. Same as above, we can decompose $s = s_1 + s_2$ and obtain the same equations in (\ref{e2.23}). The identity $R^{j+1}_z \partial^j_z+ \partial^{j-1}_z R^j_z = \id$ gives us 
\begin{equation}
(-1)^{m-1}R^{j+1}_z \hat{\zeta}_z + \partial^{j-1}_z R^j_z \hat{s}_{m, z} = \hat{s}_{i, z}, \qquad m = 1, 2. 
\end{equation}
Let $a_m=\mathscr{F}^{-1}_{I(\delta_m/2+i\alpha)}(R^j_z\hat{s}_{m, z})$ for some $\alpha \in \R$ chosen as above. It follows from \autoref{l2.1} that $a_m \in L^2_{k-j+1, \delta_m}(\tilde{V}, \tilde{E}_{j-1})$. The assumption that $\delta_1 > \delta_2$ implies $L^2_{k-j+1, \delta_m} \subset L^2_{k-j+1, \tilde{\delta}}$, we conclude that $a_m \in L^2_{k-j+1, \tilde{\delta}}(\tilde{V}, \tilde{E}_{j-1})$. Applying the inverse Fourier-Laplace transform, we get 
\begin{equation}
s-\partial^{j-1} (a_1 + a_2) = \frac{1}{2\pi i } \left( \int_{I(\delta_1/2+i\alpha)} e^{-z\tilde{f}} R^j_z \hat{\zeta}_z dz - \int_{I(\delta_2/2+i\alpha)} e^{-z\tilde{f}} R^j_z \hat{\zeta}_z dz \right). 
\end{equation}
The argument above implies that $s-\partial^{j-1}(a_1+a_2) = \sum_{z_k} \varphi(\pmb{\alpha}(z_k))$, where the sum ranges over the poles of $R^j_z$ inside the rectangle $S=[\delta_2/2, \delta_1/2] \times i[\alpha, \alpha+2\pi]$. Thus we have proved each class $[s] \in H^j(E_{\tilde{\delta}}(\tilde{V}))$ can be represented by an element in $\im \varphi$. 

The last case is when $j=n$. We show that $H^n(E_{\tilde{\delta}}(\tilde{V})) = 0$. Again, let $s$ be a representative of $[s] \in H^n(E_{\tilde{\delta}}(\tilde{V}))$. Adopting the cut-off function $\eta$ defined above, we get $s= \eta s+ (1-\eta)s=: s_1 + s_2$. Since $\partial^n_z = 0$, the identity (\ref{r2.12}) becomes $\partial^{n-1}_z \circ R_z^n = \id$. Fix an arbitrary $\alpha \in \R$, and denote by $a_m = \mathscr{F}_{I(\delta_m/2+i \alpha)}^{-1} (R^n_z \hat{s}_{m, z})$ the inverse Fourier-Laplace transform. Then $\partial^{n-1}a_m = s_m$, $m=1, 2$. The same argument above implies that $a_m \in L^2_{k-n+1, \tilde{\delta}}(\tilde{V}, \tilde{E}_{n-1})$.  We then conclude that $\partial^{n-1}(a_1+a_2) = s$.

Finally we prove the injectivity. Before procceeding further we make the following observation. 
\begin{lem}\label{l2.13}
Let $s \in L^2_{k-j, \tilde{\delta}}(\tilde{V}, \tilde{E}_j)$ be a section. Then $(\tau^* - e^{-z})^m s = 0$ if and only if $s=\varphi(\pmb{\alpha})$ with $\pmb{\alpha} \in \widehat{C}^j_z(V)$ whose $u$-degree satisfies $m(\pmb{\alpha}) \leq m$. Moreover, $(\tau^*-e^{-z})^k \varphi(\pmb{\alpha}) \neq 0$ for all $k < m(\pmb{\alpha})$. 
\end{lem}

\begin{proof}[Proof of \autoref{l2.13}]
It's straightforward to verify that $(\tau^*-e^{-z})^{m(\pmb{\alpha})} \varphi(\pmb{\alpha}) = 0$ and $(\tau^*-e^{-z})^k \varphi(\pmb{\alpha}) \neq 0$ for all $k \leq m(\pmb{\alpha})$. We prove by induction on the minimal order $m$ that makes $(\tau^* - e^{-z})^m s=0$. When $m=1$, $(\tau^* - e^{-z}) s =0$ implies that $e^{z\tilde{f}} s $ is $\tau^*$-invariant, thus descends to a section $\alpha \in L^2_{k-j}(V, E_j)$. Then we conclude $s = \varphi(\alpha \cdot u^{-1})$ with $\alpha = e^{z \tilde{f}} s$. In general suppose $(\tau^*-e^{-z})^m s = 0$ for $m > 1$, and $(\tau^*-e^{-z})^k s = 0$ for all $k < m$. By induction, we can find $\pmb{\beta} \in \widehat{C}^j_z(V)$ so that $\varphi(\pmb{\beta}) = (\tau^* - e^{-z}) s$ and $m(\pmb{\beta}) = m-1$. 

We claim that one can find $\pmb{\gamma} \in \widehat{C}^j_z(V)$ with $m(\pmb{\gamma}) = m$ satisfying $(\tau^* - e^{-z}) \varphi(\pmb{\gamma}) = \varphi(\pmb{\beta})$. To see this, let's write $\pmb{\gamma} = \sum_{l=-m}^{-1} \gamma_l \cdot u^l$. Then it amounts to solve the equation
\[
\begin{split}
(\tau^* -e^{-z}) \varphi(\pmb{\gamma})  & = e^{-z \tilde{f}} \cdot e^{-z} \left(\sum_{l=-m}^{-1} (-1)^{-l-1} \left( \frac{(\tilde{f} + 1)^{-l-l} - (\tilde{f})^{-l-1}}{(-l-1)!} \right) \cdot \gamma_l \right) \\
& = e^{-z \tilde{f}} \sum_{k=-m+1}^{-1} (-1)^{-k-1} \frac{(\tilde{f})^{-k-1}}{(-k-1)!} \cdot \beta_k = \varphi(\pmb{\beta}).  
\end{split}
\]
Comparing the coefficient of $(\tilde{f})^n$, we get 
\[
\sum_{q=n+2}^m {q-1 \choose n} \frac{(-1)^{q-1} e^{-z}}{(q-1)!} \gamma_{-q} = \frac{(-1)^n}{n!} \beta_{-n-1}, \quad n=0, ..., m-2
\]
We let $\gamma_{-1} = 0$. The matrix relating $(\beta_{-m+1}, ..., \beta_{-1})$ to $(\gamma_{-m}, ..., \gamma_{-2})$ is then an upper diagonal matrix whose diagonal entries are all non-zero, thus invertible. So we can solve $\pmb{\gamma}$ in terms of $\pmb{\beta}$. This completes the proof of the claim. 

Now we have $(\tau^* - e^{-z})(s- \varphi(\pmb{\gamma})) = 0$. The induction process gives us $\pmb{\alpha} \in \widehat{C}^j_z$ with $\varphi(\pmb{\alpha}) = s - \varphi(\pmb{\gamma})$. Then we get $s = \varphi(\pmb{\alpha} + \pmb{\gamma})$ as desired. 
\end{proof}

Going back to the proof, we first consider the case when $j=0$. \autoref{l2.13} implies that $\varphi$ is actually injective on the cochain level, thus induces an injective map on the zero-th cohomology. Next we consider the general case when $1 \leq j \leq n$. Suppose $[\pmb{\alpha}] \in \widehat{H}^j_z(V)$ satisfies $\varphi(\pmb{\alpha}) = \partial^{j-1} a$ for some $a \in L^2_{k-j+1, \tilde{\delta}}(\tilde{V}, \tilde{E}_{j-1})$. Then we know 
\[
(\tau^*-e^{-z})^{m(\pmb{\alpha})} \varphi(\pmb{\alpha}) =\partial^{j-1} (\tau^*-e^{-z})^{m(\pmb{\alpha})} a = 0. 
\]
It follows from the surjectivity that $(\tau^*-e^{-z})^{m(\pmb{\alpha})}a$ is cohomologous to $\sum_{z_k} \varphi(\pmb{\beta}_k)$ for some $\pmb{\beta}_k \in \widehat{C}^{j-1}_{z_k}(V)$. Since $\tau^*$ commutes with $\partial$ and $(\tau^*-e^{-z})$ restricts as an isomorphism on the generalized $e^{-z_k}$-eigenspace of $H^*(E_{\tilde{\delta}}(\tilde{V})$ for $z_k \neq z$, we may assume $(\tau^*-e^{-z})^{m(\pmb{\alpha})}a = \varphi(\pmb{\beta})$ for some $\pmb{\beta} \in \widehat{C}^{j-1}_z(V)$. The proof of \autoref{l2.13} implies that one can find $\pmb{\gamma} \in \widehat{C}^{j-1}_z(V)$ so that $(\tau^*-e^{-z})^{m(\pmb{\alpha})} \varphi(\pmb{\gamma}) = \varphi(\pmb{\beta})$. We  conclude that 
\[
(\tau^* - e^{-z})^{m(\pmb{\alpha})} (a - \varphi(\pmb{\gamma})) = 0.
\]
Invoking \autoref{l2.13} again, we get $\pmb{\lambda} \in \widehat{C}^{j-1}_z(V)$ so that $a - \varphi(\pmb{\gamma}) = \varphi(\pmb{\lambda})$, equivalently $a = \varphi(\pmb{\gamma} + \pmb{\lambda})$. Now we have 
\[
\varphi(\pmb{\alpha}) = \partial^{j-1} \varphi(\pmb{\gamma} + \pmb{\lambda}) = \varphi(\widehat{\partial}^{j-1}_z (\pmb{\gamma} + \pmb{\lambda})).
\]
Since $\varphi$ is injective on the cochain level, we conclude that $\pmb{\alpha}$ is exact. This completes the proof. 
\end{proof}

Through the proof of \autoref{t2.12}, we see that elements in $\widehat{H}^j_z(V)$ span the generalized $e^{-z}$-eigenspace of $H^j(E_{\tilde{\delta}}(\tilde{V}))$ under the map $\varphi$. It also tells us the range of $z$'s that can be realized as the generalized $e^{-z}$-eigenspace of the covering transformation. This process determines the index of $E_{\tilde{\delta}}(\tilde{V})$. \autoref{t1.1} then follows from \autoref{l2.11} and the following corollary. 

\begin{cor}\label{c2.14}
Suppose $\delta_1 > \delta_2$. Regarding $H^j(E_{\tilde{\delta}}(\tilde{V}))$ as a $\C[t, t^{-1}]$-module with $t$ acting by the pullback of the covering transformation, we have an isomorphism of $\C[t, t^{-1}]$-modules
\[
H^j(E_{\tilde{\delta}}(\tilde{V})) \simeq \bigoplus_{z} \C[t, t^{-1}]/ (t - e^{-z})^{m_z},
\]
where $m_z = \dim_{\C} \widehat{H}^j_z(V)$, and $z$ ranges over $(\delta_2/2, \delta_1/2) \times i[0, 2\pi)$. In particular 
\[
\ind \left(E_{\tilde{\delta}}(\tilde{V}) \right) = - \sum_z \chi\left(\widehat{H}^*_z(V) \right), 
\]
where $z$ ranges over $(\delta_2/2, \delta_1/2) \times i[0, 2\pi)$. 
\end{cor} 

Although we have stated in a way that makes the sum infinite, actually only the spectral points $z \in \Sigma(E)$ could potentially contribute non-zero terms as we will note shortly. With this index jump formula, the problem of computing the index of $E_{\tilde{\delta}}(\tilde{V})$ has been reduced to understanding the cohomology $\widehat{H}^*_z(V)$ which is more practical with the help of the spectral sequence $\{\widehat{E}^{*,*}_{z,r}, d_r\}$ mentioned above. Conversely the knowledge of $E_{\tilde{\delta}}(\tilde{V})$ also tells us properties of $\widehat{H}^*_z(V)$. We end this section by summarizing some properties of $\widehat{H}^*_z(V)$ for later use.

\begin{prop}\label{p2.15}
Let $(E(V), 1_V)$ be a non-degenerate pair. Then the following hold.
\begin{enumerate}
\item $H^j(E_z(V)) = 0$ implies that $\widehat{H}^j_z(V) = 0$, $j = 0, ..., n$.
\item $\dim \widehat{H}^j_z(V) < \infty$, $j=0, ..., n$.
\item $\widehat{H}^n_z(V) = 0$
\item Suppose $\dim_{\C} H^j(E_z(V)) = 1$, $\bar{\sigma}^j: H^j(E_z(V))/\im \bar{\sigma}^{j-1} \to H^{j+1}(E_z(V))$ is injective, and $H^{j-1}(E_z(V))/ \im \bar{\sigma}^{j-2} = 0$. Then $\dim_{\C} \widehat{H}^j_z(V) = 1$.
\end{enumerate}
\end{prop}

\begin{proof}
To prove (1), we note that $H^j(E_z(V)) = 0$ implies that the terms $\widehat{E}^{i,j-i}_{z, 1}$ in the first page vanishes identically, thus the cohomology $\widehat{H}^j_z(V)$ vanishes as well. (2) and (3) are implied by \autoref{t2.12}. To see (4), the injectivity of $\bar{\sigma}^j$ implies that $\widehat{E}^{i, j-i}_{z, 2} = 0$ for all $i < -1$. Since $H^{j-1}(E_z(V)) / \im \bar{\sigma}^{j-2} = 0$, we conclude that $\widehat{E}^{-1, j+1}_{z, 2} = H^j(E_z(V)) \cdot u^{-1}$, which collapses at the second page and contributes as the only term that abuts to $\widehat{H}^j_z(V)$. 
\end{proof}

We believe there is a direct proof for (2) and (3) in \autoref{p2.15} without appealing to \autoref{t2.12}. This requires a further understanding between the non-degeneracy assumption of $(E(V), 1_V)$ and the differentials of higher pages of the spectral sequence $\{\widehat{E}^{*,*}_{z, r}, d_r\}$, which we choose not to pursuit here. 

In the end, we explain why \autoref{t1.1} generalizes the main result of \cite{MRS14}. In the case of DeRham complexes, Miller \cite{M06} proved that the cohomology $H^*(E_{\tilde{\delta}}(\tilde{V}))$ stabilizes when $\delta_1 = -\delta_2 > 0$ is sufficiently large. Under such a situation, $H^*(E_{\tilde{\delta}}(\tilde{V}))$ can be identified with the singular cohomology $H^*(\tilde{V}; \C)$, which by the universal coefficient theorem can be further identified with the singular homology $H_*(\tilde{V} ;\C)$. The main result Theorem 1.2 in \cite{MRS14} asserts that the index jump of $E_{\delta}(Z)$ about $\delta$ is given by the multiplicities of the roots $e^{-z}$, which satisfies $\Rea z = \delta/2$, of the Alexander polynomial $A_*(t)$ with respect to the covering transformation $\tau_*: H_*(\tilde{V}; \C) \to H_*(\tilde{V}; \C)$. It's straightforward to see that such a multiplicity corresponds to the dimension of the generalized $e^{-z}$-eigenspace of $H^*(E_{\tilde{\delta}}(\tilde{V}))$ under the identification above, which is identified with $\dim_{\C} \widehat{H}^j_z(V)$.

\section{\large \bf Periodic Rho Invariant of the ASD DeRham Operator}\label{s3}

In this section, we consider a more concrete situation when manifolds are $4$-dimensional and further develop the index theory of the anti-self-dual(ASD) DeRham operator $-d^* \oplus d^+$ over end-periodic $4$-manifolds based on the work of Mrowka-Ruberman-Saveliev \cite{MRS16}. The goal is to define a periodic version of the Atiyah-Patodi-Singer rho invariant \cite{APS2} for certain classes of $4$-manifolds. We shall write $\Lambda^*_{\C} M = \Lambda^* T^*M \otimes \C$ for the complexified bundle of alternating tensors over a smooth manifold $M$. When the underlying manifold is clear from the context, we simply write $\Lambda^*_{\C}$. 

\subsection{\em The ASD DeRham Operator}\label{ss3.1} \hfill
 
\vspace{3mm}

Let $V$ be a closed smooth $4$-manifold with a fixed primitive class $1_V \in H^1(V;\Z)$. In particular, we have $b_1(V) \geq 1$. Same as \autoref{s2}, we denote by $\tilde{V} = \cup_{n \in \Z} W_n$ the cyclic cover of $V$ corresponding to $1_V$ obtained by cutting a codimension-$1$ non-separating representative $Y$ of $\PD1_V$ and gluing up infinite copies with head to tail. Such a representative $Y$ is referred to as a cross-section of the pair $(V, 1_V)$. We denote by $\tau$ the generating covering transformation, and $\tilde{f}: \tilde{V} \to \R$ a fixed smooth function equivariant with respect to $\tau$, i.e. $\tilde{f}(\tau(x)) = \tilde{f}(x) +1$. We also form an end-periodic $4$-manifold $Z = M\cup \tilde{V}_+$ by attaching the positive end of $V$ to a $4$-manifold $M$ with boundary $\partial M = Y$, and extend $\tilde{f}$ to a smooth function $\rho : Z \to \R$. We fix an integer $k \geq 1$ to define the Sobolev space $L^2_k$. 

Given $z \in \C$, the $z$-twisted anti-self-dual DeRham complex over $V$:
\begin{equation}\tag{$E_z(V)$}
0 \longrightarrow L^2_{k+1}(V, \Lambda^0_{\C}) \xrightarrow{-d_z} L^2_k(V, \Lambda^1_{\C}) \xrightarrow{d^+_z} L^2_{k-1}(V, \Lambda^+_{\C}) \longrightarrow 0,
\end{equation}
where $d_z = d - z \cdot df$, $d^+_z = d^+ - z \cdot ( df \wedge - )^+$. When $z=0$, we simply write $E(V) = E_0(V)$. 

\begin{lem}[{\cite[Lemma 3.2]{T87}}]\label{l3.2}
The pair $(E(V), 1_V)$ corresponding to the ASD DeRham complex is non-degenerate if and only if 
\begin{enumerate}
\item $1 - b_1(V) + b^+(V) = 0$; 
\item $\ker (1_V \smile: H^1(V; \R) \to H^2(V; \R) ) = \langle 1_V \rangle$. 
\end{enumerate}
\end{lem}

\begin{proof}
The symbol sequence takes the following form 
\[
0 \longrightarrow H^0(V; \C) \xrightarrow{1_V \smile} H^1(V;\C) \xrightarrow{(1_V \smile - )^+} H^+(V;\C) \longrightarrow 0. 
\]
It's clear that (1) and (2) are necessary conditions for the exactness of the symbol sequence. Taubes' argument \cite{T87} shows that the second condition (2) implies this sequence is exact at $H^0(V;\C)$ and $H^1(V;\C)$. The dimension condition (1) then implies that the sequence is also exact at $H^+(V;\C)$. 
\end{proof}

\begin{rem}
\autoref{l3.2} shows that the rational cohomology of a non-degenerate pair $(V, 1_V)$ is equivalent to that of $(S^1 \times Y, d\theta)$, where $Y$ is a closed $3$-manifold and $d\theta \in H^1(S^1; \Z)$ is the angular form. 
\end{rem}

Throughout this section, we always assume $(V, 1_V)$ is a non-degenerate pair. The $z$-twisted ASD DeRham operator corresponding to $E_z$ is given by
\begin{equation}\label{e3.1}
Q_z:= -d^*_z \oplus d_z^+: L^2_k(V, \Lambda^1_{\C})  \longrightarrow L^2_{k-1}(V, \Lambda^0_{\C} \oplus \Lambda^+_{\C}), 
\end{equation}
where $d^*_z  = d^* - \bar{z} \cdot \iota_{df}$ is the $L^2$-adjoint of $d_z$. Here $\iota_{df}$ is the normalized contraction by $df$, namely for non-zero $a \in \Lambda^*_{\C}$ it takes the form 
\begin{equation}
\iota_{df} df \wedge a = \frac{|df \wedge a|^2}{|a|^2} \cdot a.  
\end{equation}
Then one readily checks that $(df \wedge a) \wedge \star \bar{b} = a \wedge \star \iota_{df} \bar{b}$, which means $\iota_{df}$ is the adjoint of $df \wedge -$. Then we have 
\[
\dim \ker Q_z = \dim H^1(E_z), \quad \dim \cok Q_z = \dim H^0(E_z) + \dim H^2(E_z),
\]
which tells us that $z$ is a regular point of $E_z$ if and only if $Q_z$ is invertible. When $z=0$, we recover the usual ASD DeRham operator over $V$, which we denote by $Q:=-d^* \oplus d$. 

Now we switch to consider the ASD DeRham complex over the end-periodic manifold $Z$ with weight $\delta \in \R$:
\begin{equation}\tag{$E_{\delta}(Z)$}\label{eZ}
0 \longrightarrow L^2_{k+1, \delta}(Z, \Lambda^0_{\C}) \xrightarrow{-d} L^2_{k, \delta}(Z, \Lambda^1_{\C}) \xrightarrow{d^+} L^2_{k-1, \delta}(Z, \Lambda^+_{\C}) \longrightarrow 0. 
\end{equation}
Since $z=0$ is a spectral point of $E_z$, we need to work with weights $\delta \neq 0$ to get a Fredholm complex $E_{\delta}(Z)$ with well-defined index. The corresponding ASD DeRham operator of $E_{\delta}(Z)$  takes the form 
\begin{equation}
-d^*_{\delta} \oplus d^+: L^2_{k, \delta}(Z, \Lambda^1_{\C}) \longrightarrow L^2_{k-1, \delta}(Z, \Lambda^0_{\C} \oplus \Lambda^+_{\C}),
\end{equation}
where $d^*_{\delta}:=e^{-\delta \rho} d^* e^{\delta \rho} = d^* - \delta \cdot \iota_{d\rho}$ is the $L^2_{\delta}$-adjoint of $d$. 
After identifying $L^2_k = e^{\delta\rho/2}L^2_{k, \delta}$, this operator is equivalent to the following operator defined on Sobolev spaces with zero weight:
\begin{equation}
Q_{\delta}(Z):= -d^*_{\delta/2} \oplus d^+_{\delta/2}: L^2_{k}(Z, \Lambda^1_{\C}) \longrightarrow L^2_{k-1}(Z, \Lambda^0_{\C} \oplus \Lambda^+_{\C}),
\end{equation}
where $d^+_{\delta/2} = d^+ - \delta/2(d\rho \wedge -)^+$. 

\subsection{\em The Index Formula}\label{ss3.2} \hfill
 
\vspace{3mm}

In this subsection, we derive an Atiyah-Patodi-Singer type index formula for the ASD DeRham operator $Q_{\delta}(Z)$ following the strategy of \cite{MRS16}. Note that the $L^2$-adjoint of $Q_{\delta}(Z)$ is given by $Q^*_{\delta}(Z) = -d_{\delta/2} \oplus d^*_{\delta/2}$. Conjugating by the Fourier-Laplace transform, we get two $z$-twisted operators on $V$:
\begin{equation}\label{e3.5}
\begin{split}
Q_{\delta, z}(V) & = \left(-d^* + (\delta/2 - z)\iota_{df}, \; d^+- (\delta/2 + z)(df \wedge -)^+ \right)  \\
Q^*_{\delta, z}(V) & = \left(-d + (\delta/2 + z) df, \; d^* - (\delta/2 - z) \iota_{df} \right)
\end{split}
\end{equation}

Due to \autoref{l2.3} and \autoref{l2.5}, we may choose $\delta_0 > 0$ so that $Q_{\delta}(Z)$ is Fredholm for all non-zero $\delta \in (-\delta_0, \delta_0)$. In other words, $\delta_0$ is chosen so that 
\begin{equation}\label{e3.6}
E_z(V) \text{ is acyclic for all } \Rea z \in (-\delta_0/2, \delta_0/2)\backslash \{0\}. 
\end{equation}
Since $Q$ is an elliptic first order differential operator, it follows from the Atiyah-Singer index theory that $Q$ admits an index form which we denote by $\mathfrak{a}(Q)$. Over the closed $4$-manifold $V$, we have 
\[
\ind Q(V) = \int_V \mathfrak{a}(Q) = -\frac{1}{2}(\chi(V) + \sigma(V)) = 0.
\]
In particular, the index form is exact on $V$, say $d \omega=\mathfrak{a}(Q)$. Since $\ind Q_{\delta}(Z)$ is independent of the choice of $\delta \in (0, \delta_0)$, we denote it by $\ind_+ Q(Z)$. Similarly, we write $\ind_-Q(Z)$ for $\ind Q_{\delta}(Z)$ with $\delta \in (-\delta_0, 0)$. 

In the work of Atiyah-Patodi-Singer \cite{APS1}, the contribution to $\ind_{\pm} Q(Z)$ of the non-compact end comes from certain spectral asymmetry, which is referred to as the eta invariant. In our case, it's given by the periodic eta invariant. 

\begin{dfn}\label{d3.3}
Let $\delta \in (-\delta_0, \delta_0) \backslash \{0\}$ with $\delta_0 > 0$ satisfying (\ref{e3.6}). The periodic eta invariant of the ASD DeRham operator associated to a non-degenerate pair $(V, 1_V)$ is defined to be 
\begin{equation}
\tilde{\eta}_{\delta}(V) = \frac{1}{\pi i} \int_{I(\delta/2)} \Tr \left(df \cdot Q_z \exp(-tQ_z^*Q_z)\right) dz dt, 
\end{equation}
where $df \cdot$ is applied to $L^2_{k-1}(V, \Lambda^0_{\C} \oplus \Lambda^+_{\C})$ by $df \cdot (s, w) = (s df,  \iota_{df} w)$, and $I(\delta/2) = \delta/2 + i(0, 2\pi)$. 
\end{dfn}
We will see shortly in the proof of \autoref{t3.4} that $\tilde{\eta}_{\delta}(V)$ is independent of the choice of the sign of $\delta \in (-\delta_0, \delta_0) \backslash \{0\}$. Thus we will simply denote by $\tilde{\eta}_{\pm}(V)$ the corresponding periodic eta invariants given by positive and negative $\delta$ respectively. Finally we take the mean of the eta invariants and denote it by 
\begin{equation}
\tilde{\eta}(V):= \frac{\tilde{\eta}_+(V) + \tilde{\eta}_-(V)}{2}. 
\end{equation}
One can interpret the periodic eta invariants as the spectral asymmetry of the complex $E_z$. Explicitly, it's the regularization of the difference between the number of spectral points in $\Sigma(E_z)$ of length greater and less than $1$. See \cite[Section 6]{MRS16} for more details of this interpretation in the case of Dirac type operators. 

\begin{thm}\label{t3.4}
With the notations above, we have 
\[
\ind_{\pm} Q(Z) = \int_M \mathfrak{a}(Q) - \int_Y \omega + \int_V df \wedge \omega- \frac{\tilde{\eta}_{\pm}(V)}{2}.
\]
\end{thm}

\begin{proof}
The proof is literally the same as that of \cite[Theorem A]{MRS16} except for some minor adjustment. Their Dirac operators $\mathcal{D}^+$ and $\mathcal{D}^-$ are replaced by $Q_{\delta}$ and $Q^*_{\delta}$ respectively in our case. The corresponding symmetric trace in \cite[Section 5]{MRS16} takes the form 
\[
\Str^{\flat}(\delta, t) = \Tr^{\flat} \exp(-tQ^*_{\delta}Q_{\delta}) - \Tr^{\flat}\exp(-tQ_{\delta}Q^*_{\delta}),
\]
where $\Tr^{\flat}$ means the regularized trace of an elliptic operator over the end-periodic manifold $Z$ defined in \cite[Section 3]{MRS16}. Since the differential operator $Q_{\delta}^*Q_{\delta}$ is self-adjoint and elliptic, the estimates of smoothing kernels in \cite{MRS16} can be carried over to $Q_{\delta}^*Q_{\delta}$ as well, which leads to the following limits:
\[
\lim_{t \to \infty} \Str^{\flat}(\delta, t)  = \ind Q_{\delta}(Z), \quad \lim_{t \to 0} \Str^{\flat}(\delta, t)  = \int_M \mathfrak{a}(Q). 
\]
In the meantime, the formula \cite[(23)]{MRS16} corresponding to $Q_{\delta}$ is 
\begin{equation}
\begin{split}
\frac{d}{dt} \Str^{\flat}(\delta, t) & = \frac{1}{2\pi i } \int_{I(0)}\int_{W_0} f \cdot \tr \left(K_{\exp(-tQ^*_{\delta, z}Q_{\delta, z})} - K_{\exp(-tQ_{\delta, z}Q^*_{\delta, z})}\right) dx dz\\
& - \frac{1}{2\pi i }\int_{I(0)} \Tr\left( \frac{\partial Q^*_{\delta, z}}{\partial z} \cdot Q_{\delta, z}\exp(-tQ^*_{\delta, z}Q_{\delta, z}) \right) dz. 
\end{split}
\end{equation}
We note that $\partial Q^*_{\delta, z} /\partial z = 1/z (df, \iota_{df})$ and when $\Rea z = 0$, we have 
\[
Q_{\delta, z} = Q_{\delta/2+ z} \text{ and } Q^*_{\delta, z} = Q^*_{\delta/2+z}.
\]
Thus $\tilde{\eta}_{\delta}(V)/2$ can be identified with the integration of the second term of (\ref{e2.12}). Then the rest of the proof follows from that in \cite[Section 5.2]{MRS16}. 
\end{proof}

As we vary $\delta \in (-\delta_0, \delta_0)$, the only value that makes $Q_{\delta}(Z)$ fail to be Fredholm is $0$, which is caused by the existence of spectral points in $\Sigma(E)$ with vanishing real part. Let's denote by 
\begin{equation}
\tilde{h}(V):= \sum_{z \in I(0)} -\chi\left(\widehat{H}^*_z(V) \right)
\end{equation}
the total multiplicity of spectral points with vanishing real part. Then \autoref{t2.12} implies that 
\begin{equation}
\ind_+ Q(Z) - \ind_- Q(Z) = - \tilde{h}(V). 
\end{equation}
In this way, we obtain the a variant of \cite[Theorem C]{MRS16} corresponding to the ASD DeRham operator. 
\begin{cor}\label{c3.5}
With the notations of \autoref{t3.4}, we have 
\[
\ind_+ Q(Z) = \int_M \mathfrak{a}(Q) - \int_Y \omega + \int_V df \wedge \omega- \frac{\tilde{h}(V)+\tilde{\eta}(V)}{2}.
\]
\end{cor}

\subsection{\em The Periodic Rho Invariant}\label{ss3.3} \hfill
 
\vspace{3mm}

Let $\varphi: \pi_1(V) \to U(1)$ be a $U(1)$-representation of the fundamental group. The covariant derivative of the flat connection on the trivial bundle corresponding to $\varphi$ is denoted by $d_{\varphi}$. Then for $z \in \C$, we have the twisted ASD DeRham complex:
\begin{equation}\tag{$E_{\varphi, z}$}\label{eaX}
0 \longrightarrow L^2_{k+1}(V, \Lambda^0_{\C}) \xrightarrow{-d_{\varphi, z}} L^2_k(V, \Lambda^1_{\C}) \xrightarrow{d^+_{\varphi, z}} L^2_{k-1}(V, \Lambda^+_{\C}) \longrightarrow 0
\end{equation}
defined as before. 

\begin{dfn}\label{d3.6}
A representation $\varphi: \pi_1(V) \to U(1)$ is admissible if the following hold. 
\begin{enumerate}[label=(\alph*)]
\item For all $z$ with $\Rea z=0$, $H^*(E_{\varphi, z}) = 0$. 
\item There exists a cross-section $Y$ of the pair $(V, 1_V)$ such that 
\[
H^0(Y; \C_{\varphi}) = H^1(Y; \C_{\varphi}) = 0,
\]
where $\C_{\varphi}$ is the twisted coefficient given by the restriction of $\varphi$ to $\pi_1(Y)$. 
\end{enumerate}
\end{dfn}

The admissibility condition should be regarded as an adaptation of the admissibility from Taubes' paper \cite[Definition 1.3]{T87}.  The first condition (a) will be seen to coincide with the non-degeneracy condition \autoref{d5.2} for the moduli space of singular instantons. Meanwhile, (a) also implies that the spectral set $\Sigma(E_{\varphi, z})$ of the complex $E_{\varphi, z}$ is non-empty and discrete. The second condition (b) implies that the rho invariant $\rho_{\varphi}(Y)$ is not a jumping point with respect to varying the representation $\varphi$. 

Let $Q_{\varphi}:= -d^*_{\varphi} \oplus d^+_{\varphi}$ be the ASD DeRham operator twisted by $\varphi$. Due to (a) in the admissibility assumption, we can define the corresponding eta invariant same as \autoref{d3.3}, which we denote by $\tilde{\eta}_{\varphi}(V)$. The periodic rho invariant is considered to be a reduced periodic eta invariant. 

\begin{dfn}
Let $\varphi: \pi_1(V) \to U(1)$ be an admissible representation. We define the periodic rho invariant to be 
\begin{equation}
\tilde{\rho}_{\varphi}(V)= \tilde{\eta}_{\varphi}(V) - \tilde{\eta}(V). 
\end{equation}
\end{dfn}

\begin{lem}
The periodic rho invariant $\tilde{\rho}_{\varphi}(V)$ is independent of the choice of the smooth function $f: V \to S^1$ and the Riemannian metric $g$ on $V$. 
\end{lem}

\begin{proof}
The independence of $f$ follows from the argument of \cite[Lemma 8.2]{MRS16} directly. To see the independence on the metric, we adopt the argument in \cite[Lemma 8.3]{MRS16}. Let $g_0$ and $g_1$ be two metrics on $V$. We choose a metric $\tilde{g}$ and a function $\tilde{f}'$ on $\tilde{V} = \bigcup_{n \in \Z}W_n$ so that 
\[
\tilde{g}|_{W_n}=
\begin{cases}
g_0 & \text{ if $n < 0$ } \\
g_1 & \text{ if $n > 0$}
\end{cases}, \quad
\tilde{f}'|_{W_n}=
\begin{cases}
-\tilde{f} & \text{ if $n < 0$ } \\
\tilde{f} & \text{ if $n > 0$}
\end{cases}. 
\]
We regard $(\tilde{V}, \tilde{g}, \tilde{f}')$ as an end-periodic manifold with two periodic ends attached to $(W_0, \tilde{g}|_{W_0})$. Then \autoref{t3.4} implies that 
\[
\left( \ind_+ Q_{\varphi}(\tilde{V}, \tilde{g}) + \ind_- Q_{\varphi}(\tilde{V}, \tilde{g})\right)  - \left( \ind_+ Q(\tilde{V}, \tilde{g}) + \ind_- Q(\tilde{V}, \tilde{g})\right) = \tilde{\rho}_{\varphi}(V, g_0) - \tilde{\rho}_{\varphi}(V, g_1). 
\]
So it suffices to show that $\ind_{\pm} Q_{\varphi}(\tilde{V}, \tilde{g})$ is independent of the metric with respect to any admissible $\varphi: \pi_1(V) \to U(1)$.

Let $(g_t)_{t \in [0,1]}$ be a path of metrics connecting $g_0$ and $g_1$. This path can be lifted, after appropriate modification on $W_0$, to a path of metrics $\tilde{g}_t$ on $\tilde{V}$ from $\tilde{g}_0$ to $\tilde{g}$. Once we prove that one can find $\delta > 0$ such that $Q_{\varphi, \delta}(\tilde{V}, \tilde{g}_t)$ is a Fredholm operator on $L^2_k$ for all $t \in [0, 1]$. The conclusion then follows from the invariance of index under continuous deformation of Fredholm operators.

Taubes' criterion tells us that $Q_{\varphi, \delta}(\tilde{V}, \tilde{g}_t)$ is Fredholm if and only if $H^1(E_{\varphi, z}(g_t)) = 0$ for all $\Rea z =\delta/2$. Since $d^2_{\varphi, z} = 0$, $d_{\varphi, z}$ can be regarded as the covariant derivative given by a flat connection (not necessarily unitary) on the trivial bundle $\underline{\C}$ of $V$.  The Chern-Weil homomorphism implies that $d_{\varphi, z}a = 0$ once $d^+_{\varphi, z} a =0$. Thus $H^1(E_{\varphi, z}(g_t)) \simeq H^1(V; \C_{\varphi_z})$ with $\varphi_z: \pi_1(V) \to \C^*$ the representation given by the holonomy of the flat connection $d_{\varphi, z}$, which is clearly independent of the metric. 
\end{proof}


\subsection{\em Relation with Rho Invariants over $3$-Manifolds}\label{ss3.4} \hfill
 
\vspace{3mm}

We start by recalling the definition of the rho invariant defined over $3$-manifolds \cite{APS2}. Let $Y$ be a closed $3$-manifold and $\varphi: \pi_1(Y) \to U(1)$ a unitary representation. For simplicity, we further assume that $\varphi$ is trivial on the torsion part of $\pi_1(Y)$. Then $\varphi$ corresponds to a flat connection on the trivial bundle $\underline{\C}$ whose covariant derivative is denoted by $d_{\varphi}$. Then we have the even signature operator \cite[(4.6)]{APS1}:
\begin{equation}\label{e3.13}
\begin{split}
L'_{\varphi}: \bigoplus_{p=0, 1}\Omega^{2p}_{\C}(Y) & \longrightarrow \bigoplus_{p=0, 1} \Omega^{2p}_{\C}(Y) \\
\omega & \longmapsto (-1)^p(\star d_{\varphi} - d_{\varphi}\star) \omega
\end{split}
\end{equation}
Since $L'_{\varphi}$ is a self-adjoint elliptic operator, its spectrum is real and discrete. The regularization of the spectral asymmetry $\sum_{ \lambda \neq 0} \sign (\lambda)$ is defined as the eta invariant $\eta_{\varphi}(Y)$ in \cite{APS1}. The rho invariant is then defined as the reduced eta invariant:
\begin{equation}
\rho_{\varphi}(Y):= \eta_{\varphi}(Y) - \eta(Y). 
\end{equation}

We first consider the product case $V = S^1 \times Y$ and establish the equivalence of the periodic rho invariant with the usual rho invariant by comparing the indices of the ASD DeRham operator over end-periodic and end-cylindrical manifolds. Let $Z = M \cup [0, \infty) \times Y$ be a non-compact manifold with a cylindrical end modeled on $[0, \infty) \times Y$. We put a Riemannian metric $g$ on $Z$ so that $g|_{[0, \infty) \times Y} = dt^2 + g_Y$, where $g_Y$ is the Riemannian metric on $Y$. We identify a collar neighborhood of $Y$ in $M$ by $(-2, 0] \times Y$. Then we can choose a smooth function $\rho: Z \to \R$ of the form $\rho|_{M \backslash (-\epsilon, 0] \times Y} = 0$ and $\rho|_{[0, \infty) \times Y} = t$. Suppose the unitary representation $\varphi: Y \to U(1)$ extends to $\tilde{\varphi}: \pi_1(Z) \to U(1)$. We can then consider the ASD DeRham operator $Q_{\tilde{\varphi}, \delta}(Z): L^2_k \to L^2_{k-1}$ with $\delta$ chosen to ensure the Fredholmness. Let $I \subset [0, \infty)$ be a sub-interval. Recall there is a canonical way to identify forms on $I \times Y$ with time-dependent forms on $Y$:
\begin{equation}\label{e3.15}
\begin{split}
\Omega^0_{\C}(I \times Y) & = C^{\infty}(I, \Omega^0_{\C}(Y)) \text{ via } s \mapsto s|_{\{t \} \times Y} \\
\Omega^1_{\C}(I \times Y) & = C^{\infty}(I, \Omega^0_{\C}(Y) \oplus \Omega^1_{\C}(Y)) \text{ via } a = dt \otimes s(t) + b(t) \mapsto (s(t), b(t)) \\
\Omega^+_{\C}(I \times Y) & =C^{\infty}(I, \Omega^1_{\C}(Y)) \text{ via } \omega = dt \wedge b(t) + \star b(t) \mapsto 2b(t)
\end{split}
\end{equation}
With respect to this identification, the ASD operator $Q_{\tilde{\varphi}, \delta}|_{I \times Y}$ takes the form $Q_{\tilde{\varphi}, \delta} = d/dt + L_{\varphi, \delta}$, where $L_{\varphi, \delta}$ is the self-adjoint operator
\begin{equation}\label{e3.16}
L_{\varphi, \delta} = 
\begin{pmatrix}
\delta/2 & -d_{\varphi}^*  \\
-d_{\varphi}  & \star d_{\varphi} - \delta/2
\end{pmatrix}:
\Omega^0(Y) \oplus \Omega^1_{\C}(Y) \longrightarrow \Omega^0_{\C}(Y) \oplus \Omega^1_{\C}(Y).
\end{equation}
Since the spectrum of $L_{\varphi}$ is real and discrete, we see that for $\delta > 0$ sufficiently small $Q_{\tilde{\varphi}, \delta}(Z): L^2_k \to L^2_{k-1}$ is Fredholm. When $\delta = 0$, the operator $L_{\varphi}$ is equivalent to the even signature operator $L'_{\varphi}$ (\ref{e3.13}) under the identification $\Omega^1_{\C}(Y) \simeq \Omega^2_{\C}(Y)$ by $b \mapsto \star b$. 

\begin{prop}\label{p3.9}
Under the situation above, we write $h^i_{\varphi}(Y) = \dim H^i(Y, \C_{\varphi})$. Then we have 
\[
\ind_{\pm} Q_{\tilde{\varphi}}(Z) = -\frac{1}{2}( \chi(M) + \sigma(M)) - \frac{\eta_{\varphi}(Y)}{2} \pm \frac{h^0_{\varphi}(Y) - h^1_{\varphi}(Y)}{2},
\]
where $\chi(M)$ and $\sigma(M)$ are the Euler characteristic and signature respectively. 
\end{prop}

\begin{proof}
The argument follows the proof of \cite[Proposition 8.4.1]{MMR94} closely. Let's denote by $\Pi^{\pm}$ the $L^2$-projection of  $\Omega^0_{\C}(Y) \oplus \Omega^1_{\C}(Y)$ to the positive/negative eigenspaces of the operator $L_{\varphi}$. Let $C^{\infty}(M, \Lambda^1_{\C}M; 1 - \Pi^-)$ be the space of $C^{\infty}$ $1$-forms $a$ on $M$ so that $(1- \Pi^-) a|_{\partial M} = 0$, where $a|_{\partial M}$ means the restriction of $a$ to $\partial M$ as a path in $\Omega^0_{\C}(Y) \oplus \Omega^1_{\C}(Y)$ under the identification of (\ref{e3.13}). Then \cite[Theorem (3.10)]{APS1} tells us the index of the ASD DeRham operator 
\[
Q_{\tilde{\varphi}, \Pi}(M):= -d^*_{\tilde{\varphi}} \oplus d^+_{\tilde{\varphi}}: C^{\infty}(M, \Lambda^1_{\C}M; 1 - \Pi^-) \to  C^{\infty}(M, \Lambda^0_{\C}M \oplus \Lambda^+_{\C}M)
\]
is given by
\begin{equation}
\ind Q_{\tilde{\varphi}, \Pi}(M) =  -\frac{1}{2}( \chi(M) + \sigma(M)) - \frac{1}{2} (h^1_{\varphi}(Y) + h^0_{\varphi}(Y) + \eta_{\varphi}(Y))
\end{equation}
Let's pick $\delta > 0$ sufficiently small. The result follows once we show that 
\[
\ind Q_{\tilde{\varphi}, \delta}(Z) = \ind Q_{\tilde{\varphi}, \Pi}(M) + h^0_{\varphi}(Y) \qquad \ind Q_{\tilde{\varphi}, -\delta}(Z) = \ind Q_{\tilde{\varphi}, \Pi}(M) + h^1_{\varphi}(Y) 
\]

The index comparison follows from an excision argument. We start with the positive case. Let's denote a closed collar neighborhood of $Y$ in $M$ by $N = [-2, 0] \times Y$, and write $M' = M \backslash (-2, 0] \times Y$. Consider the operator 
\[
Q_{\tilde{\varphi}, \Pi}(N): C^{\infty}(N, \Lambda^1_{\C}N; \Pi^- \cup 1-\Pi^-) \longrightarrow C^{\infty}(N, \Lambda^0_{\C}N \oplus \Lambda^+_{\C}N),
\]
where $C^{\infty}(N, \Lambda^1_{\C}N; \Pi^- \cup 1-\Pi^-)$ is the space of smooth $1$-forms on $N$ so that $\Pi^- a|_{\{-2\} \times Y} = 0$ and $(1-\Pi^-) a|_{\{0\} \times Y} = 0$. Since $d\rho|_{M} = 0$ near $\{-1\} \times Y$, we have $Q_{\tilde{\varphi}, \delta}(Z)|_{\{-1\} \times Y} = Q_{\tilde{\varphi}, \Pi}(N)|_{\{-1\} \times Y}$. Applying excision \cite[Chapter 7]{DK90} to the pair $(Z, N)$, we get $M$ and $N' = [-2, \infty) \times Y$ with corresponding operators $Q_{\tilde{\varphi}, \Pi}(M)$ and 
\[
Q_{\tilde{\varphi}, \Pi, \delta}(N'): L^2_k(N', \Lambda^1_{\C}N'; \Pi^-) \longrightarrow L^2_{k-1}(N', \Lambda^0_{\C}N' \oplus \Lambda^+_{\C}N'),
\]
where $L^2_k(N', \Lambda^1_{\C}N'; \Pi^-) = \{ a \in L^2_k(\Lambda^1_{\C}N'): \Pi^- a|_{\{-2\} \times Y} = 0\}$. The index of these operators are related by 
\begin{equation}
\ind Q_{\tilde{\varphi}, \delta}(Z) + \ind Q_{\tilde{\varphi}, \Pi}(N) = \ind Q_{\tilde{\varphi}, \Pi, \delta}(N') + \ind Q_{\tilde{\varphi}, \Pi}(M).  
\end{equation}
After identifying the domain and range of $Q_{\tilde{\varphi}, \Pi}(N)$ via (\ref{e3.13}) and dualizing boundary conditions, we see that $Q_{\tilde{\varphi}, \Pi}(N)$ is self-adjoint. Thus $\ind Q_{\tilde{\varphi}, \Pi}(N) = 0$. 

Using the interpretation \cite[Corollary (3.14)]{APS1}, $\ker Q_{\tilde{\varphi}, \Pi, \delta}(N')$ consists of $1$-forms $a$ on $\R \times Y$ that are asymptotically constant as $t \to -\infty$, decay in $L^2_k([0, \infty) \times Y)$, and satisfy $Q_{\tilde{\varphi}, \delta} a = 0$. This equation can be solved explicitly as follows. Let's write $\Spec(L_{\varphi}) = \{ \lambda_i\}_{i=-\infty}^{\infty}$ with $\lambda_0 = 0$, and $a_i = dt \otimes s_i(t) + b_i(t)$ be the $\lambda_i$-eigenforms of $L_{\varphi}$. Then we need to solve
\begin{equation}\label{e3.18}
\begin{split}
\dot{s}_i + (\delta\rho'(t)/2 + \lambda_i) s_i & = 0 \\
\dot{b}_i - ( \delta \rho'(t)/2 - \lambda_i) b_i & = 0
\end{split}
\end{equation}
The solution to the second equation is $b_i(t) = \mathfrak{c}_i \exp( \delta \rho(t)/2 - \lambda_i t)$. Since $\delta/2 \in (0, \lambda_1)$, the decay conditions force $\mathfrak{c}_i = 0$ and $\lambda_i = 0$. Thus each $s \in H^0(Y, \C_{\varphi})$ gives rise to a unique solution $s_i(t)$ with $s_i(0) = s$. This shows that $\dim \ker Q_{\tilde{\varphi}, \Pi, \delta}(N') = h^0_{\varphi}(Y)$. 

To identify $\cok Q_{\tilde{\varphi}, \Pi, \delta}(N')$ we consider the adjoint  operator 
\[
Q^*_{\tilde{\varphi}, \Pi, \delta}(N'): L^2_k(N', \Lambda^0_{\C}N' \oplus \Lambda^+_{\C}N'; 1- \Pi^-) \longrightarrow L^2_{k-1}(N', \Lambda^1_{\C}N'). 
\]
Then $\ker Q^*_{\tilde{\varphi}, \Pi, \delta}(N') \simeq \cok Q_{\tilde{\varphi}, \Pi, \delta}(N')$ and elements in $\ker Q^*_{\tilde{\varphi}, \Pi, \delta}(N')$ are represented by sections $(s, \omega)$ that decay in $L^2_k$ on both ends and satisfy $Q^*_{\tilde{\varphi}, \delta}(s, \omega) = 0$. We note that $Q^*_{\tilde{\varphi}, \delta}(s, \omega) = -d/dt + L_{\varphi, \delta}$. Writing $\omega = 1/2(dt \wedge b(t) + \star b(t))$ and decomposing $s(t)$ and $b(t)$ into eigenforms $(s_i, b_i)$ of $L_{\varphi, \delta}$, the corresponding equation to solve is 
\begin{equation}
\begin{split}
\dot{s}_i - (\delta \rho'(t)/2 + \lambda_i) s_i & = 0 \\
\dot{b}_i + ( \delta \rho'(t)/2 - \lambda_i) b_i & = 0
\end{split}
\end{equation}
Due to the $L^2_k$-decay assumption on both ends, there are no solutions to this equation. Thus $\dim \cok Q_{\tilde{\varphi}, \Pi, \delta}(N') = 0$, which completes the proof for the positive part. 

For the negative part, the argument is the same except that the equation (\ref{e3.18}) now takes the form 
 \begin{equation}\label{e3.20}
\begin{split}
\dot{s}_i - (\delta \rho'(t)/2 - \lambda_i) s_i & = 0 \\
\dot{b}_i + ( \delta \rho'(t) /2+ \lambda_i) b_i & = 0
\end{split}
\end{equation}
Due to the decay constraints on both ends, solutions takes the form $s_i = 0$, $\lambda_i = 0$. Then we conclude that $\dim \ker Q_{\tilde{\varphi}, \Pi, -\delta}(N') = h^1_{\varphi}(Y)$, which implies the desired result. 
\end{proof}

\autoref{t3.4} and \autoref{p3.9} allow us to establish the equivalence of the periodic rho invariant of the ASD DeRham operator with the usual rho invariant of the signature operator in \cite{APS2}. 
\begin{cor}
Let $V= S^1 \times Y$ be the product of $S^1$ with a closed $3$-manifold $Y$, and $\varphi: \pi_1(V) \to U(1)$ an admissible representation. Then 
\[
\tilde{\rho}_{\varphi}(V) = \rho_{\varphi}(Y). 
\]
\end{cor}

\begin{proof}
When $V = S^1 \times Y$, we have $\int_Y \omega = \int_V df \wedge \omega$. The admissibility of $\varphi$ implies that $h^0_{\varphi}(Y) = h^1_{\varphi}(Y) = 0$. We conclude that $\tilde{\eta}_{\varphi}(V) = \eta_{\varphi}(Y)$. On the other hand, and \autoref{t3.4} and \autoref{p3.9} tell us 
\[
\int_M \mathfrak{a}(Q) - \tilde{\eta}(V) = \ind Q_+(Z) + \ind Q_-(Z) =-\frac{1}{2}(\chi(M) + \sigma(M)) -\eta(Y).
\]
Since $\int_M \mathfrak{a}(Q) = -1/2(\chi(M) + \sigma(M))$, the corresponding rho invariants are equal. 
\end{proof}

The same idea can be applied to the general case. Let $(V, 1_V)$ be a non-degenerate pair, $\varphi: \pi_1(V) \to U(1)$ an admissible representation, and $Y$ a cross-section of $(V, 1_V)$ satisfying $H^0(Y;\C_{\varphi}) = H^1(Y; \C_{\varphi}) = 0$. We wish to compare $\tilde{\rho}_{\varphi}(V)$ with $\rho_{\varphi}(Y)$, which will be achieved with the help of an index computation from \cite{T87}. 

Let $Z = M \cup \tilde{V}_+$ be an end-periodic manifold as before. Since $\pi_1(\tilde{V}) \to \pi_1(V)$ is injective, we get a unitary representation $\tilde{\varphi}: \pi_1(\tilde{V}) \to U(1)$ by restriction. Suppose we can extend this representation to one over $Z$, which we still denote by $\tilde{\varphi}: \pi_1(Z) \to U(1)$. We further assume that $\tilde{\varphi}$ is trivial on the torsion part of $\pi_1(Z)$. Then $\tilde{\varphi}$ gives rise to a flat connection $d_{\tilde{\varphi}}$ on the trivial bundle $\underline{\C}$. We can consider the complex 
\begin{equation}\tag{$E_{\tilde{\varphi}}(Z)$}\label{eZ1}
0 \longrightarrow L^2_{k+1}(Z, \Lambda^0_{\C}) \xrightarrow{-d_{\tilde{\varphi}}} L^2_{k}(Z, \Lambda^1_{\C}) \xrightarrow{d^+_{\tilde{\varphi}}} L^2_{k-1}(Z, \Lambda^+_{\C}) \longrightarrow 0.
\end{equation}
The following identification of the cohomology is an adaptation of the corresponding result in Taubes' paper {\cite[Proposition 5.1]{T87}}. 

\begin{prop}\label{p3.11}
Let $\varphi: \pi_1(V) \to U(1)$ be an admissible representation. Then the complex $E_{\tilde{\varphi}}(Z)$ is Fredholm with cohomology given by 
\[
H^0(E_{\tilde{\varphi}}(Z)) = 0, \quad H^1(E_{\tilde{\varphi}}(Z)) = H^1(M; \C_{\tilde{\varphi}}), \quad H^2(E_{\tilde{\varphi}}(Z)) = H^+(M; \C_{\tilde{\varphi}}),
\]
where $H^+(M; \C_{\tilde{\varphi}}):=\left\{\omega \in \Omega^2_{\C}(M): d_{\tilde{\varphi}} \omega = d^*_{\tilde{\varphi}} \omega = 0, \; \; \omega|_{\partial M = 0}, \; \;  \int_{M} \omega \wedge \omega > 0 \right\}$. 
\end{prop}

\begin{proof}
Due to (a) in the admissibility assumption, \autoref{l2.3} implies that $E_{\tilde{\varphi}}(Z)$ is Fredholm. Since $H^0(E_{\tilde{\varphi}}(Z))$ consists of $\tilde{\varphi}$-invariant functions of $L^2$-decay, only the constant zero function lies here. 

To identify $H^1(E_{\tilde{\varphi}}(Z))$, we take $w \in \ker d^+_{\tilde{\varphi}}$ as a representative of $[\omega] \in H^1(E_{\tilde{\varphi}}(Z))$. Stokes' formula gives us 
\[
0  = \int_Z d( w \wedge d_{\tilde{\varphi}} w) = \int_Z |d^+_{\tilde{\varphi}} w|^2 - \int_Z |d^-_{\tilde{\varphi}} w|^2. 
\]
So we conclude that $d_{\tilde{\varphi}} \omega = 0$. Let's fix a cut-off function $\eta$ over $Z$ with $\supp \eta \subset \tilde{V}_+$.  Combining (a) and (b) in \autoref{d3.6}, we see that $H^1(\tilde{V}_+; \C_{\tilde{\varphi}}) = 0$. Since $H^0(\tilde{V}_+; \C_{\tilde{\varphi}}) =0$, one can find a unique function $f(\omega): \tilde{V}^+ \to \C$ satisfying $\eta f(\omega) \in L^2_{k+1}(Z, \Lambda^0_{\C})$ and $\omega|_{\tilde{V}^+} = d_{\tilde{\varphi}} f(w)$.Then it's straightforward to verify that the map
\[
r: \omega \longmapsto \omega - d_{\tilde{\varphi}} (\eta f(\omega))
\]
induces an isomorphism on cohomology $r_*: H^1(E_{\tilde{\varphi}}(Z)) \to H^1(M, \partial M; \C_{\tilde{\varphi}})$. The long exact sequence of the pair $(M, \partial M)$ with local coefficient $\C_{\tilde{\varphi}}$ gives the identification $H^1(M, \partial M; \C_{\tilde{\varphi}}) \simeq H^1(M; \C_{\tilde{\varphi}})$. 

Finally, we identify $H^2(E_{\tilde{\varphi}}(Z))$. A non-zero class $[\omega] \in H^2(E_{\tilde{\varphi}}(Z))$ is represented by a smooth $2$-form $\omega$ satisfying 
\[
\omega = \star \omega \quad  \text{ and } \quad d^*_{\tilde{\varphi}} \omega = 0. 
\]
Let's write $Z_n = M \cup W_0 \cup ... \cup W_n$ for $n \in \mathbb{N}$, and $N_{n-1}$ a neighborhood of the outgoing boundary of $W_{n-1}$ in $Z_n$. The Hodge star identifies $H^2(Y;\C_{\varphi}) \simeq H^1(Y; \C_{\varphi}) = 0$. Then one can find a unique $\alpha_n(\omega) \in \Omega^2_{\C}(N_{n-1})$ satisfying $\omega|_{N_{n-1}} = d_{\tilde{\varphi}} \alpha_n(\omega)$. Take a cut-off function $\eta_n$ on $Z_n$ with $\eta_n|_{Z_{n-1}} = 1$ and $\eta_n|_{W_n \backslash N_{n-1}} = 0$. Then we get a map $r_{n, *}: H^2(E_{\tilde{\varphi}}(Z)) \to H^+(Z_n, \partial Z_n; \C_{\tilde{\varphi}})$ induced from 
\[
r_n: \omega \longmapsto
\begin{cases}
\omega & \text{ on $Z_{n-1}$ } \\
d_{\tilde{\varphi}} \left( \eta_n \alpha_n(\omega) \right) & \text{ on $N_{n-1}$ } \\
0 & \text{ on $W_n \backslash N_{n-1}$ } \\
\end{cases}
\]
Under our admissibility assumption (b), \cite[Lemma 5.6]{T87} implies that $r_{n, *}$ is an isomorphism for $n$ sufficiently large. By taking $z = 0$ in the admissibility assumption (a), we get $H^+(V; \C_{\varphi}) = H^2(E_{\varphi, 0}) = 0$. Then \cite[Lemma 5.7]{T87} can be applied to identify $H^+(Z_n, \partial Z_n;\C_{\tilde{\varphi}})$ with $H^+(M; \C_{\tilde{\varphi}})$ for all $n \in \mathbb{N}$. This completes the proof. 
\end{proof}

The strength of \autoref{p3.11} enables us to compare the periodic rho invariant with the usual rho invariant in general cases.

\begin{prop}\label{p3.12}
Let $(V, 1_V)$ be a non-degenerate pair, $\varphi: \pi_1(V) \to U(1)$ an admissible representation, and $Y$ a cross-section of $(V, 1_V)$ satisfying $H^0(Y;\C_{\varphi}) = H^1(Y; \C_{\varphi}) = 0$. Then 
\[
\tilde{\rho}_{\varphi}(V) - \rho_{\varphi}(Y) = \tilde{\eta}(V) - \eta(Y). 
\]
\end{prop}

\begin{proof}
It suffices to show $\eta_{\varphi}(Y) = \eta_{\tilde{\varphi}}(V)$. Let $\tilde{V}^{\infty}_+ = (-\infty, 0] \times Y \cup \tilde{V}_+$ be non-compact $4$-manifold with one cylindrical end and one periodic end. The representation $\varphi$ extends naturally to one on $\tilde{V}^{\infty}_+$, say $\tilde{\varphi}$. \autoref{t3.4} and \autoref{p3.9} implies that 
\[
\ind E_{\tilde{\varphi}} = \frac{\eta_{\varphi}(Y) - \eta_{\tilde{\varphi}}(V)}{2}. 
\]
On the other hand, we can apply \autoref{p3.11} to $M=[-1,0] \times Y$, which gives us $\ind E_{\tilde{\varphi}} =0$ as desired.
\end{proof}

\section{\large \bf Periodic Spectral Flow}\label{s4}

The periodic spectral flow was first introduced by Mrowka-Ruberman-Saveliev \cite{MRS11} for Dirac operators over homology $S^1 \times S^3$, and later generalized by the author \cite{M1} to Dirac operators over manifolds with cylindrical ends. The goal of this section is to recover such a notion for a path of ASD DeRham operators. We first consider the periodic spectral flow over closed $4$-manifolds, then derive a corresponding notion over end-cylindrical $4$-manifolds. 

\subsection{\em The Definition}\label{ss4.1} \hfill

\vspace{3mm}

Let $V$ be a closed Riemannian smooth $4$-manifolds equipped with a primitive class $1_V \in H^1(V;\Z)$ so that the pair $(E(V), 1_V)$ associated to the ASD DeRham complex is non-degenerate. We choose a smooth function $f: V \to S^1$ corresponding to $1_V$ as before. We denote by $\A_k(V, \C)$ the space of $L^2_k$ unitary connections on the trivial line bundle over $V$ for some fixed $k \geq 3$. To distinguish with $SU(2)$-connections, we write $A^{\dagger}$ for a typical connection in $\A_k(V, \C)$. To better cooperate with Yang-Mills theory, we consider perturbations of the ASD DeRham operator. Let 
\begin{equation}
\mathcal{P}= L^2_k\left(V, (\Lambda^1_{C}V)\smvee \otimes \Lambda^+_{\C}V \right)
\end{equation}
be the space of perturbations. Since $k \geq 3$, the Sobolev multiplication tells us $\pi: L^2_k(V, \Lambda^1_{\C}) \to L^2_k(V, \Lambda^+_{\C})$ is a bounded linear map with norm given by $\|\pi\|_{L^2_k}$. Typical examples of such perturbations arise as the differential of holonomy perturbations \cite{D87}. We are primarily interested in the case when the unitary connection $A^{\dagger}$ is flat. Given $\delta > 0$ and $\pi \in \mathcal{P}$, the $(\delta, \pi)$-perturbed ASD DeRham operator is denoted by
\begin{equation}
Q_{A^{\dagger}, \delta}(V, \pi) = -d_{A^{\dagger}, \delta}^* \oplus d^+_{A^{\dagger}, \delta, \pi}: L^2_k(V, \Lambda^1_{\C}) \longrightarrow L^2_{k-1}(V, \Lambda^0_{\C} \oplus \Lambda^+_{\C}),
\end{equation}
where $d_{A^{\dagger}, \delta}^* = d_{A^{\dagger}}^* - \delta/2 \cdot \iota_{df}$, and $d^+_{A^{\dagger}, \delta, \pi} = d^+_{A^{\dagger}} - \delta/2 \cdot (df \wedge -)^+ + \pi$. We note that as an operator from $L^2_k$ to $L^2_{k-1}$ spaces, the perturbation $\pi$ is compact. So 
\[
\ind Q_{A^{\dagger}, \delta}(V, \pi) = \ind Q_{A^{\dagger}, \delta}(V)  = \ind Q(V) = 0.
\] 
Given $z \in \C$, we can equally consider the $z$-twisted operator $Q_{A^{\dagger}, \delta, z}(V, \pi) = -d^*_{A^{\dagger}, \delta, z} \oplus d^+_{A^{\dagger}, \delta, z, \pi}$ corresponding to the elliptic complex
\begin{equation}\tag{$E_{A^{\dagger}, \delta, z}(V, \pi)$}
0 \longrightarrow L^2_{k+1}(V, \Lambda^0_{\C}) \xrightarrow{-d_{A^{\dagger}, \delta, z}} L^2_{k}(V, \Lambda^1_{\C}) \xrightarrow{d^+_{A^{\dagger}, \delta, z, \pi}} L^2_{k-1}(V, \Lambda^+_{\C}) \longrightarrow 0.
\end{equation}
Due to the $2\pi i $-periodicity, sometimes it would be more convenient to work with $e^z$. We denote the spectral set by 
\begin{equation}
\Sigma_{\delta}(A^{\dagger}, \pi):= \{ e^z \in \C^*: E_{A^{\dagger}, \delta, z}(V, \pi) \text{ is not acyclic} \}.
\end{equation}

\begin{dfn}\label{d4.1}
A continuous path $(A^{\dagger})_{t \in [0, 1]}$ of connections in $\A_k$ is $\delta$-regular with respect to a perturbation $\pi \in \mathcal{P}$ if the following hold. 
\begin{enumerate}[label=(\alph*)]
\item When $i=0, 1$, the operator $Q_{A^{\dagger}_i, \delta, z}(V, \pi)$ has trivial kernel with respect to all $z \in \C$ of vanishing real part. 
\item The parametrized spectral set 
\[
\Sigma^I_{\delta}(A^{\dagger}_t, \pi):=\left\{(t, e^z) \in [0, 1] \times S^1: \ker Q_{A^{\dagger}_t, \delta, z}(V, \pi) \neq 0 \right\}
\]
is discrete and 
\[
\ker Q_{A^{\dagger}_t, \delta, z}(V, \pi)  = \C, \;\; \forall (t, e^z) \in \Sigma^I_{\delta}(A^{\dagger}_t, \pi). 
\]
\end{enumerate}
\end{dfn}

\begin{rem}\label{r4.2}
When $d_{A^{\dagger}, \delta, z}$ is not the product connection, it's necessary that $H^0(E_{A^{\dagger}, \delta, z}) = 0$. Because $H^0(E_{A^{\dagger}, \delta, z})$ consists of complex valued functions on $V$ that are parallel with respect to $d_{A^{\dagger}, \delta, z}$. Thus the evaluation of such a function at a given point is invariant under the action of the holonomy group of $d_{A^{\dagger}, \delta, z}$. When the holonomy group is nontrivial, the only invariant value is $0$. So the $\delta$-regularity essentially means at a spectral point $(t, e^z)$, the symbol sequence of $H^*(E_{A^{\dagger}_t, \delta, z}(V, \pi))$ is exact and $-\chi(\widehat{H}^*_z(V)) = 1$ due to (4) of \autoref{p2.15}. 
\end{rem}

\begin{lem}\label{l4.3}
Given $\delta \in \R$, let $(A^{\dagger}_t)$ be a path of flat connections in $\A_k$ such that $A^{\dagger}_t - (\delta/2 + z) \cdot df$ is not the product connection for all $t \in [0, 1]$ and $e^z \in S^1$. Then $(A^{\dagger}_t)$ is a $\delta$-regular path with respect to a generic perturbation $\pi \in \mathcal{P}$. 
\end{lem}

\begin{proof}
Let $C^{[0,1]} = [0, 1] \times S^1$ be the parameter space. Take $(t_0, e^{z_0}) \in C^{[0, 1]}_{\delta}:=[0, 1] \times S^1_{\delta}$, where $S^1_{\delta}$ means the circle in $\C$ with radius $\delta/2$. Let's write $A^{\dagger}_0$ for $A^{\dagger}_{t_0}$. Denote by $\mathcal{H}_1 = \ker Q_{A^{\dagger}_0, \delta, z_0} $ and $\mathcal{H}_2 =( \im Q_{A^{\dagger}_0, \delta, z_0})^{\perp}$. Since $Q_{A^{\dagger}_0, \delta, z_0}$ is Fredholm of index $0$, we know $\dim \mathcal{H}_1 = \dim \mathcal{H}_2$ is finite. We consider the map 
\begin{equation}
\begin{split}
\eta: \mathcal{P} \times C^{[0,1]}_{\delta} \times \ker d^*_{A^{\dagger}_0, \delta, z_0} & \longrightarrow \im d^+_{A^{\dagger}_0, \delta, z_0} \\
(\pi, t, e^z, a) & \longmapsto \Pi d^+_{A^{\dagger}_{t}, \delta, z, \pi} a,
\end{split}
\end{equation}
where $\Pi: L^2_{k-1}(V, \Lambda^+_{\C}) \to \im d^+_{A^{\dagger}_0, \delta, z_0}$ is the $L^2$-projection to the image of $d^+_{A^{\dagger}_0, \delta, z_0}$. Since $A^{\dagger}_0$ is a flat connection, Hodge theory tells us that the differential $D\eta|_{(0, t_0, e^{z_0}, a)}$ is surjective on the $\ker d^*_{A^{\dagger}_0, \delta, z_0}$ component. Invoking the implicit function theorem, we can find a neighborhood $U_1 \times U_2$ of $(0, t_0, e^{z_0})$ in $\mathcal{P} \times C^{[0,1]}_{\delta}$  and a map $h: U_1 \times U_2 \times \mathcal{H}_1 \to \ker d^*_{A^{\dagger}_0, \delta, z_0}$ such that for all $(\pi, t, e^z, a) \in U_1 \times U_2 \times \mathcal{H}_1$ we have $\eta(\pi, t, e^z, a+h(\pi, t, e^z, a)) = 0$. Equivalently, $d^+_{A^{\dagger}_t, \delta, z, \pi}(a+ h(\pi, t, e^z, a)) \in \mathcal{H}_2$. Since $\eta$ is linear in $a$, so is $h$. In this way we get a map 
\begin{equation}
\begin{split}
\xi: U_1 \times U_2 & \longrightarrow \Hom_{\C}(\mathcal{H}_1, \mathcal{H}_2) \\
(\pi, t, e^z) & \longmapsto \left(a \mapsto d^+_{A^{\dagger}_t, \delta, z, \pi}(a + h(\pi, t, e^z, a)) \right). 
\end{split}
\end{equation}
The differential of $\xi$ at $(0, 0)$ is $D\xi|_{(0,0)} (\pi, 0): a \mapsto \pi (a) + d^+_{A^{\dagger}_0, \delta, z_0} \circ Dh|_{(0, 0, a)} \pi$. By the construction when $\pi(\mathcal{H}_1) \subset \mathcal{H}_2$, we have $D\xi|_{(0, 0)}(\pi, 0) = \pi(a)$. Restricting to such perturbations, we conclude that $\rho$ is a submersion at $(0, 0)$. Note that the stratum of linear maps in $\Hom_{\C}(\mathcal{H}_1, \mathcal{H}_2)$ with dimension-$i$ kernels is of complex codimension $i^2$. Since $\dim_{\R} U_2 = 2$, the Sard-Smale theorem tells us that for a generic perturbations $\dim_{\C} \Pi' \ker d^+_{A^{\dagger}_t, \delta, z, \pi} \leq 1$ with equality only at a discrete set of points, where $\Pi'$ is the projection to $\ker d^*_{A^{\dagger}_0, \delta, z_0} \cap \ker d^+_{A^{\dagger}_0, \delta, z_0} = \ker Q_{A^{\dagger}_0, \delta, z_0}$. 

Note that the dimension of kernel is an upper semi-continuous function on Fredholm operators and the projection $\Pi': \ker Q_{A^{\dagger}, \delta, z} \to \ker Q_{A^{\dagger}_0, \delta, z_0}$ is injective for $A^{\dagger}$ near $A_0$ and $z$ near $z_0$. By shrinking $U_2$ possibly, we have proved that for each $(t_0, e^{z_0})$ one can find a neighborhood $U_2$ of $(t_0, e^{z_0})$ in $C^{[0, 1]}$  over which the $\delta$-regular condition is satisfied with respect to generic perturbations. Since $C^{[0, 1]}$ is compact, we can run the argument in finite steps to conclude. We also remark that since $\dim_{\R} \partial C^{[0, 1]} = 1$, one can first run the argument there to arrive at the case that $Q_{A^{\dagger}_t, \delta, z}$ is invertible for all $(t, e^z) \in \partial C^{[0, 1]}$.
\end{proof}

Given a $\delta$-regular path $(A^{\dagger}_t)$ with respect to $\pi$, we write $(t_j, e^{z_j}) \in \Sigma^I_{\delta}(A^{\dagger}_t, \pi)$ for the spectral points on the cylinder. The following result in \cite{MRS11} describes the spectral set in a neighborhood of the cylinder. Although the original statement is proved for Dirac operators, the proof only uses the property of index zero Fredholm operators. So it also works in our case. 

\begin{prop}{\label{p4.4}\cite[Theorem 4.8]{MRS11}}
Let $(A^{\dagger}_t)$ be a $\delta$-regular path with respect to a perturbation $\pi \in \mathcal{P}$. Then there exists $\epsilon_0 > 0$ and $\epsilon_1 > 0$
\[
\bigcup_{|t-t_j|< \epsilon_0} \{t \} \times (\Sigma_{\delta}(A^{\dagger}_t, \pi) \cap B_{\epsilon_1}(e^{z_j})) \subset [0, 1] \times \C^*
\]
is a smoothly embedded curve for all $j$, where $B_{\epsilon_1}(e^{z_j})$ is the open disk of radius $\epsilon_1$ centered at $e^{z_j}$. 
\end{prop}

We recall the definition of the periodic spectral flow introduced in \cite{MRS11}. For more details and explanation of the following definition, one may consult \cite[Section 4]{M1}

\begin{dfn}
We say $\lambda > 0$ is an excluded value for $Q_{A^{\dagger}, \delta}(V, \pi)$ if
\[
\text{ $Q_{A^{\dagger}, \delta, z}(V, \pi)$ is invertible for all $z$ with $\Rea z = \ln \lambda$}. 
\]
\end{dfn}

\begin{dfn}
Let $(A^{\dagger}_t)$ be a $\delta$-regular path with respect to $\pi$. A system of excluded values for $Q_{A^{\dagger}_t, \delta}(X, \pi)$ is a finite sequence of pairs $\{(t_l, \lambda_l)\}_{l=1}^n$ satisfying 
\begin{enumerate}[label=(\alph*)]
\item $0=t_0< t_1 < ... < t_n =1$ is a partition of $[0, 1]$;
\item $\lambda_l \in (1 - \epsilon_1, 1 + \epsilon_1)$ is an excluded value for $Q_{A^{\dagger}_t, \delta}(V, \pi)$ for all $t \in [t_{l-1}, t_l]$, where $\epsilon_1$ is the constant in \autoref{p4.4}. Moreover $\lambda_1 = \lambda_n = 1$.
\end{enumerate}
\end{dfn}

\begin{dfn}
Let $\{(t_l, \lambda_k)\}_{l=1}^n$ be a system of excluded values for a $\delta$-regular path $(A^{\dagger}_t)$ with respect to $\pi$. For $l=1, ..., n-1$, we define 
\[
a_l=
\begin{cases}
1 & \text{ if $\lambda_l > \lambda_{l+1}$} \\
0 & \text{ if $\lambda_l = \lambda_{l+1}$} \\
-1 & \text{ if $\lambda_l < \lambda_{l+1}$} 
\end{cases}
\]
and 
\[
b_l = \#\left\{e^z \in \Sigma_{\delta}(A^{\dagger}_{t_l}, \pi): \min(\lambda_l, \lambda_{l+1}) \leq |e^z| \leq \max(\lambda_l, \lambda_{l+1}) \right\}. 
\]
The periodic spectral flow of the family $Q_{A^{\dagger}_t, \delta}(V, \pi)$ along the $\delta$-regular path $(A^{\dagger}_t)$ is defined to be 
\[
\widetilde{\Sf}(Q_{A^{\dagger}_t, \delta}(V, \pi)) := \sum_{l=1}^{n-1} a_l b_l \in \Z. 
\]
\end{dfn}

\begin{rem}\label{r3.7}
We remark several facts about the periodic spectral flow.
\begin{enumerate}[label=(\alph*)]
\item It follows from \cite[Lemma 4.9]{M1} that the periodic spectral flow is independent of the choice of systems of excluded values. 

\item A more intuitive definition of the periodic spectral flow was initially adopted by Mrowka-Ruberman-Saveliev \cite{MRS11} as follows. Near each spectral point $(t_j, e^{z_j})$, the spectral curve in \autoref{p4.4} takes the form $(t, e^{z(t)})$, $t \in (t_j-\epsilon , t_j+ \epsilon)$. We write $z(t) = r(t) + is(t)$. The $\delta$-regularity implies that $\dot{r}(t_j) \neq 0$. To $(t_j, e^{z_j})$, we assign $+1$ if $\dot{r}(t_j) > 0$, and $-1$ otherwise. Then the periodic spectral flow is given by the signed count of spectral points on $[0, 1] \times S^1$. 

\item It follows from \autoref{p4.4} that a $\delta$-regular path $(A^{\dagger}_t)$ is also $\delta_{\epsilon}$-regular for all $\delta_{\epsilon}$ near $\delta$. 
\end{enumerate}
\end{rem}

Given a path $(A^{\dagger}_t)$ of flat connections satisfying the assumption of \autoref{l4.3}, \cite[Lemma 4.10]{M1} implies that as long as the generic perturbations are small enough, the periodic spectral flow is independent of the choice of perturbations as well. In this case, we shall write $\widetilde{\Sf}(Q_{A^{\dagger}_t, \delta}(V))$ for the periodic spectral flow with respect to small perturbations. In other words, the perturbation is only auxiliary and added for the purpose of achieving transversality. 

Just like the case for Dirac operators in \cite{M1}, we also consider periodic spectral flow defined over manifolds with boundary. The idea is to attach a cylindrical end and impose certain decay assumptions for the connections, which enables us to treat the ASD DeRham operator associated to these connections as if defined over closed manifolds. 

Let $M$ be a smooth compact $4$-manifold with $\partial M = Y$. For simplicity we assume $Y$ is connected. We write $M_o = M \cup [0, \infty) \times Y$ for the manifold obtained by attaching a cylindrical end to $M$. We assume that $b_1(M) > 0$. Then we fix a primitive class $1_M \in H^1(M_o; \Z)$ represented by a function $f: M_o \to S^1$ whose restriction on the end, $f|_{\{t\} \times Y} = f_Y$ for $t \in [0, \infty)$, is time-independent. 

Any connection $A$ on the trivial line bundle over $M_o$ takes the form $A^{\dagger}|_{[0, \infty) \times Y} =dt \otimes c(t) + B^{\dagger}(t)$ with $B^{\dagger}(t)$ a time-dependent connection on $Y$ and $c(t)$ a time-dependent complex-valued function on $Y$. Using the identification (\ref{e3.13}), the ASD DeRham operator over the end $[0, \infty) \times Y$ takes the form $Q_{A^{\dagger}, \delta} = d/dt + c(t) + L_{B^{\dagger}(t), \delta}$ with 
\begin{equation}\label{e4.6}
L_{B^{\dagger}, \delta}=
\begin{pmatrix}
0 & -d^*_{B^{\dagger}} + \delta/2 \cdot \iota_{df_Y} \\
-d_{B^{\dagger}} + \delta/2 \cdot df_Y & \star (d_{B^{\dagger}} - \delta/2  \cdot df_Y)
\end{pmatrix}.
\end{equation}
In contrast with $L_{\delta}$ in (\ref{e3.16}), the $(1,1)$-entry of $L_{B^{\dagger}, \delta}$ vanishes because $f|_{\{t\} \times Y}$ is independent of $t$. Given $z \in \C$, we get the $z$-twisted operator $Q_{A^{\dagger}, \delta, z} = d/dt + c(t) + L_{B^{\dagger}(t), \delta, z}$ as in the closed case, where  
\begin{equation}\label{e4.7}
L_{B^{\dagger}, \delta, z}=
\begin{pmatrix}
0 & -d^*_{B^{\dagger}} + (\delta/2 + \bar{z}) \cdot \iota_{df_Y} \\
-d_{B^{\dagger}} + (\delta/2 +  z) \cdot df_Y & \star (d_{B^{\dagger}} - (\delta/2 + z)  \cdot df_Y)
\end{pmatrix}.
\end{equation}



\begin{dfn}\label{d4.8}
Given $\delta \in \R$, we say a connection $A^{\dagger}$ on $M_o$ is $\delta$-regularly asymptotically flat ($\delta$-RAF) if the following hold:
\begin{enumerate}[label=(\alph*)]
\item There exists $T > 0$, $\mu > 0$ and a flat connection $B^{\dagger}_o$ on $Y$ such that 
\[
\|B^{\dagger}(t) - B^{\dagger}_o\|_{C^0(Y)} + \|c(t)\|_{C^0(Y)} \leq \mathfrak{c} e^{-\mu(t - T)/2}
\]
for some constant $\mathfrak{c} > 0$;
\item The twisted operator $L_{B^{\dagger}_o, \delta, z}: L^2_{k-1/2} \to L^2_{k-3/2}$ is invertible for all $z$ with $\Rea z =0$;
\item There exists $\epsilon_Y > 0$ such that $Q_{A^{\dagger}, \delta, z}(M_o): L^2_k \to L^2_{k-1}$ has index $0$ for all $z$ satisfying $\Rea z \in (\ln (1 - \epsilon_Y), \ln (1 + \epsilon_Y))$. 
\end{enumerate}
\end{dfn}

This definition deserves some explanation. Let's write $B_z := B^{\dagger}_o - (\delta/2 +z)df_Y$ for a moment. Then we get a twisted DeRham complex 
\[
0 \to \Omega^0_{\C}(Y) \xrightarrow{d_{B_z}} \Omega^1_{\C}(Y) \xrightarrow{d_{B_z}} \Omega^2_{\C}(Y) \xrightarrow{d_{B_z}} \Omega^3_{\C}(Y) \to 0.
\]
The second assumption (b) is equivalent to the fact that $H^0_{B_z}(Y) = H^1_{B_z}(Y) = 0$ for all $\Rea z =0$. Note that 
\[
Q_{A^{\dagger}, \delta, z}|_{[T, \infty) \times Y} = \frac{d}{dt} + L_{B^{\dagger}_o, \delta, z} + L_{B^{\dagger}(t), \delta, z} - L_{B^{\dagger}_o, \delta, z} + c(t). 
\]
The assumption of exponential decay in (a) implies that the zeroth order operator over $[T, \infty) \times Y$ given by 
\[
L_{B^{\dagger}(t), \delta, z} - L_{B^{\dagger}_o, \delta, z} + c(t) : L^2_k \to L^2_{k-1}
\]
is compact. Then the second assumption (b) implies that $Q_{A^{\dagger}, \delta, z}(M_o)$ is Fredholm for $\Rea z = 0$. The openness of being Fredholm gives us the existence of such $\epsilon_Y$ to guarantee that $Q_{A, z}(M_o)$ is Fredholm with respect to $z$ in the asserted range. The requirement of having index zero can be justified by the topology of $M$ using the theory of Atiyah-Patodi-Singer \cite{APS1}:
\begin{equation}
\ind Q_{A^{\dagger}, \delta, z}(M_o) = -\frac{1}{2}\left(\chi(M) + \sigma(M) + \eta_{B_z}(Y)\right). 
\end{equation}

Over the non-compact manifold $M_o$, we consider perturbations of exponential decay along the cylindrical end. We shall fix a weight $\mu > 0$ and let 
\begin{equation}\label{e4.8}
\mathcal{P}_{\mu} = L^2_{k, \mu} \left(M_o, (\Lambda^1_{\C}M_o)\smvee \otimes \Lambda^+_{\C}M_o\right)
\end{equation}
be the space of perturbations. Then for each $\pi \in \mathcal{P}_{\mu}$, Sobolev multiplication gives us a compact operator $\pi:L^2_k \to L^2_{k-1}$. The notion of $\delta$-regularity in \autoref{d4.1} can be defined for a path $(A_t)$ of $\delta$-RAF connections over $M_o$ due to the assumption (c) in \autoref{d4.8}. The argument of \autoref{l4.3} can be applied verbatim to show that any path $(A_t)$ of flat $\delta$-RAF connections is $\delta$-regular with respect to a generic perturbation $\pi \in \mathcal{P}_{\mu}$. This leads us to define the periodic spectral flow of such a path, which we denote by $\widetilde{\Sf}(Q_{A^{\dagger}_t, \delta}(M_o, \pi)).$

\subsection{\em Properties}\label{ss3.2} \hfill
 
\vspace{3mm}

We shall demonstrate two properties of periodic spectral flow. The first is a gluing property relating the notion over closed manifolds to manifolds with cylindrical end. The second is to establish its relation with the periodic rho invariant. Both properties are known for the usual spectral flow, and proved in the case of Dirac operators in \cite{M1} and \cite{MRS11} respectively. 

We restrict ourselves to a concrete model. Suppose the $4$-manifold $V$ considered above is decomposed along an embedded $3$-torus $Y$ into two compact pieces: $V = M \cup_Y N$. Let $g$ be a metric on $V$ so that $g|_{(-1, 1) \times Y} = dt^2 + h$ for a metric $h$ on the $3$-manifold $Y$. Given $T > 0$, we write $V_T$ for the Riemannian manifold obtained from $V$ by inserting a "neck" of length $2T$ about $Y$. More specifically, let's write $M_T = M \cup [0, T] \times Y$, $N_T = [-T, 0] \times Y \cup N$. Then $V_T = M_T \cup N_T$. The "neck" is identified with $[-T, T] \times Y \subset V_T$. The limit of such metrics breaks $V$ into two pieces, which we denote by $M_o = M \cup [0, \infty) \times Y $ and $N_o = (-\infty, 0] \times Y \cup N$. 

Assume the primitive class $1_V \in H^1(V; \Z)$ restricts to primitive classes on $M$ and $N$. We then choose a representative function $f: V \to S^1$ satisfying $f|_{(-1, 1) \times Y} = f_Y$ for some function $f_Y: Y \to S^1$. The restriction of $f$ can be extended to $f_M: M_o \to S^1$ and $f_N: N_o \to S^1$ by taking $f|_{\{t\} \times Y}=f_Y$ on the cylindrical ends. 

Suppose for each $T \in [0, \infty)$ we have a path $(A^{\dagger}_{T, s})_{s\in [0, 1]}$ of flat connections over the trivial line bundle of $V_T$ which converges to $\delta$-RAF paths $(A^{\dagger}_{M, s})_{s \in [0,1]}$ and $(A^{\dagger}_{N, s})_{s \in [0, 1]}$ on $M_o$ and $N_o$ respectively in the following sense:
\begin{enumerate}[label=(\alph*)]
\item For each $s \in [0, 1]$, $A^{\dagger}_{M, s}$ and $A^{\dagger}_{N, s}$ are both exponentially asymptotic to the same flat connection $B^{\dagger}_s$ over $Y$ in the sense of \autoref{d4.8} (a);
\item For each $s \in [0, 1]$, $A^{\dagger}_{T, s}|_{M_T}$ and $A^{\dagger}_{T, s}|_{N_T}$ converges to $A^{\dagger}_{M, s}$ and $A^{\dagger}_{N, s}$ respectively as $T \to \infty$ in the $L^2_{k, loc}$ topology;
\item Given any $\epsilon > 0$, one can find $T_{\epsilon}$ such that for all $T > T_{\epsilon}$ we have 
\[
\|A^{\dagger}_{T, s}|_{[T_{\epsilon}-T, T - T_{\epsilon}] \times Y} - A^{\dagger}_{B^{\dagger}_s}\|_{L^2_k} < \epsilon, 
\]
where $A^{\dagger}_{B^{\dagger}_s} = dt + B^{\dagger}_s$ is the connection on $[T_{\epsilon}-T, T - T_{\epsilon}] \times Y$. 
\end{enumerate}
Moreover assume we can find perturbations $\pi_T \in \mathcal{P}(V_T)$ converging to $\pi_M \in \mathcal{P}_{\mu}(M_o)$ and $0 \in \mathcal{P}_{\mu}(N_o)$ as $T \to \infty$ so that the paths $(A^{\dagger}_{T, s})$, $(A^{\dagger}_{M, s})$, and $(A^{\dagger}_{N, s})$ are $\delta$-regular with respect to the perturbations $\pi_T$, $\pi_M$, and $0$ respectively. 

The following result is drawn from \cite{M1}. Although originally the author was considering Dirac operators, the proof goes through without any change. The idea is that when we impose the condition that the path $(A_{N^{\dagger}, s})$ admits empty spectral curve near the unit circle $S^1$, the spectral curves of paths $(A^{\dagger}_{T, s})$ and $(A^{\dagger}_{M, s})$ can be identified when $T \gg 0$ with the help of implicit function theorem. In general if $(A^{\dagger}_{N, s})$ contributes spectral curves, the author expects an additive formula of the periodic spectral curve as in the work of \cite{CLM96} for usual spectral flows. 

\begin{prop}[\label{p4.10}{\cite[Theorem 4.16]{M1}}]
Under the situation as above, we further assume that 
\[
Q_{A^{\dagger}_{N, s}, \delta, z}(N_o): L^2_k(N_o, \Lambda^1_{\C}) \longrightarrow L^2_{k-1}(N_o, \Lambda^0_{\C} \oplus \Lambda^+_{\C})
\]
is invertible for all $(s, e^z) \in [0, 1] \times S^1$. Then there exists $T_o > 0 $ such that for all $T > 0$ we have 
\[
\widetilde{\Sf}(Q_{A^{\dagger}_{T, s}, \delta}(V_T, \pi_T)) = \widetilde{\Sf}(Q_{A^{\dagger}_{M, s}, \delta}(M_o, \pi_M)). 
\]
\end{prop}

Next we relate the periodic spectral flow with the periodic rho invariant. Let $\varphi: \pi_1(V) \to U(1)$ be an admissible representation. Then we have a well-defined rho invariant $\tilde{\rho}_{\varphi}(V, Q) = \rho(V, Q_{\varphi}) - \rho(V, Q)$. Denote by $A^{\dagger}_{\varphi}$ the flat connection on the trivial line bundle of $V$ whose holonomy corresponds to $\varphi$. Let $A^{\dagger}_0$ be the product connection, and $A^{\dagger}_t = A^{\dagger}_0 + t(A^{\dagger}_{\varphi} - A^{\dagger}_0)$ be the path of flat connections from $A^{\dagger}_0$ to $A^{\dagger}_{\varphi}$. 

\begin{prop}\label{p4.11}
Let $\varphi$ and $(A^{\dagger}_t)$ be chosen as above, and $\delta_0> 0$ be a positive constant satisfying (\ref{e3.6}). Then for each $\delta \in (0, \delta_0)$, after choosing generic small perturbations $\pi \in \mathcal{P}$ turning $(A^{\dagger}_t)$ into a regular path with respect to both $\delta$ and $-\delta$ we have
\[
-\tilde{\rho}_{\varphi}(V, Q) = \widetilde{\Sf}(Q_{A^{\dagger}_t, \delta}(V, \pi)) + \widetilde{\Sf}(Q_{A^{\dagger}_t, -\delta}(V, \pi))
\]
\end{prop}

\begin{proof}
For fixed $\delta \in (-\delta_0, \delta_0) \backslash \{0\}$, we can choose generic perturbations $\pi \in \mathcal{P}$ which makes $(A^{\dagger}_t)$ a $\delta$-regular path appealing to \autoref{l4.3}.

Let $Z = M \cup \tilde{V}_+$ be an end-periodic manifold with its end modeled on the cyclic cover $\tilde{V}$ corresponding to $1_V$. We lift and extend the path of flat connections $A^{\dagger}_t$ to $Z$, say $\tilde{A}^{\dagger}_t$, in a way that $\tilde{A}^{\dagger}_0$ is the product connection. The function $f: V \to S^1$ is also lifted and extended to a function $\rho: Z \to \R$ as before. Taubes' criterion \cite[Lemma 4.3]{T87} implies that
\[
Q_{\tilde{A}^{\dagger}_i, \delta}(Z): L^2_k(Z, \Lambda^1_{\C}) \longrightarrow L^2_{k-1}(Z, \Lambda^0_{\C} \oplus \Lambda^+_{\C})
\]
is Fredholm for $i = 0, 1$. Given a perturbation $\pi \in \mathcal{P}$, we can equally lift and extend to one, say $\tilde{\pi}$, on $Z$. Now Sobolev multiplication implies the map $\tilde{\pi}: L^2_k\to L^2_{k-1}$ is bounded but not necessarily compact. The operator norm of $\tilde{\pi}$ is bounded by $\|\pi\|_{L^2_k} + \|\tilde{\pi}|_M\|_{L^2_k}$, which can be chosen as small as we want while allowing varying $\pi$ generically. Let's denote by $Q_{\tilde{A}^{\dagger}_i, \delta}(Z, \tilde{\pi})$ the operator obtained by perturbing the $d^+$-component of $Q_{\tilde{A}^{\dagger}_i, \delta}(Z)$ using the lift $\tilde{\pi}$ of $\pi$. 

We claim that the argument of \cite[Theorem 7.3]{MRS11} can be applied in our case to obtain
\begin{equation}\label{e4.9}
\ind Q_{\tilde{A}^{\dagger}_1, \delta}(Z, \tilde{\pi}) - \ind Q_{\tilde{A}^{\dagger}_0, \delta}(Z, \tilde{\pi}) = \widetilde{\Sf}(Q_{A^{\dagger}_t, \delta}(V, \pi)). 
\end{equation}
Let's assume this claim temporarily and continue with the proof. By choosing lifts $\tilde{\pi}$ of small operator norm, we have 
\begin{equation}
\ind Q_{\tilde{A}^{\dagger}_i, \delta}(Z, \tilde{\pi}) = \ind Q_{\tilde{A}^{\dagger}_i, \delta}(Z). 
\end{equation}
We can then apply \autoref{t3.4} to $\delta$ and $-\delta$ respectively to get
\[
-\tilde{\rho}_{\varphi}(V, Q) =\widetilde{\Sf}(Q_{A^{\dagger}_t, \delta}(V, \pi)) + \widetilde{\Sf}(Q_{A^{\dagger}_t, -\delta}(V, \pi)).
\]

Now we prove the claim. The argument is essentially local. Let $(t_0, e^{z_0}) \in \Sigma^I_{\delta}(A^{\dagger}_t, \pi)$ be a spectral point. We note that the index of $Q_{\tilde{A}^{\dagger}_t, \delta}(Z, \tilde{\pi}):L^2_k \to L^2_{k-1}$ equals the index of the differential complex:
\begin{equation}\tag{$E(\tilde{A}^{\dagger}_t, \delta, \tilde{\pi})$}
0 \longrightarrow L^2_{k+1}(Z, \Lambda^0_{\C}) \xrightarrow{-d_{t, \delta}} L^2_{k}(Z, \Lambda^1_{\C}) \xrightarrow{d^+_{t, \delta, \tilde{\pi}}} L^2_{k-1}(Z, \Lambda^+_{\C}) \longrightarrow 0,
\end{equation}
where we have omitted $\tilde{A}^{\dagger}$ in the notation of the differential operators. We choose $\delta' > 0$ so that the weighted differential complex $E_{\delta'}(\tilde{A}^{\dagger}_t, \delta, \tilde{\pi})$ is Fredholm for all $t \in (t_0-\epsilon, t_0+\epsilon)$ for some $\epsilon$ small enough. Then the index jump 
\[
\ind Q_{\tilde{A}^{\dagger}_{t_0+\epsilon}, \delta}(Z, \tilde{\pi}) - \ind Q_{\tilde{A}^{\dagger}_{t_0-\epsilon}, \delta}(Z, \tilde{\pi})
\]
is given by the difference
\begin{equation}\label{e4.10}
\left(\ind E(\tilde{A}^{\dagger}_{t_0+\epsilon}, \delta, \tilde{\pi}) - \ind_{\delta'} E(\tilde{A}^{\dagger}_{t_0+\epsilon}, \delta, \tilde{\pi}) \right) - \left(\ind E(\tilde{A}^{\dagger}_{t_0 - \epsilon}, \delta, \tilde{\pi}) - \ind_{\delta'} E(\tilde{A}^{\dagger}_{t_0 - \epsilon}, \delta, \tilde{\pi}) \right).
\end{equation}
The index jump formula \autoref{c2.14} tells us 
\begin{equation}\label{e4.11}
\ind E(\tilde{A}^{\dagger}_{t_0+\epsilon}, \delta, \tilde{\pi}) - \ind_{\delta'} E(\tilde{A}^{\dagger}_{t_0+\epsilon}, \delta, \tilde{\pi}) = \sum_{z} -\chi(\widehat{H}^*_z(V)),
\end{equation}
where the sum ranges over $z \in (0, \delta' /2) \times i(0, 2\pi)$ and $\widehat{H}^*_z(V)$ is defined with respect to the elliptic complex $E_{A^{\dagger}_{t_0+\epsilon}, \delta, z}(V, \pi)$. As pointed out in \autoref{r4.2}, the $\delta$-regularity of the path $A^{\dagger}_t$ implies that $-\chi(\widehat{H}^*_z(V)) =1$ for each spectral point $z$ and vanishes otherwise. The appearance of $\Rea z \in (0, \delta'/2)$ in the spectral curve on $t \in (t_0, t_0+\epsilon)$ enables us to take excluded values near $t_0$ as $\lambda_l \in (1, e^{\delta'/2})$ and $\lambda_{l+1} \in (e^{-\delta'/2}, 1)$. In particular $\lambda_l > \lambda_{l+1}$, and the contribution to the periodic spectral flow near $t_0$ is the number of the spectral points with $|e^z| \in (\lambda_{l+1}, \lambda_l)$, which matches up with the index jump (\ref{e4.11}). Similarly we apply the index jump formula to deal with the second term in (\ref{e4.10}). The only difference is now we are looking at $t \in (t_0-\epsilon, t_0)$ so the excluded values near $t_0$ get reversed, which would cause a minus sign in the spectral flow counting. This matches with the minus sign of the second term in (\ref{e4.10}). So we get the claimed formula (\ref{e4.9}).
\end{proof}

\begin{rem}
A priori, the choice of perturbations $\pi$ depends on the choice of $\delta$. However, if we only allow $\delta$ to vary within a compact subset of $(0, \delta_1)$, the perturbations $\pi$ can be chosen to work for all $\delta$ in this range using the argument in \autoref{l4.3}. Thus the corresponding spectral flow does not depend on the choices of small $\pi$ and $\delta$ in this range. When $\delta = 0$, defining periodic spectral flow requires perturbing the $-d^*$-component of the ASD DeRham operator $Q$ due to the appearance of the product connection. After adding such perturbations, the periodic spectral flow might depend on the choice of perturbations. 
\end{rem}

\section{\large \bf Singular Furuta-Ohta Invariants}\label{s5}

In this section, we first recall the definition of the singular Furuta-Ohta invariant introduced by Echeverria \cite{E19} as the count of singular instantons over homology $S^1 \times S^3$. Then we derive an equivalent formulation using the structure theorem of the moduli space of degree zero instantons over a homology $D^2 \times T^2$ deduced by the author \cite{M2}. 

\subsection{\em The Definition}\label{ss4.1} \hfill
 
\vspace{3mm}

We briefly recall the set-up of singular instantons in \cite{KM93}. Let $X$ be a smooth connected closed $4$-manifold with an embedded surface $\Sigma$. Let $E$ be an $SU(2)$-bundle over $X$ with a reduction $E|_{\nu(\Sigma)} = L \oplus L^*$ in a neighborhood $\nu(\Sigma)$ of the surface, where $L \to \nu(\Sigma)$ is a line bundle. Such a bundle is characterized by two topological quantities:
\begin{equation}
m = c_2(E)[X] \text{ and } l = -c_1(L)[\Sigma],
\end{equation}
which are referred to as the instanton number and monopole number respectively. We can regard $\nu(\Sigma)$ as a disk bundle over $\Sigma$. Let $\eta \in \Omega^1(\nu(\Sigma) \backslash \Sigma; \R)$ be the angular form on the punctured disk bundle. Pick a smooth unitary connection $A_0$ on $E$ as a reference connection. To a holonomy parameter $\alpha \in (0, 1)$, we assign a model connection:
\begin{equation}
A^{\alpha}_0 = A_0 + i\beta 
\begin{pmatrix}
\alpha & 0 \\
0 & -\alpha
\end{pmatrix}
\eta,
\end{equation}
where $\beta$ is a cut-off function on the radial direction of $\nu(\Sigma)$ that equals $1$ near $\Sigma$. We think of $A^{\alpha}_0$ as a unitary connection on $E|_{X \backslash \Sigma}$. 

So far the notion of metric has not entered the picture. Let $\nu$ be a positive integer. We denote by $g^{\nu}$ a metric on $X$ that is smooth on $X \backslash \Sigma$ and singular along $\Sigma$ with cone-angle $2\pi/\nu$. Alternatively we can think of $g^{\nu}$ as an orbifold metric after turning $X$ into an orbifold whose singular set is $\Sigma$ with isotropy group $\Z/\nu$. Let's write $\g_E$ for the adjoint bundle associated to $E$. We denote by $\check{L}^2_k(X, \Lambda^* \otimes {\g_E})$ the space of $\g_E$-valued forms completed by the $L^2_k$ Sobolev norm associated to the Levi-Civita connection of $g^{\nu}$ and $A^{\alpha}_0$. The key analytic input in the theory of singular instantons is the following result proved by Kronheimer-Mrowka \cite{KM93}.

\begin{prop}[\label{p5.1}{\cite[Proposition 4.17]{KM93}}]
Given any compact interval $I \subset (0, 1/2)$ and integer $k_0 > 0$, there exists $\nu_0 > 0$ such that for all $\nu \geq \nu_0$ and $k \leq k_0$ the operator 
\[
Q_{A^{\alpha}_0}= - d^*_{A^{\alpha}_0} \oplus d^+_{A^{\alpha}_0}: \check{L}^2_k(X, \Lambda^1 \otimes \g_E) \longrightarrow \check{L}^2_{k-1}(X, (\Lambda^0 \oplus \Lambda^+)\otimes \g_E)
\]
is Fredholm, $L^2$-self-adjoint, and satisfies the Fredholm alternative. 
\end{prop}

With \autoref{p5.1} in hand, we define the configuration space of singular connections to be 
\[
\A^{\alpha}_k(X, \Sigma, E):=A^{\alpha}_0 + \check{L}^2_k(X, \Lambda^1\otimes \g_E). 
\]
Denote by $\G(E)$ the $\check{L}^2_{k+1}$-gauge group. Then we have the moduli space of $\alpha$-singular instantons:
\[
\M^{\alpha}_{m, l}(X, \Sigma):= \left\{ A \in \A^{\alpha}_k(X, \Sigma, E): F^+_A = 0 \right\} / \G(E). 
\]
A priori, the moduli space depends on the cone-angle parameter $\nu$. We say a singular connection $A \in \A^{\alpha}(X, \Sigma, E)$ is reducible if the splitting $L \oplus L^*$ over $\nu(\Sigma)$ can be extended to all $X$ and respects $A$. Otherwise the $A$ is called irreducible. 

Now we restrict our attention to the following particular case. The $4$-manifold $X$ is a homology $S^1 \times S^3$, and the surface is an essentially embedded torus $\mathcal{T}$. We let $E$ be the trivial $\C^2$-bundle over $X$ with a trivial reduction near $\mathcal{T}$, in which case both the instanton and monopole numbers vanish. Due to the Chern-Weil formula \cite[Proposition 5.7]{KM93}, an anti-self-dual $\alpha$-singular connection $A$ on $E$ is flat, thus corresponds to an element in the character variety 
\begin{equation}
\chi^{\alpha}(X, \mathcal{T}):= \left\{ \varphi: \pi_1(X \backslash \Sigma) \to SU(2): \varphi(\mu_{\mathcal{T}}) \sim e^{2\pi i \alpha} \right\} / \Ad,
\end{equation}
where $\mu_{\mathcal{T}}$ is a meridian of $\mathcal{T}$. In fact, this assignment is a one-to-one correspondence between $\M^{\alpha}(X, \mathcal{T})$ and $\chi^{\alpha}(X, \mathcal{T})$. We shall denote $\chi^{\alpha, \Red}(X, \mathcal{T})$ for the subspace of reducible representations. We identify $U(1)$ as a standard maximal torus of $SU(2)$ by 
\begin{equation}\label{e5.4}
e^{i \theta} \longmapsto 
\begin{pmatrix}
e^{i \theta} & 0 \\
0 & e^{-i\theta}
\end{pmatrix}. 
\end{equation}
To define the singular Furuta-Ohta invariant, we first introduce the following non-degeneracy notion on $\chi^{\alpha}(X, \mathcal{T})$. 

\begin{dfn}\label{d5.2}
We say a reducible $\alpha$-representation $[\varphi] \in \chi^{\alpha, \Red}(X, \mathcal{T})$ is non-degenerate if $H^1(X \backslash \mathcal{T}; \C_{\varphi^2}) = 0$. We say $\chi^{\alpha, \Red}(X, \mathcal{T})$ is non-degenerate if every reducible $\alpha$-representation is non-degenerate. 
\end{dfn}

\begin{rem}
As pointed out by Echeverria \cite{E19}, the non-degeneracy condition is equivalent to the non-vanishing of the Alexander polynomial of $\mathcal{T}$ evaluating on the corresponding representation. 
\end{rem}

The singular Furuta-Ohta invariant introduced by Echeverria \cite{E19} is defined as follows. Let $(X, \mathcal{T})$ be a pair as above. For each $\alpha \in (0, 1/2)$ with $\chi^{\alpha, \Red}(X, \mathcal{T})$ non-degenerate, one can perturb the moduli space so that the irreducible moduli space $\M^{\alpha, *}_{\sigma}(X, \mathcal{T})$ is a compact $0$-manifold, where $\sigma$ stands for a generic perturbation. A choice of a generator in $H^1(X; \Z)$ provides us a homology orientation, which can be used to orient $\M^{\alpha,*}_{\sigma}(X, \mathcal{T})$. Then the singular Furuta-Ohta invariant is defined to be the signed count:
\begin{equation}
\lambda_{FO}(X, \mathcal{T}, \alpha):= \# \M^{\alpha, *}_{\sigma}(X, \mathcal{T}),
\end{equation}
which is independent of the choice of generic perturbations. Note that for each $\alpha \in (0, 1/2)$ satisfying the non-degeneracy condition we need an orbifold metric $g^{\nu}$ to apply the Fredholm package of instanton theory. The moduli space $\M^{\alpha}(X, \mathcal{T})$ thus depends on the choice of the cone-angle $2\pi/\nu$. We will show, later in this section, that the singular Furuta-Ohta invariant is independent of such a choice, which provides a positive answer to this dependence question raised by Echeverria \cite{E19}.  

\subsection{\em Instantons over Manifolds with Cylindrical End}\label{ss4.2} \hfill
 
\vspace{3mm}

We review the structure theorem of degree zero instantons over a homology $D^2 \times T^2$ with a cylindrical end modeled on $[0, \infty) \times T^3$ deduced by the author in \cite{M2}. 

Let $M_o = M \cup [0, \infty) \times T^3$ be a $4$-manifold with cylindrical end, where $M$ is a compact $4$-manifold with the same integral homology as that of $D^2 \times T^2$. We fix a metric $g$ on $M_o$ of the form $g|_{[0, \infty) \times T^3} = dt^2 + h$, where $h$ is a flat metric on $T^3$. We let $E$ be the trivial $\C^2$-bundle over $M_o$. Let $\A_k$ be the space of $L^2_{k, loc}$ $SU(2)$-connections on $E$, and $\G_{k+1}$ the space of $L^2_{k+1, loc}$ gauge transformations. The action and energy of a connection $A \in \A_k$ are defined respectively to be 
\begin{equation}
\kappa(A):= \frac{1}{8\pi^2} \int_{M_o} \tr(F_A \wedge F_A) \qquad \mathcal{E}(A):= \int_{M_o} |F_A|^2. 
\end{equation}
The unperturbed moduli space of degree zero instantons over $Z$ is given by 
\begin{equation}
\M(M_o):= \{A \in \A_k: F^+_A = 0, \; \kappa(A) = 0, \; \mathcal{E}(A) < \infty\}/\G_{k+1}. 
\end{equation}
Let $\mu > 0$ be a weight. As in \cite[Subsection 2.2]{M2}, one can consider a Banach space $\mathscr{P}_{\mu}(M_o)$ of holonomy perturbations on $\A_k$ which stand as $\G_{k+1}$-equivariant maps
\begin{equation}
\sigma: \A_k \longrightarrow L^2_{k, \mu}(Z, \Lambda^+_{\mathfrak{su}(2)}). 
\end{equation}
From its construction, we see that the differential $D\sigma|_A \in \mathcal{P}_{\mu}$, the perturbation space considered in (\ref{e4.8}), for each $A \in \A_k$. The $\sigma$-perturbed moduli space is denoted by $\M_{\sigma}(M_o)$. Let $\chi(T^3):= \Hom(\pi_1(T^3), SU(2))/\Ad$ be the $SU(2)$-character variety of the $3$-torus. In \cite[Theorem 1.7]{M2} the author deduced an asymptotic map 
\begin{equation}\label{e4.8}
\partial_+: \M_{\sigma}(M_o) \longrightarrow \chi(T^3)
\end{equation}
by showing the limit $[A|_{\{t\} \times T^3}]$ exists as $t \to \infty$ for all $[A] \in \M_{\sigma}(M_o)$ and $\sigma \in \mathscr{P}_{\mu}(M_o)$. 

Inspired by the treatment in \cite{H97}, it would be easier to work with the double cover of $\chi(T^3)$ for our purpose. We let $\tilde{\chi}(T^3):= \Hom(\pi_1(T^3), U(1))$ the $U(1)$-character variety of $T^3$, which is a double cover of $\chi(T^3)$ branched along the eight central representations. Then $\tilde{\chi}(T^3)$ is equivalent to the space of flat $U(1)$-connections on the trivial line bundle of $T^3$ modulo $U(1)$-gauge transformations. We note that the identification (\ref{e5.4}) of $U(1)$ in $SU(2)$ gives us an identification $i\R \oplus \C \simeq \mathfrak{su}(2)$ by 
\begin{equation}\label{e5.10}
(v, z) \longmapsto 
\begin{pmatrix}
v & z \\
-\bar{z} & -v
\end{pmatrix}. 
\end{equation}
We consider the restricted configuration space 
\begin{equation}
\tilde{\A}_k:=\left \{ A \in \A_k : (A|_{\{t\} \times T^3} - d)_{\C} \to 0 \text{ in } L^2_{k-1/2}(Y) \right\},
\end{equation}
where $(A|_{\{t \times Y\}} - d)_{\C}$ means the projection to the $\C$ component with respect to the decomposition (\ref{e5.10}). Then $\tilde{A}_k$ consists of connections that are asymptotically abelian. We also consider the restricted gauge group
\begin{equation}
\tilde{\G}_{k+1}:=\left\{ u \in \G_{k+1} : (du_t \cdot u^{-1}_t)_{\C}  \to 0  \text{ in } L^2_{k-1/2}(Y) \right\},
\end{equation}
where $u_t = u|_{\{t\} \times Y }$. Then for each holonomy perturbation, we get a perturbed moduli space 
\begin{equation}
\widetilde{\M}_{\sigma}(M_o) :=\{ A \in \tilde{\A}_k: F^+_A = \sigma(A), \; \kappa(A) = 0, \; \mathcal{E}(A) < \infty\} / \tilde{\G}_{k+1}. 
\end{equation}
The corresponding asymptotic map is denoted by $\tilde{\partial}_+: \widetilde{\M}_{\sigma}(M_o) \longrightarrow \tilde{\chi}(T^3)$. 


\begin{lem}\label{l5.4}
$\widetilde{\M}_{\sigma}(M_o)$ is a double cover of $\M_{\sigma}(M_o)$ branched along central instantons. 
\end{lem}

\begin{proof}
Since $\chi(T^3)$ consists of abelian flat connections, each instanton $[A] \in \M_{\sigma}(M_o)$ admits representatives in $\tilde{\A}_k$. Since $\tilde{\G}_{k+1}$ is a subgroup of $\G_{k+1}$, we get a surjective map $p: \widetilde{\M}_{\sigma}(M_o) \to \M_{\sigma}(M_o)$.

Let's denote by $\G'_{k+1} \subset \G_{k+1}$ the space of gauge transformations that preserve $\tilde{A}_k$. We claim that $\tilde{\G}_{k+1}$ is an index two subgroup of $\G'_{k+1}$. Note that elements in $\tilde{\G}_{k+1}$ take value in $U(1)$ asymptotically along the cylindrical end. To preserve $\tilde{A}_k$,  elements in $\G'_{k+1}$ take value in the normalizer of $U(1)$ asymptotically. The normalizer of $U(1)$ is $Pin(2) = U(1) \cup j\cdot U(1) \subset SU(2)$. Thus for each element in $\G'_{k+1}$ takes the form of either $u$ or $j\cdot u$ for some $u \in \tilde{G}_{k+1}$. 

When $[A] \in \M_{\sigma}(M_o)$ is non-central, the stabilizer of $A$ is either $\Z/2$ or $U(1)$. Thus $[j \cdot A] \neq [A]$ in $\widetilde{\M}_{\sigma}(M_o)$. When $[A] \in \M_{\sigma}(M_o)$ is central, $j \cdot A = A$. This finishes the proof. 
\end{proof}

With the help of \autoref{l5.4}, the structural theorem for $\M_{\sigma}(M_o)$ in \cite{M2} can be adapted to the double cover $\widetilde{\M}_{\sigma}(M_o)$ directly.

\begin{thm}[\label{t5.5}{\cite[Theorem 1.9]{M2}}]
With respect to generic small holonomy perturbations $\sigma \in \mathscr{P}_{\mu}(M_o)$, the moduli space $\widetilde{\M}_{\sigma}(M_o)$ is a compact smooth oriented stratified space of the following structure.
\begin{enumerate}[label=(\alph*)]
\item The reducible locus $\widetilde{\M}^{\Red}_{\sigma}(M_o)$ is diffeomorphic to a $2$-torus. 
\item The irreducible locus $\widetilde{\M}^*_{\sigma}(M_o)$ is a smooth oriented $1$-manifold of finitely many components, each of which is either diffeomorphic to $S^1$ or $(0, 1)$. 
\item The ends of the arcs in $\widetilde{\M}^*_{\sigma}(M_o)$ lie in $\widetilde{\M}^{\Red}_{\sigma}(M_o)$ but stay away from the central instantons. Near each end $[A] \in \widetilde{\M}^{\Red}_{\sigma}(M_o)$, the moduli space $\widetilde{\M}_{\sigma}(M_o)$ is modeled on a neighborhood of $0$ in $\mathfrak{o}^{-1}(0) \subset \R^2 \oplus \R_+$, where 
\[
\begin{split}
\mathfrak{o}: \R^2 \oplus \R_+ \longrightarrow \R \\
(x_1, x_2, r) & \longmapsto (x_1 + i x_2) \cdot r
\end{split}
\]
is the Kuranishi obstruction map. 
\item The asymptotic map $\tilde{\partial}_+: \widetilde{\M}^*_{\sigma}(M_o) \to \tilde{\chi}(T^3)$ is $C^2$ transverse to a given submanifold. 
\item The image of the irreducible locus $\widetilde{\M}^*_{\sigma}(M_o)$ under the asymptotic map $\tilde{\partial}_+$ misses central flat connections in $\tilde{\chi}(T^3)$. 
\end{enumerate}
\end{thm} 

The end points $[A] \in \widetilde{\M}_{\sigma}^{\Red}(M_o)$ appeared in (c) of \autoref{t5.5} is referred to as bifurcation points. In the following, we shall see that these points are characterized by the cohomology of the deformation complex. The deformation complex at $[A]$ is given by 
\begin{equation}\tag{$E^{SU(2)}_{A, \mu, \sigma}$}\label{ems}
 \hat{L}^2_{k+1, \mu}(M_o, \Lambda^0_{\mathfrak{su}(2)}) \xrightarrow{-d_{A}} \hat{L}^2_{k, \mu}(M_o, \Lambda^1_{\mathfrak{su}(2)}) \xrightarrow{d^+_{A, \sigma}} L^2_{k-1, \mu}(M_o, \Lambda^+_{\mathfrak{su}(2)}),
\end{equation}
where $\hat{L}^2_{k, \mu}(M_o, \Lambda^0_{\mathfrak{su}(2)})$ means $\mathfrak{su}(2)$-valued functions on $M_o$ that are asymptotically to constant functions with respect to the norm $L^2_{k+1, \mu}$, and similarly for $\hat{L}^2_{k, \mu}(M_o, \Lambda^1_{\mathfrak{su}(2)})$. Since $[A]$ is reducible, we write $A = A_{\mathfrak{l}} \oplus A_{\mathfrak{l}}^*$ with respect to the splitting $E= \underline{\C} \oplus \underline{\C}$ , and $A^{\dagger} = A^{\otimes 2}_{\mathfrak{l}}$. Then $E^{SU(2)}_{A, \mu, \sigma}$ splits into the direct sum of 
\begin{equation}\tag{$E^{U(1)}_{\mu, \sigma}$}\label{emu}
\hat{L}^2_{k+1, \mu}(M_o, \Lambda^0_{i\R}) \xrightarrow{-d} \hat{L}^2_{k, \mu}(M_o, \Lambda^1_{i\R}) \xrightarrow{d^+_{\sigma}} L^2_{k-1, \mu}(M_o, \Lambda^+_{i\R}) 
\end{equation}
and 
\begin{equation}\tag{$E_{A^{\dagger}, \mu}$}\label{emc}
L^2_{k+1, \mu}(M_o, \Lambda^0_{\C}) \xrightarrow{-d_{A^{\dagger}}} L^2_{k, \mu}(M_o, \Lambda^1_{\C}) \xrightarrow{d^+_{A^{\dagger}, \sigma}} L^2_{k-1, \mu}(M_o, \Lambda^+_{\C}).
\end{equation}
Then $[A]$ is a bifurcation point if and only if 
\begin{equation}\label{e4.14}
H^1(E_{A^{\dagger}, \mu}(M_o)) = \C. 
\end{equation}
It was proved in \cite{M2} that there are only finite many bifurcation points with respect to small generic perturbations in  $\mathscr{P}_{\mu}$. Moreover all of them are away from the central instantons in $\widetilde{\M}^{\Red}(M_o)$. Let's denote the set of bifurcation points in $\widetilde{M}^{\Red}_{\sigma}(M_o)$ by $\widetilde{\Bf}(M_o, \sigma)$. 

The orientability of the moduli space $\widetilde{\M}_{\sigma}(M_o)$ can be worked out the same as that in \cite[Section 9]{H94}. In particular, the involution on the double covers $\widetilde{\M}_{\sigma}(M_o)$ and $\tilde{\chi}(T^3)$ are not orientation-preserving. Since the bifurcation points consist of the ends of the oriented irreducible locus $\widetilde{\M}^*_{\sigma}(M_o)$, one can assign a sign given by the boundary orientation to each bifurcation points. This gives us a well-defined count $\# \widetilde{\Bf}(M_o, \sigma)$, which will play an important role in the proof of \autoref{t1.1}. In practice the sign is determined as follows. 

When $[A]$ is not a central instanton, one can choose small $\mu > 0$ so that $E^{SU(2)}_{A, \mu}$ is a Fredholm complex. Since we are using holonomy perturbations which fix central instantons, we may choose a small neighborhood $\mathcal{O}$ of the central flat connections in $\tilde{\chi}(T^3)$ so that 
\[
\widetilde{\M}^{\Red}_{\sigma}(M_o, \mathcal{O}^c):= \tilde{\partial}_+^{-1}(\mathcal{O}^c) \cap \widetilde{\M}^{\Red}_{\sigma}(M_o)
\]
contains no central instantons, where $\mathcal{O}^c = \tilde{\chi}(T^3) \backslash \mathcal{O}$ is the complement of $\mathcal{O}$. Then one can choose small $\mu > 0$ so that $E^{SU(2)}_{A, \mu}$ is Fredholm for all $[A] \in \widetilde{\M}^{\Red}_{\sigma}(M_o, \mathcal{O}^c)$. Since $E_{A^{\dagger}, \mu}$ is complex, we get a canonical trivialization on its determinant line bundle. The trivialization of $E^{U(1)}_{\mu}$ is given by the homology orientation of $M_o$. Thus we get a preferred trivialization of the determinant line $\det E^{SU(2)}_{A, \mu}$ over $\widetilde{\M}^{\Red}(M_o, \mathcal{O}^c)$. 

Now suppose $[A] \in \widetilde{\Bf}(M_o, \sigma)$ is a bifurcation point. Let's write $H^i_A$ for the $i$-th cohomology of $E^{SU(2)}_{A, \mu, \sigma}$. Then 
\[
H^0_A= i\R, \quad  H^1_A = i\R \oplus i\R \oplus \C, \quad H^2_A = \C. 
\]
We fix an ordered real basis for these spaces that agree with the trivialization of the determinant line above. Since $\Stab_A = U(1)$, then $U(1)$ acts on the complex part of $H^1_A$ and $H^2_A$ with weight $2$. One applies the Kuranishi argument to obtain a $U(1)$-equivariant map $\mathfrak{q}_A: H^1_A \to H^1_A$ so that the moduli space near $[A]$ is modeled on the $U(1)$-quotient of the zero set of the obstruction map 
\begin{equation}\label{e4.15}
\mathfrak{o}_A(a) =\Pi_A \left( F^+_{A+a+\mathfrak{q}_{A}(a)} - \sigma(A + a + \mathfrak{q}_A(a)) \right), 
\end{equation}
where $\Pi_A$ is the $L^2_{\mu}$-projection onto $H^2_A$. As proved in \cite[Proposition 5.11]{M2}, the lowest order term of $\mathfrak{o}_A$ takes the form $f_0(x_1, x_2) \cdot z$, where $f_0: i\R \oplus i\R \to \C$ a function with non-vanishing first order term sending $(0, 0)$ to $0$. Since the irreducible locus near $[A]$ is given by the $U(1)$-quotient of $\{(0, 0, z) \in \mathfrak{o}_A^{-1}(0)\}$, we assign $+1$ to $[A]$ if $f_0$ is orientation-preserving and $-1$ if $f_0$ is orientation-reversing. 

\subsection{\em A Reformulation of $\lambda_{FO}$}\label{ss4.3} \hfill
 
\vspace{3mm}

We are going to deduce another formulation of the singular Furuta-Ohta invariant using a gluing argument based on \autoref{t5.5}. 

Let $X$ be an integral homology $S^1 \times S^3$ with an essentially embedded torus $\mathcal{T}$. We also fix a primitive class $1_X \in H^1(X; \Z)$ serving the role of homology orientation. Let $M = X \backslash \nu(\mathcal{T})$ be the complement of a tubular neighborhood of $\mathcal{T}$. From the Mayer-Vietoris sequence, we see that $M$ is a homology $D^2 \times T^2$ with $\partial M = T^3$. We first choose a basis of curves $\{\mu, \lambda, \gamma\}$ on $\partial \nu(\mathcal{T}) = - \partial M$ as follows:
\begin{itemize}
\item $[\mu]$ generates $\ker H_1(\partial \nu(\mathcal{T}); \Z) \to H_1(\nu(\mathcal{T}); \Z)$;
\item $[\lambda]$ generates $\ker H_1(\partial \nu(\mathcal{T}); \Z) \to H_1(M; \Z)$;
\item $1_X \cdot [\gamma] = 1$. 
\end{itemize}
Up to isotopy and orientation-reversing, $\mu$ and $\lambda$ are uniquely determined. We will not specify the isotopy class of $\gamma$, but simply make a choice.

To a flat $U(1)$-connection $B=d+b$ on $T^3$, we assign coordinates
\[
x(B) = \frac{1}{2\pi i }\int_{\mu} b, \quad y(B) = \frac{1}{2\pi i }\int_{\lambda} b, \quad z(B) = \frac{1}{2\pi i }\int_{\gamma} b.
\]
The holonomies of $B$ around $\mu, \lambda, \gamma$ are given respectively by 
\[
\Hol_{\mu} B  =e^{-2\pi i x(B)}, \quad \Hol_{\lambda} B  =e^{-2\pi i y(B)}, \quad  \Hol_{\gamma} B  =e^{-2\pi i z(B)}.
\]
Let $\mathcal{C}_{T^3} = \left\{(x, y, z): x, y, z \in [-1/2, 1/2] \right\}$ be the fundamental cube so that $\tilde{\chi}(T^3)$ is identified with the quotient of $\mathcal{C}_{T^3}$ by identifying opposite faces of the cube. The $SU(2)$-character variety is identified as $\chi(T^3) = \tilde{\chi}(T^3)/ \sim$ under the antipodal relation $ \pmb{x} \sim - \pmb{x}$. Given $x_0 \in [-1/2, 1/2]$, we get a $2$-torus $\widetilde{T}_{x_0} \subset \tilde{\chi}(T^3)$ given by the quotient of the plane in $\mathcal{C}_{T^3}$ with fixed $x$-coordinate $x_0$. We write $ T_{x_0} = \widetilde{T}_{x_0}/ \sim$ for its quotient in $\chi(T^3)$. We note that when $x_0 \in \{-1/2, 1/2\}$, $\widetilde{T}_{x_0}$ is a double cover of $T_{x_0}$ branched along central connections. When $x_0 \in (-1/2, 1/2)$, $\widetilde{T}_{x_0}$ is diffeomorphic to $T_{x_0}$ as a $2$-torus. 

\begin{lem}\label{l5.6}
We continue with the notations as above. Let $\alpha \in (0, 1/2)$ be a holonomy parameter satisfying that $\chi^{\alpha, \Red}(X, \mathcal{T})$ is non-degenerate. Then with respect to generic small holonomy perturbations $\sigma$, we have 
\[
\lambda_{FO}(X, \mathcal{T}, \alpha) = \# \M^*_{\sigma}(M_o) \cap \partial_+^{-1}( T_{\alpha}). 
\]
\end{lem}

\begin{proof}
Let's write $N_o = (-\infty, 0] \times T^3 \cup \nu(\mathcal{T})$. Given $\alpha \in (0, 1/2)$, we put a metric $g^{\nu}$ on $X$ with cone angle $2\pi/\nu$ along $\mathcal{T}$ satisfying \autoref{p5.1} and cylindrical near $\partial \nu(\mathcal{T})$. Since we are working with instantons of topological action $0$, the neck-stretching process causes no lost of action. Recall in \autoref{t5.5}, the asymptotic map $\partial_+: \M^*_{\sigma}(M_o) \to \chi(T^3)$ can be made to be transverse to any given submanifold in $\chi(T^3)$ and misses the central flat connections. Then when the length $T$ of the "neck" is long enough, each irreducible $\alpha$-singular instanton $[A_T]$ on $(X_T, \mathcal{T})$ is glued by one irreducible instanton $[A_M] \in \M^*_{\sigma}(M_o)$ and one reducible $\alpha$-singular instanton on $[A_N] \in \M^{\alpha, \Red}(N_o, \mathcal{T})$ whose asymptotic values agree in $\chi(T^3)$. Note that $N_o$ is simply the completion of the product $D^2 \times T^2$. The unperturbed moduli space of reducible $\alpha$-singular flat connections is smoothly cut out due to the vanishing of $b^+$. So we don't need to perturb the moduli space further to achieve transversality. Thus the singular Furuta-Ohta invariant can be expressed as the count of the fiber product:
\[
\lambda_{FO}(X, \mathcal{T}, \alpha) = \# \M^*_{\sigma}(M_o) \times_{(\partial_+, \partial_-)} \M^{\alpha, \Red}(N_o, \mathcal{T}).
\]
Now it remains to identify $\partial_+(\M^{\alpha, \Red}(N_o, \mathcal{T}))$ with $T_{\alpha} \subset \chi(T^3)$. Since $\M^{\alpha, \Red}(N_o, \mathcal{T})$ consists of flat connections whose holonomy around the meridian of $\mathcal{T}$ is $e^{-2\pi i \alpha}$, the image under the asymptotic map is given by restricting to $\partial \nu(T)$ which is $T_{\alpha}$ by construction. 
\end{proof}

\begin{rem}
This formulation of the singular Furuta-Ohta invariant implies that it's independent of the choice of cone angles $2\pi/\nu$ of metrics along the singular surface $\mathcal{T}$.  
\end{rem}

In the end, we remark on the definition of the usual Furuta-Ohta invariant for an integral homology $S^1 \times S^3$ \cite{FO93}. Just like the non-degeneracy condition for defining singular Furuta-Ohta invariants, the corresponding non-degeneracy condition for the non-singular case is the following:
\begin{equation}\label{e5.16}
H^1(X; \C_{\varphi}) = 0 \text{ for all non-trivial representation } \varphi: \pi_1(X) \to U(1). 
\end{equation}
When this non-degeneracy condition is satisfied, the Furuta-Ohta invariant for a pair $(X, 1_X)$ is defined to be 
\begin{equation}
\lambda_{FO}(X):=\frac{1}{4} \# \M^*_{\sigma}(X)
\end{equation}
as a quarter of the signed count of irreducible instantons over the trivial bundle of $X$ under small generic perturbations. By the same neck-stretching argument of \autoref{l5.6}, one can show that
\begin{equation}\label{e5.18}
\lambda_{FO}(X) = \frac{1}{4} \# \M^*_{\sigma}(M_o) \cap \partial^{-1}_+(T_0) = \frac{1}{8} \#\widetilde{\M}^*_{\sigma}(M_o) \cap \tilde{\partial}^{-1}_+(\widetilde{T}_0),
\end{equation}
since $\widetilde{T}_0$ is a double cover of $T_0$ and the image of $\widetilde{\M}^*_{\sigma}(M_o)$ under the asymptotic map misses the branching points. 

\subsection{\em A Surgery Formula}\label{ss4.4} \hfill
 
\vspace{3mm}

In this section we prove the surgery formula \autoref{t1.10} using the reformulation of the singular Furuta-Ohta invariant. Given the pair $(X, \mathcal{T})$ of an essentially embedded torus in a homology $S^1 \times S^3$, the $1/q$-surgery of $X$ along $\mathcal{T}$ is given by 
\[
X_{1/q}(\mathcal{T}) = X \backslash \nu(\mathcal{T}) \cup_{\psi_q} D^2 \times T^2,
\]
where $\psi_q: \partial D^2 \times T^2 \to \partial \nu(\mathcal{T})$ is given by the matrix
\[
\psi_q = 
\begin{pmatrix}
1 & 0 & 0 \\
q & 1 & 0 \\
0 & 0 & 1
\end{pmatrix}
\]
under the basis of curves $\{\mu, \lambda, \gamma\}$ chosen as in \autoref{ss4.3}. We are going to prove the following surgery formula.

\begin{prop}\label{p4.8}
Let $\alpha \in (0, 1/2)$ be a non-degenerate parameter for the pairs $(X, \mathcal{T})$, $(X_1(\mathcal{T}), \mathcal{T}_1)$, and $(X_0(\mathcal{T}), \mathcal{T}_0)$. Then 
\[
\lambda_{FO}(X_1(\mathcal{T}), \mathcal{T}_1, \alpha) = \lambda_{FO}(X, \mathcal{T}, \alpha) +\lambda_{FO}(X_0(\mathcal{T}), \mathcal{T}_0, \alpha). 
\]
\end{prop}

Note that with respect to the framing of $\mathcal{T}_q$ in $X_{1/q}(\mathcal{T})$, performing $1$-surgery and $0$-surgery gives us $X_{1/(q+1)}(\mathcal{T})$ and $X_0(\mathcal{T})$ respectively. Thus \autoref{t1.10} is proved by applying \autoref{p4.8} repetitively. 

\begin{proof}[Proof of \autoref{p4.8}]
For $y_0 \in (-1/2, 1/2)$, we denote by $\tilde{S}_{y_0}$ the $2$-torus in $\tilde{\chi}(T^3)$ with fixed $y$-coordinate $y_0$. Then the neck-stretching argument tells us that 
\begin{equation}
\lambda_{FO}(X_0(\mathcal{T}), \mathcal{T}_0, \alpha)  = \# \widetilde{\M}^*_{\sigma}(M_o) \cap \tilde{\partial}_+^{-1}(\tilde{S}_{\alpha}). 
\end{equation}
We denote by $\tilde{P} \subset \tilde{\chi}(T^3)$ the $2$-torus given by $\{ x + y = \alpha\}$. Then 
\begin{equation}
\lambda_{FO}(X_1(\mathcal{T}), \mathcal{T}_1, \alpha)  = \# \widetilde{\M}^*_{\sigma}(M_o) \cap \tilde{\partial}_+^{-1}(\tilde{P}_{\alpha}). 
\end{equation}
Denote by $J$ the involution on $\tilde{\chi}(T^3)$ given by $J \cdot \pmb{x} = - \pmb{x}$. Since the asymptotic map intertwines with the involutions, we write $J$ for the involution on $\widetilde{\M}^*_{\sigma}(M_o)$ as well. Then we know 
\begin{equation}\label{e4.21}
\# \widetilde{\M}^*_{\sigma}(M_o) \cap \tilde{\partial}_+^{-1}( \tilde{S}_{\alpha})  =\# J \widetilde{\M}^*_{\sigma}(M_o) \cap \tilde{\partial}_+^{-1}(J \tilde{S}_{\alpha}) = \# \widetilde{\M}^*_{\sigma}(M_o) \cap \tilde{\partial}_+^{-1}(J \tilde{S}_{\alpha}) 
\end{equation}
since $J$ reverse the orientation for both $\widetilde{\M}^*_{\sigma}(M_o)$ and $\tilde{\partial}_+^{-1}( \tilde{S}_{\alpha})$, and $J \tilde{S}_{\alpha} \cap \tilde{S}_{\alpha} = \varnothing$. Note that the union 
\begin{equation}\label{e5.22}
\tilde{T}_{\alpha} \cup J \tilde{T}_{\alpha} \bigcup \tilde{S}_{\alpha} \cup J \tilde{S}_{\alpha} \bigcup  -\tilde{P}_{\alpha} \cup -J\tilde{P}_{\alpha} 
\end{equation}
is null-homologous in $\tilde{\chi}(T^3) \backslash \tilde{\mathcal{C}}$, where $\tilde{\mathcal{C}}$ is the set of the eight central representations. Thus the union (\ref{e5.22}) bounds a $3$-complex in $\tilde{\chi}(T^3) \backslash \tilde{\mathcal{C}}$, say $\tilde{\mathcal{V}}$. We can use \autoref{t5.5} to pick a generic small perturbation so that $\tilde{\partial}_+|_{\tilde{\M}^*}$ is transverse to $\tilde{\mathcal{V}}$. Then $\widetilde{\M}^*_{\sigma}(M_o) \cap \tilde{\partial}^{-1}_+(\tilde{\mathcal{V}})$ is an oriented $1$-manifold with boundary, and the surgery formula follows from the observation that 
\[
\# \partial \left(\widetilde{\M}^*_{\sigma}(M_o) \cap \tilde{\partial}^{-1}_+(\tilde{\mathcal{V}}) \right) = 0.
\]
\end{proof}

\section{\large \bf Equivariant Torus Signature}\label{s6}

In this section, we will prove the equivalence of the signature invariants. Let's recap the settings to start with. Let $(X, \mathcal{T})$ be the pair consisting of an essentially embedded torus in an integral homology $S^1 \times S^3$. We decompose $X = M \cup \nu(\mathcal{T})$ into a homology $D^2 \times T^2$ with a tubular neighborhood of $\mathcal{T}$. We fix a primitive class $1_X \in H^1(X; \Z)$. A basis of curves $\{\mu, \lambda, \gamma\}$ are chosen as in \autoref{ss4.3} on $\partial \nu(\mathcal{T})=-\partial M$, where the isotopy classes of $\mu$ and $\lambda$ are determined. Such a basis provides us with a framing $\nu(\mathcal{T}) \simeq D^2 \times T^2$. Then we can perform $0$-surgery along $\mathcal{T}$ to get $V = M \cup_{\psi} D^2 \times T^2$ with the gluing map $\psi: \partial D^2 \times T^2 \to \partial M$ specified by 
\[
\psi = 
\begin{pmatrix}
0 & 1 & 0 \\
1 & 0 & 0 \\
0 & 0 & 1
\end{pmatrix}. 
\]
Note that the gluing map preserves the third curve $\gamma$. So the surgery is well-defined despite of the ambiguity of choosing $\gamma$. Either by Mayer-Vietoris sequence or a geometric argument given in \cite[Section 3.1]{R20}, $1_X$ gives rise to a primitive class $1_V \in H^1(V;\Z)$. In the same paper, Ruberman also showed that $V$ is a cohomology $T^2 \times S^2$ endowed with a degree-$1$ map to $T^2 \times S^2$. We pick a smooth function $f: X \to S^1$ representing $1_X$ so that its restriction $f|_{\nu(\mathcal{T})}$ is the projection to the $S^1$-factor corresponding to $\gamma$. In this way, $f$ is invariant under the gluing map $\psi$, thus gives rise to a map, still denoted by $f$, on the $0$-surgered manifold $V$ representing $1_V$. We will abuse the notations $\mu, \lambda, \gamma$ for corresponding curves in either $X$ or $V$ when the context is clear. Given a representation $\varphi: \pi_1(V) \to U(1)$, we still write $\varphi$ for the representation of $\pi_1(Y_0)$ arising as the restriction of $\varphi$ to a cross-section $Y_0$ of $V$. 

\subsection{\em Proof of \autoref{t1.6}}\label{ss6.1} \hfill
 
\vspace{3mm}

To a holonomy parameter $\alpha \in (0, 1/2)$, we assign a representation $\varphi_{2\alpha}: \pi_1(V) \to U(1)$ by 
\begin{equation}
\varphi_{2\alpha}(\mu) = e^{-4\pi i \alpha} \quad \text{ and } \quad \varphi_{2\alpha}(\gamma) = 1. 
\end{equation}

\begin{lem}\label{l5.1}
$\chi^{\alpha, \Red}(X, \mathcal{T})$ is non-degenerate if and only if $H^*(E_{\varphi_{2\alpha}, z}(V)) = 0$ for all $\Rea z = 0$. 
\end{lem}

\begin{proof}
Let $[\varphi] \in \chi^{\alpha, \Red}(X, \mathcal{T})$ be an $\alpha$-reducible representation. Since $H_1(X \backslash \mathcal{T})$ and $H_1(V)$ are both freely generated by $[\mu]$ and $[\gamma]$. We can regard $\varphi^2$ as a representation from $\pi_1(V)$ to $U(1)$. By our convention $\varphi^2(\mu) = e^{-4\pi i \alpha}$. Let $B_{\varphi}$ and $B_{\alpha}$ be flat connections corresponding to $\varphi^2$ and $\varphi_{2\alpha}$ respectively. From their holonomies one can tell $B_{\alpha} - B_{\varphi} = z \cdot df $ for some $\Rea z = 0$. Since $H^1(D^2 \times T^2; \C_{\varphi^2}) = 0$ for all $[\varphi] \in \chi^{\alpha, \Red}(X, \mathcal{T})$, the Mayer-Vietoris sequence tells us that 
\[
H^1(X\backslash \mathcal{T}; \C_{\varphi^2}) \simeq H^1(V; \C_{\varphi^2}) = H^1(E_{\varphi_{2\alpha}, z}(V)). 
\]
As $[\varphi]$ ranges over $\chi^{\alpha, \Red}(X, \mathcal{T})$, $z$ ranges over $i[0, 2\pi)$. The proof is completed after combining with the fact that $H^0(E_{\varphi_{2\alpha, z}}(V)) = 0$ and $\chi(E_{\varphi_{2\alpha}, z}(V)) = 0$. 
\end{proof}

Now let $\alpha$ be a non-degenerate holonomy parameter. We further assume that $V$ satisfies the non-degenerate condition (\ref{e5.16}) in order to get well-defined Furuta-Ohta invariant. Then \autoref{l5.6} and (\ref{e5.18}) tells us that 
\begin{equation}
\lambda_{FO}(X, \mathcal{T}, \alpha) - 8 \lambda_{FO}(X) = \# \widetilde{\M}^*_{\sigma}(M_o) \cap \partial_+^{-1}( \widetilde{T}_{\alpha}) - \#\widetilde{\M}^*_{\sigma}(M_o) \cap \tilde{\partial}^{-1}_+(\widetilde{T}_0). 
\end{equation}
Let $\widetilde{P}_{[0, \alpha]} = \{(x, y, z) \in \mathcal{C}_{T^3}: x \in [0, \alpha]\} /\sim$ be the product torus in $\widetilde{\chi}(T^3)$ with boundary $\partial \widetilde{P}_{[0, \alpha]} =  -\widetilde{T}_0 \cup \widetilde{T}_{\alpha}$. It follows from \autoref{t5.5} that the closure of $\widetilde{\M}^*_{\sigma}(M_o) \cap \tilde{\partial}^{-1}_+(\widetilde{P}_{[0, \alpha]})$ in the moduli space $\widetilde{\M}_{\sigma}(M_o)$ is an oriented compact manifold whose boundary consists of 
\[
\widetilde{\M}^*_{\sigma}(M_o) \cap \partial_+^{-1}( \widetilde{T}_{\alpha})  \bigcup -\widetilde{\M}^*_{\sigma}(M_o) \cap \tilde{\partial}^{-1}_+(\widetilde{T}_0) \bigcup \widetilde{\Bf}(M_o, \sigma)  \cap \tilde{\partial}^{-1}_+(\widetilde{P}_{[0, \alpha]}),
\]
where $\widetilde{\Bf}(M_o, \sigma)$ is the set of bifurcation points in $\widetilde{\M}^{\Red}_{\sigma}(M_o)$. Thus we conclude that 
\begin{equation}\label{e6.3}
\lambda_{FO}(X, \mathcal{T}, \alpha) - 8 \lambda_{FO}(X) = - \# \widetilde{\Bf}(M_o, \sigma)  \cap \tilde{\partial}^{-1}_+(\widetilde{P}_{[0, \alpha]}) . 
\end{equation}

Due to the non-degeneracy of $V$, the set of bifurcation points $\widetilde{\Bf}(M_o, \sigma)$ is away from the preimage of the torus $\widetilde{T}_0$ under the asymptotic map $\tilde{\partial}^{-1}_+$ with respect to perturbations $\sigma$ small enough. Thus we can find $\epsilon > 0$ sufficiently small so that 
\begin{equation}\label{e6.4}
\# \widetilde{\Bf}(M_o, \sigma)  \cap \tilde{\partial}^{-1}_+(\widetilde{P}_{[0, \alpha]})  = \# \widetilde{\Bf}(M_o, \sigma)  \cap \tilde{\partial}^{-1}_+(\widetilde{P}_{[\epsilon, \alpha]}) 
\end{equation}

Let $(A_t)_{t \in [0, 1]}$ be a path of flat $SU(2)$-connections on $M_o$ whose holonomies around $\mu$ and $\gamma$ are respectively $e^{-2\pi i \alpha(\epsilon + t(1-\epsilon))}$ and $1$. With respect to the splitting $E= \underline{\C} \oplus \underline{\C}$, we write $A_t = A_{\mathfrak{l}, t} \oplus A^*_{\mathfrak{l}, t}$ and $A_{\mathfrak{l}, t} = d+ a_{\mathfrak{l}, t}$ with $a_{\mathfrak{l}, t} \in \Omega^1(M_o, i\R)$. Let's denote by $A^{\dagger}_t = d + 2a_{\mathfrak{l}, t}$ the path of flat connections on the trivial bundle $\underline{\C}$, and by $B^{\dagger}_t = A^{\dagger}_t|_{\{0\} \times T^3}$ the restriction to a slice of the end. Then the coordinates of $B^{\dagger}_t$ in $\tilde{\chi}(T^3)$ are given respectively by 
\[
x(B^{\dagger}_t) = 2\alpha(\epsilon+t(1-\epsilon)), \quad y(B^{\dagger}_t) = 1, \quad z(B^{\dagger}_t) =1.
\]

\begin{lem}\label{l6.2}
The path $(A^{\dagger}_t)$ consists of $0$-regularly asymptotically flat connections. 
\end{lem}

\begin{proof}
Since $A^{\dagger}_t$ is flat, condition (a) in \autoref{d4.8} is automatic. Condition (b) follows from the fact that $\ker L_{B^{\dagger}_t, z}=H^0_{B^{\dagger}_{t, z}}(T^3) = H^1_{B^{\dagger}_{t, z}}(T^3)=0$ since $B^{\dagger}_{t, z}$ is not the product connection. To verify the third condition (c), Note that for $(t, z) \in [0, 1] \times S^1$, the operator 
\[
Q_{A^{\dagger}_{t, z}}(M_o): L^2_1 \longrightarrow L^2
\]
is always Fredholm, thus its index is unaltered with respect to varying $(t, z)$. In particular, by letting $t=0$ and $z=0$, the vanishing of the index follows from \autoref{p3.9} and the fact that $M_o$ is a homology $D^2 \times T^2$. 
\end{proof}

With \autoref{l6.2}, the periodic spectral flow of this path $(A^{\dagger}_t)$ is well-defined with respect to a generic small perturbations $\pi \in \mathcal{P}_{\mu}$. 

\begin{prop}\label{p6.3}
Let $(A^{\dagger}_t)$ be the path of flat connections on $M_o$ as above. Then 
\[
\# \widetilde{\Bf}(M_o, \sigma)  \cap \partial^{-1}_+(\widetilde{P}_{[\epsilon, \alpha]}) = 2 \widetilde{\Sf}(Q_{A^{\dagger}_t}(M_o, \pi))
\]
with respect to generic small perturbations $\sigma \in \mathscr{P}_{\mu}$ and $\pi \in \mathcal{P}_{\mu}$.
\end{prop}

\begin{proof}
Let $\sigma \in \mathcal{P}_{\mu}$ be a generic small holonomy perturbation satisfying \autoref{t5.5}. As discussed before, $[A] \in \widetilde{\M}^{\Red}_{\sigma}(M_o)$ is a bifurcation point if and only if $H^1(E_{A^{\dagger}, \mu}(M_o)) = \C$. Just like the case for end-periodic manifolds, rather than the ASD DeRham complex we consider the associated operator
\begin{equation}
Q^{\mu}_{A^{\dagger}, \sigma}(M_o) := \left(-d^*_{A^{\dagger}} + \mu/2 \cdot \iota_{d\beta}, d^+_{A^{\dagger}, \sigma} - \mu/2 \cdot (d\beta \wedge -)^+ \right): L^2_k \to L^2_{k-1},
\end{equation}
where $\beta: M_o \to \R_{\geq 0}$ is a non-negative function satisfying $\beta|_{\{t\} \times T^3} = t$. The function $\beta$ is used to define the weighted Sobolev spaces $L^2_{k, \mu}$ of forms over $M_o$. If we write $A^{\dagger}_{\mu} = A^{\dagger} - \mu/2 \cdot d\beta$, then $Q^{\mu}_{A^{\dagger}} = Q_{A^{\dagger}_{\mu}, \sigma}$. In this way, bifurcation points are characterized by $\ker Q^{\mu}_{A^{\dagger}, \sigma}(M_o) = \C$. As shown in \cite[Proposition 5.6]{M2}, one can find a path $\sigma_s$, $s \in [0, 1]$, of holonomy perturbations in $\mathscr{P}_{\mu}$ from $0$ to $\sigma$ such that the parametrized moduli space 
\[
\mathcal{Z}_I:= \bigcup_{s \in [0, 1]} \{s\} \times \left(\widetilde{\M}^{\Red}_{\sigma_s}(M_o) \cap \tilde{\partial}^{-1}_+(\widetilde{P}_{[\epsilon, \alpha]}) \right)
\]
is a product cobordism. Let's denote by $\mathcal{Z}_s := \{s\} \times \left(\widetilde{\M}^{\Red}_{\sigma_s}(M_o) \cap \tilde{\partial}^{-1}_+(\widetilde{P}_{[\epsilon, \alpha]}) \right)$ the $s$-slice.

For each $[A] \in \mathcal{Z}_I$ of the form $A=A_{\mathfrak{l}} \oplus A^*_{\mathfrak{l}}$, we get a connection $[A^{\dagger}]$ on the trivial bundle $\underline{\C}$ with $A^{\dagger} = A^{\otimes 2}_{\mathfrak{l}}$. We note that this assignment $[A] \to [A^{\dagger}]$ is two to one since we can twist $A_{\mathfrak{l}}$ by the flat connection $\xi_{\mathfrak{l}}$ on $M_o$ whose holonomies around $\mu$ and $\gamma$ are $1$ and $-1$ respectively, and then get $(A_{\mathfrak{l}} \otimes \xi_{\mathfrak{l}} )^{\otimes 2} = A_{\mathfrak{l}}^{\otimes 2}$. Due to the assumption that the holonomy of $A_{\mathfrak{l}}$ around the meridian $\mu$ is asymptotically $e^{-2\pi i \alpha_{\mathfrak{l}}}$ with $\alpha_{\mathfrak{l}} \in [0, \alpha] \subset [0, 1/2)$, twisting around the meridian with $-1$ is not allowed. Such an assignment gives rise to a family of connections $\mathcal{Z}^{\dagger}_I$ on $\underline{\C}$. For the brevity of notations, given $[A] \in \mathcal{Z}_s$ we write 
\[
A^{\dagger}(s) = A^{\dagger} - s/2 \cdot \mu d \beta
\]
and 
\[
Q_{A^{\dagger}}(s):= Q_{A^{\dagger}(s), \sigma_s} + \pi_s : L^2_k \longrightarrow L^2_{k-1}
\]
for some path of perturbations $\pi_s \in \mathcal{P}_{\mu}$ to be determined soon. Since $\ind Q_{A^{\dagger}(s), \sigma_s} = 0$ for all $s \in [0, 1]$, we can apply the transversality argument in \autoref{l3.2} to get a generic path of small perturbations $\pi_s \in \mathcal{P}_{\mu}$, with $\pi_1 = 0$, so that the set
\[
\mathcal{S}^{\dagger}_I:=\left\{ [A^{\dagger}] \in \mathcal{Z}^{\dagger}_s: \ker Q_{A^{\dagger}}(s) \neq 0 \right\} \subset \mathcal{Z}^{\dagger}_I
\]
is a compact submanifold of codimension-$2$. Moreover for each $[A^{\dagger}] \in \mathcal{S}^{\dagger}_I$, $\ker Q_{A^{\dagger}}(s) = \C$. Everything above can be made oriented once we choose a trivialization of the index bundle of the $SU(2)$ ASD DeRham operator over the configuration space of $SU(2)$-connections over $M_o$ exponentially asymptotic to flat connections as in the usual Yang-Mills theory \cite{D87}. The exponential convergence of instantons on $\widetilde{\M}^{\Red}_{\sigma_s}(M_o)$ was derived in \cite{M2}. Thus we get a well-defined count on the compact $0$-dimensional manifold $\partial \mathcal{S}^{\dagger}_I = \mathcal{S}^{\dagger}_1 \cup -\mathcal{S}^{\dagger}_0$. 

Converting from $\mathcal{S}^{\dagger}_I$ to $\mathcal{S}_I$ as $SU(2)$-connections, we see that $\mathcal{S}_1$ is precisely the set of bifurcation points $\widetilde{\Bf}(M_o, \sigma)  \cap \tilde{\partial}^{-1}_+(\widetilde{P}_{[\epsilon, \alpha]}) $. Since the assignment $[A] \to [A^{\dagger}]$ is two-to-one, we get 
\[
\# \widetilde{\Bf}(M_o, \sigma)  \cap \tilde{\partial}^{-1}_+(\widetilde{P}_{[\epsilon, \alpha]}) = \#\mathcal{S}_1 =  2\# \mathcal{S}^{\dagger}_1 = 2\# \mathcal{S}^{\dagger}_0,
\]
where $\pi = \pi_0$. So all that left is to identify $\# \mathcal{S}^{\dagger}_0$ with the periodic spectral flow $\widetilde{\Sf}(Q_{A^{\dagger}_t, \delta}(M_o, \pi))$. Note that the spectral flow counts points with non-trivial kernel on the cylinder $(t, e^z) \in [0, 1] \times S^1$ for the operators $Q_{A^{\dagger}_t, z, \pi}: L^2_k \to L^2_{k-1}$.

Now we show that the sign of each point in $\mathcal{S}^{\dagger}_0$ agrees with that in the definition of spectral flow. Since along each component of $\mathcal{S}^{\dagger}_I$ the dimension of $\ker Q$ is constant, the orientation transport is canonical. Allowing only small perturbations, the orientation transport is canonical as well along two perturbations. So we may assume $[A] \in \M^{\Red}(M_o)$ is a bifurcation point which gives a regular RAF connection $[A^{\dagger}] \in \mathcal{S}^{\dagger}_0$ in the first place, namely $Q_{A^{\dagger}}: L^2_k \to L^2_{k-1}$ is Fredholm with $\ker Q_{A^{\dagger}} = \C$ without appealing to the weight $\mu$ and perturbations $\sigma, \pi$. Under the notation of \autoref{ss4.2}, we have 
\[
H^0_A = i\R, \quad H^1_A=i\R \oplus i\R \oplus \C, \quad H^2_A = \C,
\]
where $i\R \oplus i\R = H^1(M; i\R)$ in $H^1_A$. Let $\{idg, idf\}$ be an ordered basis for $H^1(M;i\R)$ where $dg$ represents the dual of $\mu$, and $df$ represents the dual of $\gamma$. Let $A^{\dagger}_x = A^{\dagger} + x \cdot idg - z(x) \cdot df$, $x \in (-\epsilon, \epsilon)$, be a local spectral curve near $A^{\dagger}$. We may write $z(x) = r(x) + is(x)$ with $r(0) = 0$, $s(0)=0$ and $\dot{r}(0) \neq 0$. The contribution of $A^{\dagger}_x$ in the periodic spectral flow is given by the sign of $\dot{r}(0)$. We need to show this is the sign at $[A]$ when counting bifurcation points. 

Since $ \ker Q_{A^{\dagger}_x} =\C$ for all $x \in (-\epsilon, \epsilon)$, we can pick a path of non-zero $1$-forms $\phi^{\dagger}_x \in \ker Q_{A^{\dagger}_x}$. Let $A_x$ be the path of connections on the trivial $\C^2$-bundle given by $A^{\dagger}_x$ and $\phi_x = (0, \phi^{\dagger}_x)$ the $\mathfrak{su}(2)$-valued $1$-form associated to $\phi^{\dagger}_x$. We let $\{\phi_0, i\phi_0\}$ be an ordered basis for the $\C$-component of $H^1_A$. Denote by $\Pi_A: L^2_{k-1}(M_o, \Lambda^+_{\mathfrak{su}(2)}) \to H^2_A$ the $L^2$-projection. Then the second order term of $\mathfrak{o}_A$ in (\ref{e4.15}), with respect to the ordered basis $\{idg, idf, \phi_0, i\phi_0\}$ of $H^1_A$, takes the form 
\begin{equation}\label{e5.42}
\mathcal{D}^2\mathfrak{o}_A: (x_1, x_2, x_3, x_4) \longmapsto 4 \cdot \Pi_A \left( (x_1 idg + x_2 idf) \wedge (x_3 \phi_0 + x_4 i\phi_0) \right)^+. 
\end{equation}
Note that 
\begin{equation}\label{e5.41}
d^+_{A_x} \phi_x = d^+_A \phi_x + x \cdot (dg \wedge i\phi_x)^+ - z(x) \cdot (df \wedge \phi_x)^+=0
\end{equation}
Differentiating (\ref{e5.41}) at $x=0$ gives us 
\[
 d^+_A \dot{\phi}_0  = - (idg \wedge \phi_0)^+ + \dot{r}(0) (df \wedge \phi_0)^+ + \dot{s}(0) (df \wedge i\phi_0)^+,
\]
which implies that 
\begin{equation}\label{e5.61}
\Pi_A (idg \wedge \phi_0)^+ = \dot{r}(0) \Pi_A(df \wedge \phi_0)^+ + \dot{s}(0) \Pi_A  (df \wedge i\phi_0)^+. 
\end{equation}
We claim that $\{\Pi_A(df \wedge \phi_0)^+, \Pi_A(df \wedge i\phi_0)^+\}$ forms a basis of $H^2_A$. It suffices to show that $\Pi_A (idf \wedge \phi_0)^+ \neq 0$. Suppose this fails to hold.  Then (\ref{e5.61}) shows that $\Pi_A (idg \wedge \phi_0)^+ = 0$ as well. Let's write $A = d + a$ with $a = q_1\cdot  idg + q_2 \cdot idf$.Then for any $\phi \in \C = \langle \phi_0, i\phi_0 \rangle$, we have $\Pi_A (a \wedge \phi)^+ = 0$. Then we get a non-zero complex $1$-form $\psi$ so that $d^+_A \psi = (a \wedge \phi)^+$. Equivalently, we have $d^+_{A+ \phi}(a + \psi) = 0$. Since $a + \psi \neq 0 \in H^1_A$ for all $(a, \phi)$, so we conclude that $\ker d^+_{A+\phi} \cap H^1_A \neq 0$ for all $\phi$, which violates the transversality. 

Now with respect to ordered bases $\{idg, idf\}$ of $i\R \oplus i\R \subset H^1_A$ and $\{\Pi_A(df \wedge \phi_0)^+, \Pi_A(df \wedge i\phi_0)^+\}$ of $H^2_A$, (\ref{e5.42}) and (\ref{e5.61}) tell us the map $f_0$ considered in the end of \autoref{ss4.2} takes the form 
\[
f_0 = 
\begin{pmatrix}
\dot{r}(0) & 0 \\
\dot{s}(0) & 1
\end{pmatrix}
\]
up to a positive scale. Note that $\det f_0 = \dot{r}(0)$. Thus the sign of $[A]$ as a bifurcation point is given by the sign of $\dot{r}(0)$. This completes the proof. 
\end{proof}

Now we can complete the proof of \autoref{t1.6}. Since the first homology of $V = M \cup_{T^3} D^2 \times T^2$ is generated by $[\mu]$ and $[\gamma]$, each flat connection $A^{\dagger}_t$ on $M$ extends uniquely to one on $V$ up to $U(1)$-gauge transformations. We write $N = D^2 \times T^2$ and $N_o = (-\infty, 0] \times T^3 \cup D^2 \times T^2$. Let's denoted this new path of connections by $(A^{\dagger, V}_t)$, and its restriction on $D^2 \times T^2$ by $(A^{\dagger, N}_t)$. We note that 
\[
\ker Q^{\mu}_{A^{\dagger, N}_{t, z}}(N_o) = H^1(D^2 \times T^2, A^{\dagger, N}_{t, z}),
\]
where $A^{\dagger, N}_{t, z}$ is the flat connection on $D^2 \times T^2$ whose holonomies around $\mu$ and $\gamma$ are given respectively by $e^{-4\pi i \alpha(\epsilon+t(1-\epsilon))}$ and $2 z$, since under the $0$-surgery gluing map the fundamental group of $D^2 \times T^2$ is generated by $[\mu]$ and $[\gamma]$. For all $(t, e^z) \in [0, 1] \times S^1$, $A^{\dagger, N}_{t, z}$ is not the product connection, thus admits no non-zero twisted harmonic $1$-forms by the Künneth formula. Since $\ind Q_{A^{\dagger, N}_{t, z}}(N_o) = 0$, we conclude that $Q_{A^{\dagger, N}_{t, z}}(N_o) : L^2_1 \to L^2$ is invertible for all $(t, z) \in [0, 1] \times S^1$. Now we can apply the gluing property of the periodic spectral flow \autoref{p4.10} to conclude that 
\[
\widetilde{\Sf}(Q_{A^{\dagger}_t}(M_o, \pi)) = \widetilde{\Sf}(Q_{A^{\dagger, V}_t}(V, \pi)). 
\]
Then \autoref{p4.11} gives us 
\[
\tilde{\rho}_{\varphi_{2\epsilon}}(V)-\tilde{\rho}_{\varphi_{2\alpha}}(V) = 2\widetilde{\Sf}(Q_{A^{\dagger, V}_t}(V, \pi))
\]
Combining $(\ref{e6.3})$, $(\ref{e6.4})$ and \autoref{p6.3}, we conclude that 
\[
\lambda_{FO}(X, \mathcal{T}, \alpha) - 8 \lambda_{FO}(X) = \tilde{\rho}_{\varphi_{2\alpha}}(V) - \tilde{\rho}_{\varphi_{2\epsilon}}(V).
\]
Since $\tilde{\rho}_{\varphi_{2\epsilon}}(V)$ admits no jumps near $0$, \autoref{t1.6} is thus proved.

\subsection{\em Proof of \autoref{t1.7}}\label{ss6.2} \hfill
 
\vspace{3mm}

Now we start in earnest to prove \autoref{t1.7}. Let $V = X_0(\mathcal{T})$ be the $0$-surgered manifold and $Y_0$ a cross-section of $V$ so that the pairing $C'_{\alpha}$ defined in $(\ref{e1.5})$ is non-degenerate and has zero signature. Using the Hodge star, the pairing $C'_{\alpha}$ is equivalent to one defined on $H^0(Y;\C) \oplus H^1(Y;\C)$:
\begin{equation}
C_{\alpha}(x, y):= \langle x, (L_{\varphi_{2\alpha}} - L) y \rangle_{L^2}.
\end{equation}
Such a symmetric pairing shows up in the study of jumps of the eta invariant along a path of self-adjoint elliptic operators, see for instance \cite{KK94, FL96}. 

\begin{lem}\label{l6.4}
Suppose $C_{\alpha}$ is non-degenerate. Then for all but finitely many holonomy parameters $\beta \in (0, 1/2)$, the corresponding representation $\varphi_{2\beta}: \pi_1(V) \to U(1)$ is admissible.
\end{lem}

\begin{proof}
We first prove there are only finitely many $\beta$ violating condition (a) in \autoref{d3.6}. Let $g, f:V \to S^1$ be smooth circle-valued functions whose differentials are dual to $\mu$ and $\gamma$ respectively. Then $\{idg, idf\}$ forms an ordered basis of  $H^1(V;i\R)$ as before. We can identify $H^1(V;i\R)$ with $\C$ by imposing an complex structure $J$ satisfying $J \cdot idf = idg$. Then $A^{\dagger}_{(x, y)}:= d + x \cdot idg - y \cdot idf$ is a holomorphic family of connections on the trivial line bundle of $V$. Since $(V, 1_V)$ is non-degenerate, Lemma 4.5 in \cite{T87} implies that only at a discrete subset of $\C$ can the elliptic complex $E_{A^{\dagger}(x, y)}(V)$ admits non-trivial homology. Since $E_{\varphi_{2\beta}, z}=E_{(2\beta, iz)}$ when $\Rea z=0$, we see that it has non-trivial homology only at finitely many values $\beta \in (0, 1/2)$. 

Next we show that condition (b) in \autoref{d3.6} is satisfied by most holonomy parameters. We follow the argument in \cite[Theorem 5.1]{KK94}. As $\beta$ varies in $\R$, we get an analytic path of self-adjoint elliptic operators $L_{\varphi_{2\beta}}$. Then it follows from \cite[Chapter 7]{K66} that the orthonormal eigenvectors and eigenvalues of this path are analytic as well. We note that  $\ker L_{\varphi_{2\beta}} = H^0(Y; \C_{\varphi_{2\beta}}) \oplus H^1(Y; \C_{\varphi_{2\beta}})$. So it suffices to show that $\ker L_{\varphi_{2\beta}} = 0$ for some $\beta$, in which case the set of $\beta$'s with non-trivial $\ker L_{\varphi_{2\beta}}$ has to be discrete. 

Now we go back to $\alpha \in (0, 1/2)$ with $C_{\alpha}$ non-degenerate. Note that the differential of the path $(L_{\varphi_{2t\alpha}})_{t \in (-\epsilon, \epsilon)}$ at $t=0$ is given by $L_{\varphi_{2\alpha}} - L$. Let $\lambda(t)$ be a path of eigenvalues of $L_{\varphi_{2t\alpha}}$ satisfying $\lambda(0) = 0$, and $x(t)$ a path of $\lambda(t)$-eigenvectors. Then we know $x(0) \in \ker L$ and 
\[
L \dot{x}(0) + (L_{\varphi_{2\alpha}} - L)x(0) = \evalat*{{d \over dt}}{t=0} L_{\varphi_{2t\alpha}} x(t) = \evalat*{{d \over dt}}{t=0} \lambda(t) x(t) = \dot{\lambda}(0) x(0). 
\]
In particular we have 
\[
\langle (\dot{\lambda}(0) x(0), y \rangle  = \langle (L_{\varphi_{2\alpha}} - L)x(0), y \rangle. 
\]
As we vary the path $\lambda(t)$, the eigenvectors $x(0)$'s at $t=0$ form a basis of $\ker L$. The non-degeneracy of $C_{\alpha}$ now implies that $\dot{\lambda}(0) \neq 0$ for all paths. In particular, we know the kernel of $L_{\varphi_{2t\alpha}}$ is nonvanishing when $t>0$ is small enough. 
\end{proof}

Due to \autoref{l6.4}, we may assume the holonomy parameter $\alpha \in (0, 1/2)$ is chosen so that $\varphi_{2\alpha}: \pi_1(V) \to U(1)$ is admissible. It follows from \autoref{p3.12} that 
\begin{equation}
\tilde{\rho}_{\varphi_{2\alpha}}(V) - \tilde{\rho}_{\varphi_{2\epsilon}}(V) = \rho_{\varphi_{2\alpha}}(Y) - \rho_{\varphi_{2\epsilon}}(Y) =  \eta_{\varphi_{2\epsilon}}(Y) - \eta(Y)
\end{equation}
when $\epsilon > 0$ is sufficiently small. In the proof of \autoref{l6.4}, we see that the eigenvalues of $C_{\alpha}$ are precisely the derivatives $\dot{\lambda}(0)$ of the eigenvalues of the paht $L_{\varphi_{2t\alpha}}$ at $t=0$. The non-degeneracy of $C_{\alpha}$ implies that the spectral flow of $L_{\varphi_{2t\alpha}}$ near $t=0$ is given by the signature of $C_{\alpha}$, which we assume to be zero. Thus there are no jumps of the eta invariant near $0$. This shows that $\eta_{\varphi_{2\epsilon}}(Y) = \eta(Y)$ when $\epsilon$ is small enough, hence completes the proof of \autoref{t1.7}. 

As mentioned in the introduction, the assumption on the pairing $C_{\alpha}$ is satisfied by a large class of pairs $(X, \mathcal{T})$. So we provide some examples here. We choose $Y$ to be a cross-section of $X$ intersecting $\mathcal{T}$ into a knot $\mathcal{K}$, which can always get achieved by the argument of \cite[Lemma 2.3]{M1}. Then the cross-section $Y_0$ of $X_0(\mathcal{T})$ is obtained as the surgery of $Y$ along $\mathcal{K}$ so that the meridian $\mu$ of the knot $\mathcal{K}$ generates a $\Z$-summand in $H_1(Y_0;\Z)$. Let $dg \in H^1(Y_0; \C)$ be the class arising as the differential of a harmonic function $g: Y_0 \to \R$ which is dual to $[\mu]$. After possibly rescaling, the difference $L_{\varphi_{2\alpha}} - L$ takes the form 
\[
L_{\varphi_{2\alpha}} - L = 2\alpha
\begin{pmatrix}
0 & -\star dg \star \\
-dg & \star dg
\end{pmatrix}. 
\]
Regarding $1$ as the generator of $H^0(Y_0; \C)$, the restriction of $C_{\alpha}$ to $\Span\{1, dg\}$ is represented by the hyperbolic matrix 
\[
\begin{pmatrix}
0 & -2\alpha \\
-2\alpha & 0
\end{pmatrix},
\]
which is non-degenerate and of zero signature when $\alpha \neq 0$. One can complete the basis $\{1, dg\}$ to a basis of $H^0(Y_0;\C) \oplus H^1(Y_0; \C)$. If the real cohomology of $Y_0$ takes the form $H^*(S^1 \times \Sigma; \R)$, where $\Sigma$ is a closed surface, and $dg$ generates the summand corresponding to $S^1$, then $C_{\alpha}$ satisfies the assumption we have imposed. This has covered all the examples considered by Echeverria \cite{E19} and Ruberman \cite{R20}. In particular, when $b_1(Y_0) = 1$, the assumption always holds.

The assumption on $C_{\alpha}$ is nevertheless not sharp, although we are not aware of any examples violating it. A sharp assumption can be interpreted from the work of Farber-Levine \cite{FL96}, where they have given a complete description on the jumps of eta invariants. However, one cannot expect the representation $\varphi_{2\alpha}$ to be admissible for a generic $\alpha$ in general. In this case, the index computation \autoref{p3.11} no longer works, and the periodic eta invariant cannot be identified with the usual eta invariant as a result. However, the author expects this difference compensates the jumps of the eta invariant near $0$, which would lead to the proof of \autoref{t1.7} in the general case. This requires the study of the jumps for periodic eta invariants that are not generic as the case considered in \autoref{s4}, and a complete description (which only depends on homology in the case of ASD DeRham complexes) should generalize the result of Farber-Levine \cite{FL96} in some sense. Such a result would also generalize the index computation \autoref{p3.11}, and lead to the comparison between $\lim_{\epsilon \to 0^+} \tilde{\rho}_{\varphi_{2\epsilon}}(V)$ and $\tilde{\rho}(V)$. 

\subsection{\em Computations on Mapping Tori}\label{ss6.3} \hfill
 
\vspace{3mm}

In this subsection, we supply an argument for (\ref{e1.7}). For ease of notation, we write $\Sigma$ for the $n$-fold branched cover $\Sigma_n(Y, \mathcal{K})$, $\tau$ for the covering transformation $\tau_n$, and $X$ for the mapping torus $X_n(Y, \mathcal{K})$. We write $\mathcal{J} \subset \Sigma$ for the preimage of $\mathcal{K}$, which we used to write $\mathcal{K}_n$ in the introduction, and $\mathcal{T}$ for the mapping torus of $\mathcal{J}$ in $X$. Let's write $i: \Sigma \to X$ for the inclusion map on a slice of the mapping torus.

The argument of \cite[Proposition 3.1]{RS04} tells us that the irreducible $\alpha$-representations of $(X, \mathcal{T})$ is in two-to-one correspondence with the $\tau$-invariant irreducible $\alpha$-representations of $(\Sigma, \mathcal{J})$, namely $i^*: \chi^{\alpha, *}(X, \mathcal{T}) \to \chi^{\alpha, *}_{\tau}(\Sigma, \mathcal{J})$ is a double cover. As pointed out in \cite[Section 8]{E19}, the Zariski tangent spaces and orientations of $\M^{\alpha, *}_{\sigma}(X, \mathcal{T})$ can be identified with those of $\M_{\sigma, \tau}^{\alpha, *}(\Sigma, \mathcal{J})$ with respect to generic perturbations, where $\M^{\alpha, *}_{\sigma, \tau}(\Sigma, \mathcal{J})$ means the space of irreducible perturbed $\tau$-invariant flat $\alpha$-connections on $(\Sigma, \mathcal{J})$. Thus it suffices to count $\M^{\alpha, *}_{\sigma, \tau}(\Sigma, \mathcal{J})$. Since the holonomy perturbations preserve the $\tau$-invariance, and away from the support of a perturbation each connection in $\M^{\alpha, *}_{\sigma, \tau}(\Sigma, \mathcal{J})$ is flat, we conclude that each connection in $\M^{\alpha, *}_{\sigma, \tau}(\Sigma, \mathcal{J})$ is pulled-back from a perturbed flat connection in $(Y, \mathcal{K})$ with respect to certain holonomy parameter $\alpha'$. 

In \cite[Lemma 48]{E19}, Echeverria proved that the pull-back map $p^*: \M^{\alpha', *}_{\sigma}(Y, \mathcal{K}) \to \M^{\alpha, *}_{\sigma, \tau}(\Sigma, \mathcal{J})$ is either injective or empty given a holonomy parameter $\alpha' \in (0, 1/2)$. So the key point is to determine the holonomy parameters $\alpha'$ with respect to which the pull-back map is non-empty. Let $[A] = p^*[A'] \in \M^{\alpha, *}_{\sigma, \tau}(\Sigma, \mathcal{J})$. The the holonomy of $A$ around the meridian $\mu_{\mathcal{J}}$ is $e^{-2\pi i \alpha}$. Thus one can choose representative $A'$ of $[A']$ such that the holonomy of $A'$ around the meridian $\mu_{\mathcal{K}}$ satisfies
\begin{equation}\label{e6.11}
\left(\Hol_{\mu_{\mathcal{K}}} A' \right)^n = \Hol_{\mu_{\mathcal{J}}} A = e^{-2\pi i \alpha} \Longrightarrow \Hol_{\mu_{\mathcal{K}}} A' = e^{-2\pi i (\alpha+ j)/n},
\end{equation}
for some $j = 0,..., n-1$. Conversely any perturbed flat connection $A'$ on $Y \backslash K$ satisfies (\ref{e6.11}) pulls back to one on $\Sigma \backslash \mathcal{J}$. Due to the normalization that the holonomy parameter lies in $(0, 1/2)$, we conclude that there are $n$ such choices parametrized by $j=0, ..., n-1$:
\begin{equation}\label{e6.12}
\alpha'_j = 
\begin{cases}
(\alpha+j)/n & \text{ if $(\alpha+j)/n \in (0, 1/2)$} \\
1- (\alpha+j)/n & \text{ if $(\alpha+j)/n \in (1/2, 1)$}
\end{cases}
\end{equation}
Since $\alpha \in (0, 1/2)$, it's direct to check that all values of $\alpha'_j$ in (\ref{e6.12}) are distinct. Invoking the result of Herald \cite[Theorem 0.1]{H97}, when $e^{-4\pi i \alpha'_j}$ is not a root of the Alexander polynomial $\Delta_{(Y, \mathcal{K})}$ for each $j=0 , ..., n-1$, we get 
\[
\# \M^{\alpha'_j, *}_{\sigma}(Y, \mathcal{K}) = 4\lambda(Y) + \frac{1}{2} \sigma_{2\alpha'_j}(Y, \mathcal{K}) = 4\lambda(Y) + \frac{1}{2} \sigma_{2(\alpha+j)/n}(Y, \mathcal{K}),
\] 
where the second equality has used the symmetry of the Levine-Tristram invariant \cite[Proposition 2.3]{L92} in the case when $(\alpha+j)/n \in (1/2, 1)$. Thus we arrive at the claimed computation
\begin{equation}
\lambda_{FO}(X, \mathcal{T}, \alpha) = 2\sum_{j=0}^{n-1} \# \M^{\alpha'_j, *}_{\sigma}(Y, \mathcal{K}) = 8n\lambda(Y) + \sum_{j=0}^{n-1} \sigma_{2(\alpha+j)/n}(Y, \mathcal{K}). 
\end{equation}

\bibliographystyle{alpha}
\bibliography{RefTS}
\end{document}